\documentclass[a4paper]{amsart}
\pdfoutput=1
\usepackage{physics}
\usepackage{amsmath}
\usepackage{tikz}
\usepackage{mathdots}
\usepackage{yhmath}
\usepackage{cancel}
\usepackage{color}
\usepackage{siunitx}
\usepackage{array}
\usepackage{graphicx}
\usepackage{multirow}
\usepackage{amssymb}
\usepackage{gensymb}
\usepackage{enumitem}
\usepackage{tabularx}
\usepackage{extarrows}
\usepackage{booktabs}
\usepackage{pgfplots} 
\usetikzlibrary{fadings}
\usetikzlibrary{patterns}
\usetikzlibrary{shadows.blur}
\usetikzlibrary{shapes}
\usepackage{amsthm}
\usepackage{amsfonts}
\usepackage{amsxtra}
\usepackage{float}
\usepackage{pdflscape}
\usepackage{slashed}
\usepackage{colortbl}
\usepackage{cite}
\usepackage{graphicx}
\usepackage{textcomp}
\usepackage{subcaption}
\usepackage{amsbsy} 
\usepackage{mathrsfs} 
\usepackage{verbatim}
\usepackage{stmaryrd}
\usepackage{makecell}
\usepackage[all,2cell]{xy}
\usepackage{etex}
\usepackage{mathtools}
\usetikzlibrary{arrows.meta}
\usepackage{tikz-cd}
\usepackage{adjustbox}
\usepackage{quiver}
\usepackage[bookmarksnumbered,colorlinks,plainpages]{hyperref}
\usetikzlibrary{positioning,arrows}
\usetikzlibrary{decorations.pathreplacing}
\newtheorem*{Thm-A}{Theorem A}
\newtheorem*{Thm-B}{Theorem B}

\newtheorem{thm}{Theorem}[section]

\newtheorem{cor}[thm]{Corollary}
\newtheorem{lem}[thm]{Lemma}
\newtheorem{obs}[thm]{Observation}

\newtheorem{prop}[thm]{Proposition}

\theoremstyle{definition}
\newtheorem{defn}[thm]{Definition}
\newtheorem{rem}[thm]{\bf Remark}

\numberwithin{equation}{section}

\newcommand{\vx}{{\vec{x}}}

\def\Hom{{\rm Hom}}
\def\Ext{{\rm Ext}}

\def\bbL{{\mathbb L}}

\begin{document}
	\title[Geometric Model for Vector Bundles]{Geometric Model for Vector Bundles via Infinite Marked Strips}

	\author[J. Chen, S. Ruan and J. Zhang] {Jianmin Chen, Shiquan Ruan and Jinfeng Zhang$^*$}
	
	\thanks{$^*$ the corresponding author}
	\footnote[0]{\textcolor{black}{2020} \textit{Mathematics Subject Classification}.  14F06, 18E10, 05E10, 16S99, 57M50.}
	
	\dedicatory{}%
	\commby{}%
	
	\begin{abstract}
We present a geometric model for the category of vector bundles over the weighted projective line of type $(2,2,n)$. This model is based on the orbit space of an infinite marked strip under a specific group action. We establish a bijection between indecomposable bundles and orbits of line segments on the strip, which yields geometric interpretations for various aspects, including the Picard group action, vector bundle duality, dimension of $\Ext^1$, projective cover and injective hull of extension bundle, etc.
	\end{abstract}
	
	\maketitle
	\section{Introduction}
\subsection{Background} 
Recently, several works have provided geometric-combinatorial descriptions for various categories. For instance, Caldero, Chapoton, and Schiffler \cite{MR2187656} realized the cluster category of type $A_n$ using a regular $(n+3)$-gon. This work was later extended to the Dynkin type $D_n$ using a regular $n$-gon with a puncture by Schiffler \cite{MR2366159}. Subsequently, a series of studies have focused on cluster categories related to triangulated surfaces. For instance, Br\"ustle and Zhang \cite{MR2870100} studied the unpunctured case,  Qiu and Zhou \cite{MR3705277} investigated the punctured case, and Amiot and Plamondon\cite{MR4289034} explored the punctured case via group actions. Concurrently, progress has been made in developing geometric models for abelian categories, such as tube categories \cite{MR2921637} and module categories over gentle algebras \cite{MR4294120} and skew-gentle algebras \cite{skew-gentlealgebra}. In these studies, curves on surfaces correspond to indecomposable objects in the categories. Applications of these geometric models include interpreting the dimension of $\operatorname{Ext}^1$ as intersection numbers of curves, corresponding triangulations of surfaces to tilting objects, and so on. For more related work, see \cite{ MR3081629, opper2018geometric, MR3826729, chang2022geometric, ptolemyAn, ptolemyDn, MR3183483}.

Weighted projective lines and their coherent sheaves categories were introduced by Geigle and Lenzing in \cite{MR915180}, in order to give a geometric realization of canonical algebras in the sense of Ringel \cite{MR774589}. The study of weighted projective lines has been closely related to many branches of mathematics, such as Lie theory \cite{MR2734343,MR4165467,MR2970461,MR2031167}, singularity theory \cite{MR2640048, Hbner1996ExzeptionelleVU,hubner1989classification,MR1308987,MR1491919,MR2890586} and homological mirror symmetry \cite{MR1987870,MR0568900}.  In \cite{chen2023geometric}, the authors presented a geometric model for the category of coherent sheaves over a weighted projective line of type $(p, q)$ using a marked annulus. 

In this paper, we consider the weighted projective line $\mathbb{X}$ of type $(2, 2, n)$. As is known that each indecomposable coherent sheaf over $\mathbb{X}$ is either a vector bundle or a torsion sheaf. The torsion subcategory of ${\rm coh}\mbox{-}\mathbb{X}$ consists of a collection of mutually orthogonal tube categories. Baur and Marsh \cite{MR2921637} have provided the geometric model for tube categories.  We focus on the category ${\rm vect}\mbox{-}\mathbb{X}$ of vector bundles over $\mathbb{X}$.

Note that the category ${\rm vect}\mbox{-}\mathbb{X}$ differs from module categories in various aspects:
\begin{itemize}
    \item [-] ${\rm vect}\mbox{-}\mathbb{X}$ has Picard group action and admits vector bundle duality.
    \item [-] Any object in ${\rm vect}\mbox{-}\mathbb{X}$ has an infinite chains of subobjects.
    \item [-] ${\rm vect}\mbox{-}\mathbb{X}$ has various exact structures, which related to Kleinian/Fuchsian singularities.
    \item [-] As a full subcategory of ${\rm coh}\mbox{-}\mathbb{X}$, ${\rm vect}\mbox{-}\mathbb{X}$ has no projective or injective objects.
    \item [-] ${\rm vect}\mbox{-}\mathbb{X}$ carries a distinguished Frobenius exact structure induced from the category of graded Cohen–Macaulay modules over its coordinate ring, which ensures that ${\rm vect}\mbox{-}\mathbb{X}$ has enough projective-injective objects.
\end{itemize}

These motivate us to seek a geometric model to reflects the intrinsic properties of ${\rm vect}\mbox{-}\mathbb{X}$. Let $\widetilde{\mathcal{S}}$ be an infinite strip with countably infinite marked points on both the upper and lower boundaries, and let $G$ be the group generated by two bijections, $\sigma_n$ and $\theta$, on $\widetilde{\mathcal{S}}$.  The bijection $\sigma_n$ translates points, while $\theta$ reflects all points with respect to a certain point (see \S \ref{A marked infinite strip}). We present a geometric model for the category ${\rm vect}\mbox{-}\mathbb{X}$ in terms of the orbit space $\mathcal{S}$ of $\widetilde{\mathcal{S}}$ respect to $G$. 

\subsection{Main results}
The Picard group $\mathbb{L}$ on the weighted projective line $\mathbb{X}$ of type $(2,2,n)$ is  generated by $\vec{x}_1, \vec{x}_2, \vec{x}_3$ subject to the relations $2 \vec{x}_1=2 \vec{x}_2=n \vec{x}_3$. The generalized extension bundles are defined by the unique middle term (up to isomorphism) of the non-split exact sequence
\[
\begin{tikzcd}[ampersand replacement=\&,cramped,sep=small]
    0 \& {L(\vec{\omega})} \& \mathsf{E}_{L}\langle \vec{x} \rangle \& {L(\vec{x})} \& 0
    \arrow[from=1-2, to=1-3]
    \arrow[from=1-3, to=1-4]
    \arrow[from=1-1, to=1-2]
    \arrow[from=1-4, to=1-5]
\end{tikzcd}
\]
for some line bundle \(L\) and \(\vec{x}\) in \(\mathbb{L}\) with \(0 \leq \vec{x} \leq (n-1)\vec{x}_3\). Let $F$ be the degree shift automorphism by $\vec{x}_1-\vec{x}_2\in \mathbb{L}$ on ${\rm vect}\mbox{-}\mathbb{X}$. We denote by ${\rm vect}^F\mbox{-}\mathbb{X}$ the full subcategory of ${\rm vect}\mbox{-}\mathbb{X}$ consisting of $F$-stable objects. For a category $\mathcal{C}$, we denote by $\operatorname{ind}\mathcal{C}$ the set of isomorphism classes of indecomposable objects in $\mathcal{C}$. 
 
 Let $\widetilde{\operatorname{Seg}(M)}$ be the set of $G$-orbits of line segments connecting marked points on the upper and lower boundaries of $\widetilde{\mathcal{S}}$. We represent the elements in \(\widetilde{\operatorname{Seg}(M)}\) as \(\widetilde{[i,j]}\), where \([i,j]\) denotes a line segment whose endpoints are indicated by indices \(i\) and \(j\). Depending on whether the line segment passes through a specific point, we refine \(\widetilde{\operatorname{Seg}(M)}\) to construct a new set \(\widetilde{\operatorname{Seg}^*(M)}\) of $G$-orbits of line segments (see \S \ref{Section 4.3}).
 Let $\mathcal{MG}(\mathcal{S})$ be the mapping class group of $\mathcal{S}$. We fix an arbitrary line bundle $L_0$ in ${\rm vect}\mbox{-}\mathbb{X}$. The main results of the paper are summarized as follows.

\begin{Thm-A}\label{thmA}(Propostions \ref{correspondence1}, \ref{group iso}, \ref{MCG and L action}, and \ref{correspondence}) Keep notations as above. There exists a group isomorphism
\[
\psi: \mathcal{MG}(\mathcal{S}) \xrightarrow[]{\cong} \mathbb{L}/\mathbb{Z}(\vec{x}_1 - \vec{x}_2)
\]
and a bijection
  \begin{align*}
\phi: \widetilde{\operatorname{Seg}(M)} & \rightarrow \operatorname{ind}({\rm vect}^F\mbox{-}\mathbb{X}), \\
\widetilde{[i,j]} & \mapsto \mathsf{E}_{L_0(-i\vec{x}_3)}\langle (i+j-1)\vec{x}_3 \rangle
\end{align*}
such that the following diagram  is commutative
\begin{figure}[H]
    \centering
\begin{tikzcd}
{\mathcal{MG} (\mathcal{S})} \arrow[r, "\psi"]                                     & \mathbb{L}/\mathbb{Z}(\vec{x}_1-\vec{x}_2)                                                    \\
\widetilde{\operatorname{Seg}(M)} \arrow[r, "\phi"] \arrow[loop, distance=3em, in=125, out=55] & \operatorname{ind}({\rm vect}^F\mbox{-}\mathbb{X}) \arrow[loop, distance=3em, in=125, out=55]
\end{tikzcd}
\end{figure}
\noindent in the sense that
\[
\phi(x\cdot \widetilde{[i,j]}) = \psi(x)\cdot \phi(\widetilde{[i,j]}) \quad \text{for all } x \in \mathcal{MG} (\mathcal{S}) \text{ and } \widetilde{[i,j]} \in \widetilde{\operatorname{Seg}(M)}.
\]
Moreover, $\phi$ induces a bijection $\widehat{\phi}:\widetilde{\operatorname{Seg}^*(M)} \rightarrow \operatorname{ind}({\rm vect}\mbox{-}\mathbb{X})$, which is explicitly described by the following table:
		\begin{center}
\begin{tabular}{|c|c|}
\hline
\text{$G$-orbits in $\widetilde{\operatorname{Seg}^*(M)}$} & \text{Indecomposable objects in ${\rm vect}\mbox{-}\mathbb{X} $} \\
\hline
$\widetilde{[i,-i]^{+}}$ & $F(L_0(-(i+1)\vec{x}_3))$ \\
$\widetilde{[i,-i]^{-}}$ & $L_0(-(i+1)\vec{x}_3)$ \\
$\widetilde{[i,n-i]^{+}}$ & $L_0(\vec{x}_1-(i+1)\vec{x}_3)$ \\
$\widetilde{[i,n-i]^{-}}$ & $F(L_0(\vec{x}_1-(i+1)\vec{x}_3))$ \\
$\widetilde{[i,k-i]}$ & $\mathsf{E}_{L_0(-i\vec{x}_3)}\langle (k-1)\vec{x}_3 \rangle$ \\
\hline
\end{tabular}
\end{center}
   where $i,k$ are integers with $1\leq k\leq n-1$. 
\end{Thm-A}
 As applications, we give geometric interpretations for the slope of indecomposable bundles, the Picard group action and vector bundle duality (see Propositions \ref{slop}, \ref{vector bundle duality} and Corollary \ref{L-action}). Moreover, we interpret the dimensions of extension spaces $\Ext^1$ between two indecomposable bundles as intersection indices of their correspondence $G$-orbits in $\widetilde{\operatorname{Seg}^*(M)}$ (see Theorem \ref{dimension and positive intersection}).
 
As we know that ${\rm vect}\mbox{-}\mathbb{X}$ carries a distinguished Frobenius exact structure, with the system of all line bundles as the indecomposable projective-injective objects. We realize the universal properties of projective covers and injective envelopes of extension bundles as the proximity between line segments. 
\begin{Thm-B}(Proposition \ref{cover and hull})
Assume that $X=\widehat{\phi}(\widetilde{[i,j]})$ is an extension bundle in ${\rm vect}\mbox{-}\mathbb{X} $ for some $\widetilde{[i,j]} \in \widetilde{\operatorname{Seg}^*(M)}$. Then its projective cover $\mathfrak{P}(X)$ and its injective hull $\mathfrak{I}(X)$ are given by
  \[\mathfrak{P}(X)=\widehat{\phi}(\widetilde{[i,\leftarrow]})\oplus \widehat{\phi}(\widetilde{[\rightarrow,j]}) \text{ and } 
 \mathfrak{I}(X)=\widehat{\phi}(\widetilde{[i,\rightarrow]})\oplus \widehat{\phi}(\widetilde{[\leftarrow,j]}),\]
which can be illustrated as below:
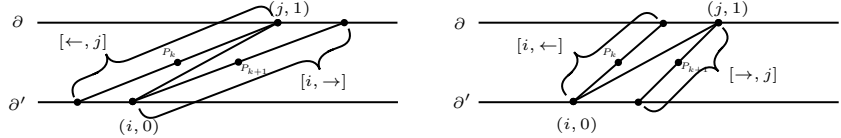
\begin{figure}[H]
    \tikzset{every picture/.style={line width=0.75pt}}          
\begin{tikzpicture}[x=0.75pt,y=0.75pt,yscale=-1,xscale=1]
\draw    (73.33,109.33) -- (253.08,109.33) ;
\draw    (73.58,149.33) -- (253.33,149.33) ;  
\draw [color={rgb, 255:red, 0; green, 0; blue, 0 }  ,draw opacity=1 ]   (120.88,148.97) -- (193.33,109.33) ;
\draw [shift={(193.33,109.33)}, rotate = 331.32] [color={rgb, 255:red, 0; green, 0; blue, 0 }  ,draw opacity=1 ][fill={rgb, 255:red, 0; green, 0; blue, 0 }  ,fill opacity=1 ][line width=0.75]      (0, 0) circle [x radius= 1.34, y radius= 1.34]   ;
\draw [shift={(120.88,148.97)}, rotate = 331.32] [color={rgb, 255:red, 0; green, 0; blue, 0 }  ,draw opacity=1 ][fill={rgb, 255:red, 0; green, 0; blue, 0 }  ,fill opacity=1 ][line width=0.75]      (0, 0) circle [x radius= 1.34, y radius= 1.34]   ;
\draw    (193.33,109.33) -- (93.33,149.33) ;
\draw [shift={(93.33,149.33)}, rotate = 158.2] [color={rgb, 255:red, 0; green, 0; blue, 0 }  ][fill={rgb, 255:red, 0; green, 0; blue, 0 }  ][line width=0.75]      (0, 0) circle [x radius= 1.34, y radius= 1.34]   ;
\draw [shift={(143.33,129.33)}, rotate = 158.2] [color={rgb, 255:red, 0; green, 0; blue, 0 }  ][fill={rgb, 255:red, 0; green, 0; blue, 0 }  ][line width=0.75]      (0, 0) circle [x radius= 1.34, y radius= 1.34]   ;
\draw [shift={(193.33,109.33)}, rotate = 158.2] [color={rgb, 255:red, 0; green, 0; blue, 0 }  ][fill={rgb, 255:red, 0; green, 0; blue, 0 }  ][line width=0.75]      (0, 0) circle [x radius= 1.34, y radius= 1.34]   ; 
\draw    (226.54,109.31) -- (120.88,148.97) ;
\draw [shift={(120.88,148.97)}, rotate = 159.42] [color={rgb, 255:red, 0; green, 0; blue, 0 }  ][fill={rgb, 255:red, 0; green, 0; blue, 0 }  ][line width=0.75]      (0, 0) circle [x radius= 1.34, y radius= 1.34]   ;
\draw [shift={(173.71,129.14)}, rotate = 159.42] [color={rgb, 255:red, 0; green, 0; blue, 0 }  ][fill={rgb, 255:red, 0; green, 0; blue, 0 }  ][line width=0.75]      (0, 0) circle [x radius= 1.34, y radius= 1.34]   ;
\draw [shift={(226.54,109.31)}, rotate = 159.42] [color={rgb, 255:red, 0; green, 0; blue, 0 }  ][fill={rgb, 255:red, 0; green, 0; blue, 0 }  ][line width=0.75]      (0, 0) circle [x radius= 1.34, y radius= 1.34]   ;
\draw    (293.58,109.33) -- (473.33,109.33) ;
\draw    (293.58,149.33) -- (473.33,149.33) ;
\draw [color={rgb, 255:red, 0; green, 0; blue, 0 }  ,draw opacity=1 ]   (373.33,149.33) -- (413.33,109.33) ;
\draw [shift={(413.33,109.33)}, rotate = 315] [color={rgb, 255:red, 0; green, 0; blue, 0 }  ,draw opacity=1 ][fill={rgb, 255:red, 0; green, 0; blue, 0 }  ,fill opacity=1 ][line width=0.75]      (0, 0) circle [x radius= 1.34, y radius= 1.34]   ;
\draw [shift={(393.33,129.33)}, rotate = 315] [color={rgb, 255:red, 0; green, 0; blue, 0 }  ,draw opacity=1 ][fill={rgb, 255:red, 0; green, 0; blue, 0 }  ,fill opacity=1 ][line width=0.75]      (0, 0) circle [x radius= 1.34, y radius= 1.34]   ;
\draw [shift={(373.33,149.33)}, rotate = 315] [color={rgb, 255:red, 0; green, 0; blue, 0 }  ,draw opacity=1 ][fill={rgb, 255:red, 0; green, 0; blue, 0 }  ,fill opacity=1 ][line width=0.75]      (0, 0) circle [x radius= 1.34, y radius= 1.34]   ;
\draw    (340.88,148.97) -- (385.79,109.69) ;
\draw [shift={(385.79,109.69)}, rotate = 318.83] [color={rgb, 255:red, 0; green, 0; blue, 0 }  ][fill={rgb, 255:red, 0; green, 0; blue, 0 }  ][line width=0.75]      (0, 0) circle [x radius= 1.34, y radius= 1.34]   ;
\draw [shift={(363.33,129.33)}, rotate = 318.83] [color={rgb, 255:red, 0; green, 0; blue, 0 }  ][fill={rgb, 255:red, 0; green, 0; blue, 0 }  ][line width=0.75]      (0, 0) circle [x radius= 1.34, y radius= 1.34]   ;
\draw [shift={(340.88,148.97)}, rotate = 318.83] [color={rgb, 255:red, 0; green, 0; blue, 0 }  ][fill={rgb, 255:red, 0; green, 0; blue, 0 }  ][line width=0.75]      (0, 0) circle [x radius= 1.34, y radius= 1.34]   ;
\draw [color={rgb, 255:red, 0; green, 0; blue, 0 }  ,draw opacity=1 ]   (340.88,148.97) -- (413.33,109.33) ;
\draw [shift={(413.33,109.33)}, rotate = 331.32] [color={rgb, 255:red, 0; green, 0; blue, 0 }  ,draw opacity=1 ][fill={rgb, 255:red, 0; green, 0; blue, 0 }  ,fill opacity=1 ][line width=0.75]      (0, 0) circle [x radius= 1.34, y radius= 1.34]   ;
\draw [shift={(340.88,148.97)}, rotate = 331.32] [color={rgb, 255:red, 0; green, 0; blue, 0 }  ,draw opacity=1 ][fill={rgb, 255:red, 0; green, 0; blue, 0 }  ,fill opacity=1 ][line width=0.75]      (0, 0) circle [x radius= 1.34, y radius= 1.34]   ;
\draw   (190.67,105.83) .. controls (188.88,101.52) and (185.83,100.26) .. (181.52,102.04) -- (119.53,127.72) .. controls (113.37,130.27) and (109.4,129.39) .. (107.61,125.08) .. controls (109.4,129.39) and (107.21,132.83) .. (101.05,135.38)(103.83,134.23) -- (95.46,137.69) .. controls (91.15,139.48) and (89.88,142.53) .. (91.67,146.84) ;
\draw   (124.17,152.5) .. controls (125.84,156.85) and (128.86,158.19) .. (133.21,156.52) -- (202.44,129.9) .. controls (208.67,127.51) and (212.62,128.5) .. (214.29,132.85) .. controls (212.62,128.5) and (214.89,125.12) .. (221.11,122.73)(218.31,123.8) -- (224.64,121.37) .. controls (229,119.7) and (230.34,116.68) .. (228.66,112.32) ;
\draw   (382.83,107.33) .. controls (379.82,103.77) and (376.53,103.5) .. (372.97,106.51) -- (353.42,123.05) .. controls (348.33,127.36) and (344.28,127.73) .. (341.27,124.17) .. controls (344.28,127.73) and (343.25,131.66) .. (338.16,135.97)(340.45,134.03) -- (338.16,135.97) .. controls (334.59,138.98) and (334.32,142.27) .. (337.34,145.84) ;
\draw   (376.17,151.5) .. controls (379.49,154.78) and (382.79,154.76) .. (386.07,151.44) -- (401.34,135.97) .. controls (406.03,131.22) and (410.03,130.49) .. (413.35,133.77) .. controls (410.03,130.49) and (410.71,126.48) .. (415.4,121.73)(413.29,123.87) -- (415.4,121.73) .. controls (418.68,118.41) and (418.66,115.11) .. (415.34,111.84) ;

\draw (57.33,144.93) node [anchor=north west][inner sep=0.75pt]  [font=\tiny]  {$\partial ^{\prime }$}; 
\draw (59.13,104.6) node [anchor=north west][inner sep=0.75pt]  [font=\tiny]  {$\partial $};
\draw (112.17,156.07) node [anchor=north west][inner sep=0.75pt]  [font=\tiny]  {$( i,0)$}; 
\draw (186.8,97.2) node [anchor=north west][inner sep=0.75pt]  [font=\tiny]  {$( j,1)$}; 
\draw (279.2,143.73) node [anchor=north west][inner sep=0.75pt]  [font=\tiny]  {$\partial ^{\prime }$}; 
\draw (279.8,105.4) node [anchor=north west][inner sep=0.75pt]  [font=\tiny]  {$\partial $}; 
\draw (330.67,155.07) node [anchor=north west][inner sep=0.75pt]  [font=\tiny]  {$( i,0)$}; 
\draw (405.87,96.8) node [anchor=north west][inner sep=0.75pt]  [font=\tiny]  {$( j,1)$}; 
\draw (134,120.6) node [anchor=north west][inner sep=0.75pt]  [font=\tiny]  {\scalebox{0.6}[0.6]{$P_{k}$}}; 
\draw (310.87,115.33) node [anchor=north west][inner sep=0.75pt]  [font=\tiny]  {$\left[ i,\leftarrow \right]$}; 
\draw (202.77,133.67) node [anchor=north west][inner sep=0.75pt]  [font=\tiny]  {$\left[ i,\rightarrow \right]$}; 
\draw (416.9,129.97) node [anchor=north west][inner sep=0.75pt]  [font=\tiny]  {$\left[\rightarrow ,j\right]$}; 
\draw (82.37,113.4) node [anchor=north west][inner sep=0.75pt]  [font=\tiny]  {$\left[\leftarrow ,j\right]$}; 
\draw (173,129) node [anchor=north west][inner sep=0.75pt]  [font=\tiny]  {\scalebox{0.6}[0.6]{$P_{k+1}$}}; 
\draw (354,120.6) node [anchor=north west][inner sep=0.75pt]  [font=\tiny]  {\scalebox{0.6}[0.6]{$P_{k}$}}; 
\draw (393,128) node [anchor=north west][inner sep=0.75pt]  [font=\tiny]  {\scalebox{0.6}[0.6]{$P_{k+1}$}};
\end{tikzpicture}
\end{figure}
\end{Thm-B}
We refer to  Definition \ref{def line} for the relevant notations.
\subsection{Context}The paper is organized as follows.  
In \S \ref{Section 2}, we recall some basic facts on the weighted projective lines of type $(2,2,n)$. 
In \S \ref{Section 3}, we provide the definition and basic properties of generalized extension bundles. 
In \S \ref{Section 4}, we introduce the orbit space of an infinite marked strip under a $G$-action as a geometric model of ${\rm vect}\mbox{-}\mathbb{X}$. We establish a correspondence between $G$-orbits of line segments and indecomposable bundles, and give some applications of this geometric model. In \S \ref{Section 5}, we interpret dimensions of $\Ext^1$ as intersection indices of $G$-orbits of line segments.
In \S \ref{Projective cover and injective hull}, we explore the geometric interpretation for the projective covers and injective hulls of extension bundles. In Appendix \ref{Appendix}, we collect some exact sequences containing generalized extension bundles in ${\rm coh}\mbox{-}\mathbb{X}$.
\subsection{Conventions}
Throughout this paper,  $\mathbb{X}:=\mathbb{X}(2, 2, n)$ denotes the weighted projective line of weight type $(2, 2, n)$, where the integer $n$ is at least two and  
 the base field $\mathbf{k}$ is supposed to be algebraically closed. The cardinality of a set $I$ is denoted by $\mid I\mid$. For a rational number $r$, we denote by $\lfloor r\rfloor$ the largest integer less than or equal to $r$. 
\section{Weighted projective lines of type \textnormal{(2,2,n)}}\label{Section 2}
We recall some basic notions and facts about weighted projective lines from \cite{MR3028577, MR2931898}. We restrict our treatment to the case of type $(2, 2, n)$.
 
 Let $\mathbb{L}:=\mathbb{L}(2, 2, n)$ be the rank one abelian group on three generators $\vec{x}_1, \vec{x}_2, \vec{x}_3$ subject to the relations $2 \vec{x}_1=2 \vec{x}_2=n \vec{x}_3=: \vec{c}$, where $\vec{c}$ is called the \emph{canonical element}. We call $\mathbb{L}$ the \emph{Picard group} on $\mathbb{X}$. We define the \emph{dualizing element} of $\mathbb{L}$ as $\vec{\omega}=\vec{c}-\left(\vec{x}_1+\vec{x}_2+\vec{x}_3\right)$.  The abelian group $\mathbb{L}$ is ordered by defining the positive cone $\{\vec{x} \in \mathbb{L} \mid \vec{x} \geq \vec{0}\}$ to consist of the elements of the form $l_1 \vec{x}_1+l_2 \vec{x}_2+l_3 \vec{x}_3$, where $l_1, l_2, l_3 \geq 0$. Each element $\vec{x} \in \mathbb{L}$ can be written in normal form
\begin{equation}\label{normal form}
    \vec{x}=l_1 \vec{x}_1+l_2 \vec{x}_2+l_3 \vec{x}_3+l \vec{c}
\end{equation}
with unique $l_i, l \in \mathbb{Z}, 0 \leq l_1,l_2<2$ and $0 \leq l_3<n$. 
 
 Denote by $R$ the commutative algebra $R:=\mathbf{k}\left[x_1, x_2, x_3\right] /\left(x_1^{2}+x_2^{2}+x_3^{n}\right)$. The algebra $R$ is $\mathbb{L}$-graded by setting deg $x_i:=\vec{x}_i(i=1,2,3)$, hence $R=\bigoplus_{\vec{x} \in \mathbb{L}} R_{\vec{x}}$. Here, by \cite[Proposition 1.3]{MR915180}, the set $$\{x_1^{l_1+2a}x_2^{l_2+2b}x_3^{l_3}\; |\; a+b=l, a, b\geq 0\}$$ form a $\mathbf{k}$-basis of $R_{\vec{x}}$. Then the homogeneous component $R_{\vec{x}}$ is nonzero  if and only if  $\vec{x} \geq \vec{0}$, and equivalently, $l \geq 0$ in the normal form \eqref{normal form} of $\vec{x}$. 
 
 By an $\mathbb{L}$ graded version of the Serre construction \cite{MR0068874}, the weighted projective line $\mathbb{X}$ of weight type $\left(2,2, n\right)$ is given by its category of coherent sheaves coh-$\mathbb{X}:={\rm mod}^{\bbL}\mbox{-}R/{{\rm mod}_0^{\bbL}\mbox{-}R}$, the quotient category of finitely generated $\mathbb{L}$-graded modules modulo the Serre subcategory of graded modules of finite length. The image $\mathcal{O}$ of $R$ in ${\rm mod}^{\bbL}\mbox{-}R/{{\rm mod}_0^{\bbL}\mbox{-}R}$ serves as the \emph{structure sheaf} of coh-$\mathbb{X}$, and $\mathbb{L}$ acts on coh-$\mathbb{X}$ by degree shift.  

 The category ${\rm coh}\mbox{-}\mathbb{X}$ is a connected, hereditary, abelian, ${\rm Hom}$-finite, $\mathbf{k}$-linear, noetherian category satisfies Serre duality in the form 
 \[D\operatorname{Ext}{ }^1(X, Y)\cong\operatorname{Hom}(Y, X(\vec{\omega})),\]
functorially in $X$ and $Y$, where $D=\Hom_\mathbf{k}(-,\mathbf{k})$. Moreover, Serre duality implies the existence of almost split sequences for coh-$\mathbb{X}$ with the Auslander-Reiten translation $\tau$ given by the shift with $\vec{\omega}$.

 The category ${\rm coh}\mbox{-}\mathbb{X}$ can be expressed as  ${\rm coh}\mbox{-}\mathbb{X}={\rm vect}\mbox{-}\mathbb{X}\vee{\rm coh}_{0}\mbox{-}\mathbb{X}$, where ${\rm vect}\mbox{-}\mathbb{X}$ (resp. ${\rm coh}_{0}\mbox{-}\mathbb{X}$) denotes the full subcategory of ${\rm coh}\mbox{-}\mathbb{X}$ consisting of coherent
sheaves without any simple subobjects (resp. coherent
sheaves of finite length) in ${\rm coh}\mbox{-}\mathbb{X}$. The symbol $\vee$ indicates that that each indecomposable object of ${\rm coh}\mbox{-}\mathbb{X}$ is either in ${\rm vect}\mbox{-}\mathbb{X}$ or in ${\rm coh}_{0}\mbox{-}\mathbb{X}$, and there are no non-zero morphisms from ${\rm coh}_{0}\mbox{-}\mathbb{X}$ to ${\rm vect}\mbox{-}\mathbb{X}.$ 

The objects in ${\rm vect}\mbox{-}\mathbb{X}$ are called \emph{vector bundles.} There is a specific type of vector bundles called \emph{line bundles}. Up to isomorphism, each line bundle  has
the form $\mathcal{O}(\vec{x})$ for a uniquely determined $\vec{x} \in \mathbb{L}$. In addition, we have natural isomorphisms 
\begin{equation}\label{hom space}
    \operatorname{Hom}(\mathcal{O}(\vec{x}), \mathcal{O}(\vec{y}))\cong R_{\vec{y}-\vec{x}}.
\end{equation}
 In the rest of the paper, We fix an arbitrary line bundle $L_0$ in ${\rm vect}\mbox{-}\mathbb{X}$. Recall that there are exact sequences 
 \begin{equation}\label{important exact sequence1}
  \begin{tikzcd}[ampersand replacement=\&,cramped,sep=small]
	0 \& {L_0(-\vec{x}_i)} \& L_0 \& {S_i} \& 0
	\arrow[from=1-1, to=1-2]
	\arrow["{{x_i}}", from=1-2, to=1-3]
	\arrow[from=1-3, to=1-4]
	\arrow[from=1-4, to=1-5]
\end{tikzcd}
\end{equation} 
\begin{equation}\label{important exact sequence2}
 \begin{tikzcd}[ampersand replacement=\&,cramped,sep=small]
	0 \& {L_0(-\vec{c})} \& L_0 \& {S_x} \& 0
	\arrow[from=1-1, to=1-2]
	\arrow[from=1-2, to=1-3]
	\arrow[from=1-3, to=1-4]
	\arrow[from=1-4, to=1-5]
\end{tikzcd}
\end{equation}
 defining simple sheaves $S_i$, concentrated in an exceptional point $x_i(1\leq i \leq 3)$ and $S_x$, concentrated in an ordinary point $x$. Moreover, $S_i\left(\vec{x}_j\right)$ equals $S_i$ or $S_i(-\vec{\omega})$ according as $j \neq i$ or $j=i$, and further that $S_x(\vec{y})=S_x$ holds for any $\vec{y}$ from $\mathbb{L}$. 

Define $\vec{\delta}:=2\vec{\omega}+\vec{c}=(n-2)\vec{x}_3$ as the \emph{dominant element} of $\mathbb{L}$. For any line bundle $L$ and $\vec{x}$ in $\mathbb{L}$ with $0 \leq \vec{x} \leq \vec{\delta}$, we have 
\begin{equation}\label{dim of ext eqe 1}
{\rm Ext^{1}}(L(\vec{x}),L(\vec{\omega}))\cong D\Hom(L,L(\vec{x}))\cong \mathbf{k}. 
\end{equation}
Consequently, the indecomposable middle term $E$ in the non-split exact sequence
\begin{equation}\label{extension bundle}
    \begin{tikzcd}[ampersand replacement=\&,cramped,sep=small]
	0 \& {L(\vec{\omega})} \& E \& {L(\vec{x})} \& 0
	\arrow[from=1-2, to=1-3]
	\arrow[from=1-3, to=1-4]
	\arrow[from=1-4, to=1-5]
	\arrow[from=1-1, to=1-2]
\end{tikzcd}
\end{equation}
	is uniquely determined up to isomorphism. We refer to $E_L\langle\vec{x}\rangle:=E$ as the \emph{extension bundle} for $L$ and $\vec{x}$. In particular, for $\vec{x} = 0$, the sequence \eqref{extension bundle} is almost-split, and $E_L := E_L\langle 0\rangle$ is called the \emph{Auslander bundle} associated with $L$.  
 
Each indecomposable bundle in ${\rm vect}\mbox{-}\mathbb{X}$ is either a line bundle or an extension bundle. Let $K_0(\mathbb{X})$ be the \emph{Grothendieck group} of ${\rm coh}\mbox{-}\mathbb{X}$. Then the class $\mathcal{O}(\vec{x}), 0 \leq \vec{x} \leq \vec{c}$, forms a $\mathbb{Z}$-basis of $K_0(\mathbb{X})$. There are two additive functions, \emph{rank} and \emph{degree}, on $K_0(\mathbb{X})$. The rank function $\mathrm{rk}: K_0(\mathbb{X})\rightarrow \mathbb{Z}$ is characterized by
$$
\operatorname{rk}(\mathcal{O}(\vec{x}))=1 \quad \text { for } \vec{x} \in \mathbb{L},
$$
and the degree function deg : $K_0(\mathbb{X}) \rightarrow \mathbb{Z}$ is uniquely determined by setting
$$
\operatorname{deg}(\mathcal{O}(\vec{x}))=\delta(\vec{x}) \quad \text { for } \vec{x} \in \mathbb{L},
$$
where $\delta: \mathbb{L} \rightarrow \mathbb{Z}$ is the group homomorphism defined on generators by
$$
\delta\left(\vec{x}_1\right)=\delta\left(\vec{x}_2\right)=\frac{\text { l.c.m. }(2, n)}{2} \text { and } \quad \delta\left(\vec{x}_3\right)=\frac{\text { l.c.m. }(2, n)}{n} .
$$

For each non-zero sheaf $X \in {\rm coh}\mbox{-}\mathbb{X}$, define the \emph{slope} of $X$ as
$$
\mu X=\frac{\operatorname{deg} X}{\operatorname{rk} X} .
$$

It is an element in $\mathbb{Q} \cup\{\infty\}$. According to \cite{MR915180}, each vector bundle has positive rank, its slope belongs to $\mathbb{Q}$; each non-zero object in ${\rm coh}_0\mbox{-}\mathbb{X}$ has rank 0 and positive degree, its slope is $\infty$. In \cite{MR2384609} Lenzing proved that each indecomposable vector bundle $X$ is exceptional in coh $\mathbb{X}$, that is, $X$ is extension-free and $\operatorname{End}(X)=\mathbf{k}$. Moreover, for any two indecomposable vector bundles $X$ and $X^{\prime}$, $\operatorname{Hom}\left(X, X^{\prime}\right) \neq 0$ implies $\mu X \leq \mu X^{\prime}$.

\section{Generalized extension bundles}\label{Section 3}
To facilitate subsequent work, we first give the definition
and basic properties of generalized extension bundles.

The following proposition describes the features of extension bundles.
\begin{prop}(\cite[Proposition 2.3]{dong2024two})\label{Grothendieck group}
    Assume $\vec{x},\vec{y},\vec{z}$ are elements in $\mathbb{L}$ such that $0\leq \vec{x},\vec{y}\leq  \vec{\delta}$. Then $E_L\langle\vec{x}\rangle=E_{L(\vec{z})}\langle\vec{y}\rangle$ if and only if one of following holds: 
    \begin{itemize}
        \item $\vec{y}=\vec{x}$ and $\vec{z}=0$ or $\vec{x}_1-\vec{x}_2$; 
        \item  $\vec{y}=\vec{\delta}-\vec{x}$ and $\vec{z}=\vec{x}+\vec{x}_3-\vec{x}_1$ or $\vec{x}+\vec{x}_3-\vec{x}_2$.
    \end{itemize}   
\end{prop}
\begin{rem}
    In the rest, for any line bundle $L$, we adopt $L^*$ to represent $L(\vec{x}_1-\vec{x}_2)$. Clearly, $(L^*)^*=L$ holds, and by Proposition \ref{Grothendieck group}, we have $E_L\langle\vec{x}\rangle=E_{L^*}\langle\vec{x}\rangle$.
\end{rem}

Consider the following commutative diagram obtained  by pushout 
\[
\begin{tikzcd}[ampersand replacement=\&,cramped]
	\& 0 \& 0 \\
	0 \& {L(\vec{\omega})} \& {L^*(\vec{x}_1-\vec{\omega})} \& {S_2(-\vec{x}_2)} \& 0 \\
	0 \& {L(\vec{x}_1-\vec{\omega})} \& {L((n-1)\vec{x}_3)} \& {S_2(-\vec{x}_2)} \& 0 \\
	\& {S_1(-\vec{x}_1)} \& {S_1(-\vec{x}_1)} \\
	\& 0 \& 0
	\arrow["{x_2}", from=3-2, to=3-3]
	\arrow[from=2-4, to=2-5]
	\arrow[from=3-3, to=4-3]
	\arrow[from=1-3, to=2-3]
	\arrow[Rightarrow, no head, from=4-2, to=4-3]
	\arrow[Rightarrow, no head, from=2-4, to=3-4]
	\arrow[from=3-4, to=3-5]
	\arrow[from=1-2, to=2-2]
	\arrow["{x_2}", from=2-2, to=2-3]
	\arrow["{x_1}"', from=2-2, to=3-2]
	\arrow["{x_1}", from=2-3, to=3-3]
	\arrow[from=4-2, to=5-2]
	\arrow[from=4-3, to=5-3]
	\arrow[from=3-2, to=4-2]
	\arrow[from=3-1, to=3-2]
	\arrow[from=2-1, to=2-2]
	\arrow[from=2-3, to=2-4]
	\arrow[from=3-3, to=3-4]
\end{tikzcd}
\]
with exact rows and columns.  Since the equation \eqref{dim of ext eqe 1} also holds for $\vec{x}= (n-1)\vec{x}_3$, we obtain the following non-split exact sequence
\[\begin{tikzcd}[ampersand replacement=\&,cramped,sep=small]
	0 \& {L(\vec{\omega})} \& {L(\vec{x}_1-\vec{\omega})\oplus L^*(\vec{x}_1-\vec{\omega})} \& {L((n-1)\vec{x}_3)} \& 0
	\arrow[from=1-3, to=1-4]
	\arrow[from=1-1, to=1-2]
	\arrow[from=1-2, to=1-3]
	\arrow[from=1-4, to=1-5]
\end{tikzcd}\]
uniquely, up to isomorphism. This inspires us to extend the notation $``E_L\langle\vec{x}\rangle"$ as follows:

\begin{defn}\label{generalized extension bundle}
    For $\vec{x}=l_3\vec{x}_3+l\vec{c} \in  \mathbb{L}$ written in normal form, define
\[\mathsf{E}_L\langle\vec{x}\rangle : = 
\begin{cases}
E_{L(l\vec{x}_1)}\langle l_3\vec{x}_3\rangle, & \text{if } 0 \leq l_3 \leq n-2; \\
L((l+1)\vec{x}_1-\vec{x}_3) \oplus L^*((l+1)\vec{x}_1-\vec{x}_3), & \text{otherwise.}
\end{cases}\]
\end{defn}

\begin{rem}\label{rem1}
 (i) We refer to $\mathsf{E}_L\langle\vec{x}\rangle$ as \emph{generalized extension bundles} in ${\rm vect}\mbox{-}\mathbb{X}$.   By comparing the classes of $\mathsf{E}_L\langle\vec{x}\rangle$ in $K_0(\mathbb{X})$, we can check that \[\mathsf{E}_L\langle\vec{x}\rangle=\mathsf{E}_{L^*}\langle\vec{x}\rangle=\mathsf{E}_{L(\vec{x}-\vec{x}_1+\vec{x}_3)}\langle\vec{\delta}-\vec{x}\rangle=\mathsf{E}_{L^*(\vec{x}-\vec{x}_1+\vec{x}_3)}\langle\vec{\delta}-\vec{x}\rangle,\] 
 as a general version of Propositon \ref{Grothendieck group}. 
 
 (ii) All the generalized extension bundles are parameterized by the set  \[\{ \mathsf{E}_{L(i\vec{x}_3)}\langle j\vec{x}_3\rangle \mid -1\leq j\leq n-1, i,j\in \mathbb{Z}\}.\] 
 Moreover, $\mathsf{E}_{L(i\vec{x}_3)}\langle j\vec{x}_3\rangle$ is indecomposable if and only if $0 \leq j\leq   n-2$.
\end{rem}

 Recall from \cite[Proposition 5.1]{MR3028577} that the Auslander-Reiten quiver $\Gamma({\rm vect}\mbox{-}\mathbb{X})$ of ${\rm vect}\mbox{-}\mathbb{X}$ has the form $\mathbb{Z} \tilde{D}_{n+2}$, where the quiver $\tilde{D}_{n+2}$ can be given as follows:
\[\begin{adjustbox}{scale=1}
\begin{tikzcd}[ampersand replacement=\&,cramped,sep=small]
	\& \bullet \&\&\&\& \bullet \\
	{\tilde{D}_{n+2}:} \&\& \bullet \& \cdots \& \bullet \\
	\& \bullet \&\&\&\& \bullet
	\arrow["{{{\alpha_1^+}}}", from=1-2, to=2-3]
	\arrow["{{{\alpha_1^-}}}"', from=3-2, to=2-3]
	\arrow["{{{\alpha_2}}}", from=2-3, to=2-4]
	\arrow["{{{\alpha_{n-1}}}}", from=2-4, to=2-5]
	\arrow["{{{\alpha_n^+}}}", from=2-5, to=1-6]
	\arrow["{{{\alpha_n^-}}}"', from=2-5, to=3-6]
\end{tikzcd}
\end{adjustbox}\]
For example, let $n=4$. The Auslander-Reiten quiver $\Gamma({\rm vect}\mbox{-}\mathbb{X})$ is illustrated as below:
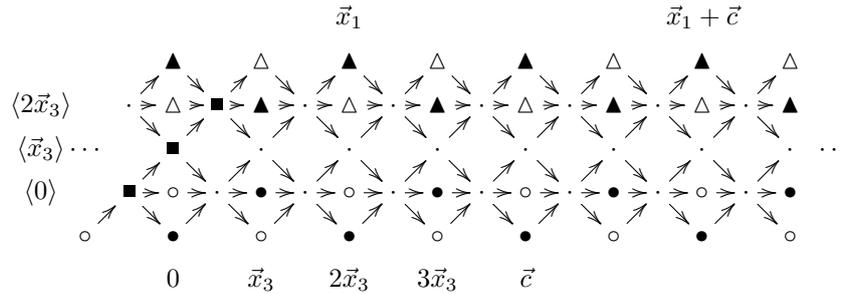
\begin{figure}[H]
    \hspace{2cm}
\xymatrix@-1pc @!R=4.5pt @!C=4.5pt{
&&&&&&&\vx_1&&&&&&&&\vx_1+\vec{c}&&&\\
&&&\blacktriangle\ar[rd]            &&\vartriangle\ar[rd]&&\blacktriangle\ar[rd]&&\vartriangle\ar[rd]            &&\blacktriangle\ar[rd]            &&\vartriangle\ar[rd]&&\blacktriangle\ar[rd]&&\vartriangle\\
\langle 2\vx_3 \rangle && \cdot\ar[r]\ar[ru]\ar[rd]&\vartriangle\ar[r]      &\rule[.2ex]{1ex}{1ex}\ar[ru]\ar[r]\ar[rd]&\blacktriangle\ar[r]&\cdot \ar[ru]\ar[r]\ar[rd]&\vartriangle\ar[r]&\cdot\ar[ru]\ar[r]\ar[rd]&      \blacktriangle\ar[r]&\cdot\ar[ru]\ar[rd]\ar[r]&\vartriangle\ar[r]      &\cdot\ar[ru]\ar[r]\ar[rd]&\blacktriangle\ar[r]&\cdot\ar[ru]\ar[r]\ar[rd]&\vartriangle\ar[r]&\cdot\ar[ru]\ar[r]\ar[rd]&      \blacktriangle\\
\langle \vx_3 \rangle  &\cdots&&\rule[.2ex]{1ex}{1ex}\ar[ru]\ar[rd]      &&\cdot\ar[ru]\ar[rd]&&\cdot\ar[ru]\ar[rd]&&      \cdot\ar[ru]\ar[rd]&&\cdot\ar[ru]\ar[rd]      &&\cdot\ar[ru]\ar[rd]&&\cdot\ar[ru]\ar[rd]&&  \cdot &\cdots   \\
\langle 0 \rangle &&\rule[.2ex]{1ex}{1ex}\ar[r]\ar[ru]\ar[rd]&\circ\ar[r]      &\cdot\ar[ru]\ar[r]\ar[rd]&\bullet\ar[r]&\cdot\ar[ru]\ar[r]\ar[rd]&\circ\ar[r]&\cdot\ar[ru]\ar[r]\ar[rd]&      \bullet\ar[r]&\cdot\ar[ru]\ar[rd]\ar[r]&\circ\ar[r]      &\cdot\ar[ru]\ar[r]\ar[rd]&\bullet\ar[r]&\cdot\ar[ru]\ar[r]\ar[rd]&\circ\ar[r]&\cdot\ar[ru]\ar[r]\ar[rd]&      \bullet\\
&\circ\ar[ru]&&\bullet\ar[ru]      &&\circ\ar[ru]&&\bullet\ar[ru]&&      \circ\ar[ru]&&\bullet\ar[ru]      &&\circ\ar[ru]&&\bullet\ar[ru]&&      \circ\\
 &&&0      &&\vx_3&&2\vx_3&&3\vx_3&&\vec{c}&
}
\end{figure}
\noindent In the illustration, we mark the positions of the extension bundles $E_{L_0}\langle \vec{x}\rangle$, $0\leq \vec{x}\leq \vec{\delta}$, by the symbol $\rule[.2ex]{1ex}{1ex}$, and the position of the line bundles $L_0(j\vec{x}_3)$, $L_0^*(j\vec{x}_3)$, $L_0(\vec{x}_1+j\vec{x}_3)$,  $L_0^*(\vec{x}_1+j\vec{x}_3)$, $j\in\mathbb{Z}$ by the symbols $\bullet$, $\circ$, $\blacktriangle$, $\vartriangle$ respectively. The remaining extension bundles are marked by the symbol $\cdot$. 

Let $F$ denote the degree shift automorphism by $\vec{x}_1-\vec{x}_2$ on ${\rm vect}\mbox{-}\mathbb{X}$.  An object $X$ in ${\rm vect}\mbox{-}\mathbb{X}$ is \emph{$F$-stable} provided that $X$ is isomorphic to $F(X)$. We denote by ${\rm vect}^F\mbox{-}\mathbb{X}$ the full subcategory of ${\rm vect}\mbox{-}\mathbb{X}$ consisting of $F$-stable objects. 

\begin{obs}
    Generalized extension bundles are $F$-stable objects in ${\rm vect}^F\mbox{-}\mathbb{X}$. Moreover, indecomposable objects in ${\rm vect}^F\mbox{-}\mathbb{X}$ are generalized extension bundles. 
\end{obs}

Let $\sigma_F$ be the automorphism  of $\Gamma(\text{vect-}\mathbb{X})$ induced by the automorphism $F$ on $\text{vect-}\mathbb{X}$. We observe that $\sigma_F$ exchanges the $\blacktriangle$ and $\vartriangle$ on the same vertical lines, as well as the $\circ$ and $\bullet$, while fixing all other vertices. Consequently, we can construct a  valued translation quiver $(\Gamma, d)$ by folding $\Gamma(\text{vect-}\mathbb{X})$ as follows:
\begin{itemize}
\item $\Gamma_0=\{\text{vertex orbits of}\;\sigma_F\}$,
\item $\Gamma_1=\{\text{arrow orbits of}\; \sigma_F\}$,
\item For each $x \in \Gamma_0$, $d(x)$ is the number of vertices in the orbit $x$,
\item For each $\alpha: x\rightarrow y \in \Gamma_1$, $d(\alpha)$ is the pair $(d(y), d(x))$.
\end{itemize}
In fact, $\Gamma$ is of the form $\mathbb{Z}A_{n+1}$, where $A_{n+1}$ is the quiver 
\[\begin{tikzcd}[ampersand replacement=\&,cramped]
	{A_{n+1}:} \& \bullet \& \bullet \& \cdots \& \bullet
	\arrow[from=1-2, to=1-3]
	\arrow["{n+1}"'{pos=1.3}, shift right, draw=none, from=1-4, to=1-5]
	\arrow["1"'{pos=-0.2}, shift right, draw=none, from=1-2, to=1-3]
	\arrow[from=1-3, to=1-4]
	\arrow["2"'{pos=-0.2}, shift right, draw=none, from=1-3, to=1-4]
	\arrow[from=1-4, to=1-5]
\end{tikzcd}.\]
For example, let $n=4$. Then $(\Gamma,d)$ is as follows:
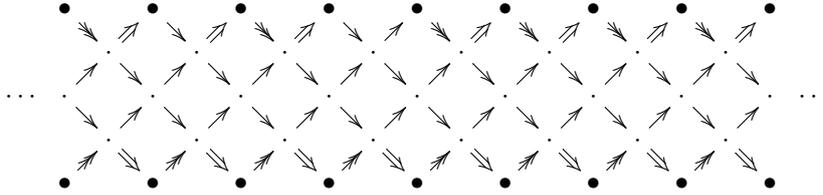
\begin{figure}[H]
    \hspace{0.1cm}
\xymatrix@-1pc @!R=4.5pt @!C=4.5pt{
&\bullet \ar@{->>}[rd]&&\bullet \ar[rd]            &&\bullet\ar@{->>}[rd]&&\bullet\ar[rd]&&\bullet\ar@{->>}[rd]            &&\bullet\ar@{->>}[rd]            &&\bullet\ar@{->>}[rd]&&\bullet\ar@{->>}[rd]&&\bullet\\
 && \cdot\ar@{=>}[ru]\ar[rd]&      &\cdot\ar@{=>}[ru]\ar[rd]&&\cdot \ar@{=>}[ru]\ar[rd]&&\cdot\ar[ru]\ar[rd]&      &\cdot\ar@{=>}[ru]\ar[rd]&     &\cdot\ar@{=>}[ru]\ar[rd]&&\cdot\ar@{=>}[ru]\ar[rd]&&\cdot\ar@{=>}[ru]\ar[rd]&      \\
 \cdots&\cdot\ar[ru]\ar[rd]&&\cdot\ar[ru]\ar[rd]      &&\cdot\ar[ru]\ar[rd]&&\cdot\ar[ru]\ar[rd]&&      \cdot\ar[ru]\ar[rd]&&\cdot\ar[ru]\ar[rd]      &&\cdot\ar[ru]\ar[rd]&&\cdot\ar[ru]\ar[rd]&&  \cdot &\cdots    \\
 && \cdot\ar[ru]\ar@{=>}[rd]&      &\cdot\ar[ru]\ar@{=>}[rd]&&\cdot \ar[ru]\ar@{=>}[rd]&&\cdot\ar[ru]\ar@{=>}[rd]&      &\cdot\ar[ru]\ar@{=>}[rd]&     &\cdot\ar[ru]\ar@{=>}[rd]&&\cdot\ar[ru]\ar@{=>}[rd]&&\cdot\ar[ru]\ar@{=>}[rd]&      \\
&\bullet\ar@{->>}[ru]&&\bullet\ar@{->>}[ru]      &&\bullet\ar@{->>}[ru]&&\bullet\ar@{->>}[ru]&&      \bullet\ar@{->>}[ru]&&\bullet\ar@{->>}[ru]      &&\bullet\ar@{->>}[ru]&&\bullet\ar@{->>}[ru]&&      \bullet
}
\end{figure}
\noindent In the illustration, we mark the points with the value 2 by the symbol $\bullet$,  while the remaining points (with the value 1) are marked by the symbol $\cdot$. We use the symbols $\rightarrow$, $\Rightarrow$, and $\twoheadrightarrow$ to represent arrows with values $(1,1)$, $(2,1)$, and $(1,2)$ respectively. Indeed, this  valued translation quiver serves as the ``Auslander-Reiten quiver'' of ${\rm vect}^F\mbox{-}\mathbb{X}$ in following sense:
\begin{itemize}
    \item [(a)] The points of $(\Gamma,d)$ correspond to the indecomposable objects in ${\rm vect}^F\mbox{-}\mathbb{X}$;
    \item [(b)] Let $[X]$ and $[Y]$ be the points in ${\rm vect}^F\mbox{-}\mathbb{X}$ corresponding to the indecomposable object $X,Y$ in ${\rm vect}^F\mbox{-}\mathbb{X}$. Then $\Hom(X,Y)\neq 0$ if and only if there exists a path from $[X]$ to $[Y]$ in $(\Gamma,d)$;
    \item [(c)] The valuation of $[X]$ equals the number of indecomposable direct summands (from ${\rm vect}\mbox{-}\mathbb{X}$) contained in $X$.
\end{itemize}

Denote by $\tau_F$ the translation of the valued translation quiver $(\Gamma, d)$. Recall that automorphisms of $(\Gamma, d)$ are defined as the quiver automorphisms that commute with the translation $\tau_F$ and are compatible with $d$.   
From \S \ref{Section 3}, we already know that $\Gamma$ is of the form $\mathbb{Z}A_{n+1}$. The set of points is \[(\mathbb{Z}A_{n+1})_0=\mathbb{Z}\times A_{n+1}=\{(s,i)\mid s\in \mathbb{Z} \text{ and } 1\leq i \leq n+1\}.\] 
The translation $\tau_F$ sends the point $(s,i)$ to $(s+1,i)$ for any $s\in \mathbb{Z}$ and $1\leq i \leq n+1$. Let $\rho$ be the automorphism that sends the point $(s, i)$ to the point $(s+i-1,n-i+2)$ for any $s\in \mathbb{Z}$ and $1\leq i \leq n+1$. As shown in the figure
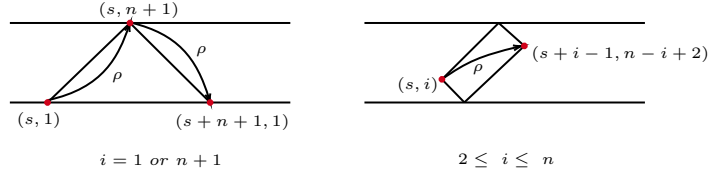
\begin{figure}[H]
    \centering

\tikzset{every picture/.style={line width=0.75pt}}          

\begin{tikzpicture}[x=0.75pt,y=0.75pt,yscale=-1,xscale=1]

\draw    (190.63,79.83) .. controls (212.64,85.46) and (222.69,96.62) .. (228.99,118.16) ;
\draw [shift={(229.47,119.83)}, rotate = 254.52] [color={rgb, 255:red, 0; green, 0; blue, 0 }  ][line width=0.75]    (4.37,-1.32) .. controls (2.78,-0.56) and (1.32,-0.12) .. (0,0) .. controls (1.32,0.12) and (2.78,0.56) .. (4.37,1.32)   ;
  
\draw    (148.3,118.67) .. controls (176.43,112.65) and (183.82,96.19) .. (188.85,81.64) ;
\draw [shift={(189.47,79.83)}, rotate = 108.7] [color={rgb, 255:red, 0; green, 0; blue, 0 }  ][line width=0.75]    (4.37,-1.32) .. controls (2.78,-0.56) and (1.32,-0.12) .. (0,0) .. controls (1.32,0.12) and (2.78,0.56) .. (4.37,1.32)   ;
  
\draw    (229.47,119.83) -- (189.47,79.83) ;
  
\draw    (345.17,108.17) .. controls (357.61,99.34) and (364.36,96.35) .. (383.37,91.79) ;
\draw [shift={(385.17,91.37)}, rotate = 166.76] [color={rgb, 255:red, 0; green, 0; blue, 0 }  ][line width=0.75]    (4.37,-1.32) .. controls (2.78,-0.56) and (1.32,-0.12) .. (0,0) .. controls (1.32,0.12) and (2.78,0.56) .. (4.37,1.32)   ;
  
\draw    (129.5,79.83) -- (269.5,79.83) ;
  
\draw    (306.5,79.83) -- (446.5,79.83) ;
  
\draw    (129.5,119.83) -- (269.5,119.83) ;
  
\draw    (306.5,119.83) -- (446.5,119.83) ;
  
\draw    (356.37,119.77) -- (345.17,108.17) ;
  
\draw    (386.33,91.37) -- (373.5,79.83) ;
  
\draw    (345.17,108.17) -- (373.5,79.83) ;
  
\draw    (356.37,119.77) -- (386.33,91.37) ;
  
\draw  [color={rgb, 255:red, 208; green, 2; blue, 27 }  ,draw opacity=1 ][fill={rgb, 255:red, 208; green, 2; blue, 27 }  ,fill opacity=1 ] (344,108.17) .. controls (344,107.52) and (344.52,107) .. (345.17,107) .. controls (345.81,107) and (346.33,107.52) .. (346.33,108.17) .. controls (346.33,108.81) and (345.81,109.33) .. (345.17,109.33) .. controls (344.52,109.33) and (344,108.81) .. (344,108.17) -- cycle ;
  
\draw  [color={rgb, 255:red, 208; green, 2; blue, 27 }  ,draw opacity=1 ][fill={rgb, 255:red, 208; green, 2; blue, 27 }  ,fill opacity=1 ] (385.17,91.37) .. controls (385.17,90.72) and (385.69,90.2) .. (386.33,90.2) .. controls (386.98,90.2) and (387.5,90.72) .. (387.5,91.37) .. controls (387.5,92.01) and (386.98,92.53) .. (386.33,92.53) .. controls (385.69,92.53) and (385.17,92.01) .. (385.17,91.37) -- cycle ;
  
\draw    (148.3,119.83) -- (189.47,79.83) ;
  
\draw  [color={rgb, 255:red, 208; green, 2; blue, 27 }  ,draw opacity=1 ][fill={rgb, 255:red, 208; green, 2; blue, 27 }  ,fill opacity=1 ] (147.13,119.83) .. controls (147.13,119.19) and (147.66,118.67) .. (148.3,118.67) .. controls (148.94,118.67) and (149.47,119.19) .. (149.47,119.83) .. controls (149.47,120.48) and (148.94,121) .. (148.3,121) .. controls (147.66,121) and (147.13,120.48) .. (147.13,119.83) -- cycle ;
  
\draw  [color={rgb, 255:red, 208; green, 2; blue, 27 }  ,draw opacity=1 ][fill={rgb, 255:red, 208; green, 2; blue, 27 }  ,fill opacity=1 ] (188.3,79.83) .. controls (188.3,79.19) and (188.82,78.67) .. (189.47,78.67) .. controls (190.11,78.67) and (190.63,79.19) .. (190.63,79.83) .. controls (190.63,80.48) and (190.11,81) .. (189.47,81) .. controls (188.82,81) and (188.3,80.48) .. (188.3,79.83) -- cycle ;
  
\draw  [color={rgb, 255:red, 208; green, 2; blue, 27 }  ,draw opacity=1 ][fill={rgb, 255:red, 208; green, 2; blue, 27 }  ,fill opacity=1 ] (228.3,119.83) .. controls (228.3,119.19) and (228.82,118.67) .. (229.47,118.67) .. controls (230.11,118.67) and (230.63,119.19) .. (230.63,119.83) .. controls (230.63,120.48) and (230.11,121) .. (229.47,121) .. controls (228.82,121) and (228.3,120.48) .. (228.3,119.83) -- cycle ;

\draw (321.27,105.4) node [anchor=north west][inner sep=0.75pt]  [font=\tiny]  {$( s,i)$};
 
\draw (388.33,89.77) node [anchor=north west][inner sep=0.75pt]  [font=\tiny]  {$( s+i-1,n-i+2)$};
 
\draw (359.07,99) node [anchor=north west][inner sep=0.75pt]  [font=\tiny]  {$\rho $};
 
\draw (131.5,123.23) node [anchor=north west][inner sep=0.75pt]  [font=\tiny]  {$( s,1)$};
 
\draw (172.3,68.03) node [anchor=north west][inner sep=0.75pt]  [font=\tiny]  {$( s,n+1)$};
 
\draw (179.07,103) node [anchor=north west][inner sep=0.75pt]  [font=\tiny]  {$\rho $};
 
\draw (210.47,124.23) node [anchor=north west][inner sep=0.75pt]  [font=\tiny]  {$( s+n+1,1)$};
 
\draw (221.07,89.8) node [anchor=north west][inner sep=0.75pt]  [font=\tiny]  {$\rho $};
 
\draw (348.43,143.6) node [anchor=north west][inner sep=0.75pt]  [font=\tiny]  {$\ 2\leq \ i\leq \ n$};
 
\draw (169.43,144.6) node [anchor=north west][inner sep=0.75pt]  [font=\tiny]  {$\ i=1\ or\ n+1$};

\end{tikzpicture}

\end{figure}
\noindent Since, for each point of $\Gamma$, the number of incoming and outgoing arrows is preserved under a quiver automorphism, it is evident that every automorphism of $\Gamma$ can be expressed as a combination of $\tau_F$ and $\rho$. We note the following relations between these two generators:
\begin{itemize}
    \item  the translation $\tau_F$ is of infinite order;
    \item $\rho$ commutes with $\tau_F$;  
    \item $\rho^2$ equals $\tau_F^n$.
\end{itemize}
Thus, we get a presentation of the group of automorphisms of $(\Gamma, d)$ as
\begin{equation}\label{equ 111}
    \operatorname{Aut}(\Gamma, d)=\langle \tau_F,\rho\mid \rho^2=\tau_F^n, \tau_F\rho=\rho\tau_F\rangle.
\end{equation}

\begin{rem}
    Note that the Picard group $\mathbb{L}$ does not act faithfully on the category ${\rm vect}^F\mbox{-}\mathbb{X}$, since the degree shift by $\vec{x}_1-\vec{x}_2$ fixes all generalized extension bundles. Consequently, $\mathbb{L}/\mathbb{Z}(\vec{x}_1-\vec{x}_2)$ acts on the category ${\rm vect}^F\mbox{-}\mathbb{X}$.  It can be verified that $\tau_F$ corresponds to the degree shift by $\vec{x}_3$ and $\rho$ corresponds to the degree shift by $\vec{x}_1$ in $\mathbb{L}/\mathbb{Z}(\vec{x}_1-\vec{x}_2)$. Therefore, in the rest, we identify $\operatorname{Aut}(\Gamma, d)$ with $\mathbb{L}/\mathbb{Z}(\vec{x}_1-\vec{x}_2)$.
\end{rem}

\section{An infinite marked strip under a specific group action}\label{Section 4}
In this section, we present the orbit space of an infinite marked strip with respect to a $G$-action as a geometric model for ${\rm vect}\mbox{-}\mathbb{X}$. We give the correspondence between $G$-orbits of line segments and indecomposable objects in ${\rm vect}\mbox{-}\mathbb{X}$.  Moreover, we give geometric interpretations for various aspects including the automorphism group of $(\Gamma, d)$, the slope of indecomposable bundles, the action of the Picard group action and vector bundle duality. 

\subsection{An infinite marked strip}\label{A marked infinite strip}
Consider an infinite  marked strip in the plane, which can be represented in $\mathbb{R}^2$ as $\widetilde{\mathcal{S}}=\{(x,y)\in\mathbb{R}^{2}\mid 0\leq y\leq 1\}$. The lines $y=0$ and $y=1$ respectively form the lower boundary $\partial^\prime$ and upper boundary $\partial$ of $\widetilde{\mathcal{S}}$. Let $M$ be the set of points $\{(i,0),(j,1)\mid i,j\in \mathbb{Z}\}$, called \emph{marked points} of $\widetilde{\mathcal{S}}$. 

Consider two bijections $\sigma_n$ and $\theta$ defined on the strip $\widetilde{\mathcal{S}}$, where:

\begin{itemize}
    \item $\sigma_n$ translates all points on $\widetilde{\mathcal{S}}$ along the positive $x$-axis by $n$ units:
    \[
    \sigma_n: \widetilde{\mathcal{S}}\rightarrow\widetilde{\mathcal{S}} \quad (x, y)\mapsto (x+n, y);
    \]
    
    \item $\theta$ reflects all points on $\widetilde{\mathcal{S}}$ with respect to the point $(0,\frac{1}{2})$:
    \[
    \theta: \widetilde{\mathcal{S}}\rightarrow\widetilde{\mathcal{S}} \quad (x, y)\mapsto (-x, 1-y).
    \]
\end{itemize}

The bijections $\sigma_n$ and $\theta$ are depicted as follows.
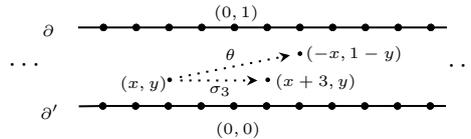
\begin{figure}[H]
    \centering
\tikzset{every picture/.style={line width=0.75pt}}          
\begin{tikzpicture}[x=0.75pt,y=0.75pt,yscale=-1,xscale=1]
\draw    (43,19) -- (55.67,19) ;
\draw [shift={(55.67,19)}, rotate = 360] [color={rgb, 255:red, 0; green, 0; blue, 0 }  ][fill={rgb, 255:red, 0; green, 0; blue, 0 }  ][line width=0.75]      (0, 0) circle [x radius= 1.34, y radius= 1.34]   ;
\draw    (55.67,19) -- (72.02,19) ;
\draw [shift={(72.02,19)}, rotate = 0] [color={rgb, 255:red, 0; green, 0; blue, 0 }  ][fill={rgb, 255:red, 0; green, 0; blue, 0 }  ][line width=0.75]      (0, 0) circle [x radius= 1.34, y radius= 1.34]   ;
\draw [shift={(55.67,19)}, rotate = 0] [color={rgb, 255:red, 0; green, 0; blue, 0 }  ][fill={rgb, 255:red, 0; green, 0; blue, 0 }  ][line width=0.75]      (0, 0) circle [x radius= 1.34, y radius= 1.34]   ;
\draw    (72.02,19) -- (88.36,19) ;
\draw [shift={(88.36,19)}, rotate = 0] [color={rgb, 255:red, 0; green, 0; blue, 0 }  ][fill={rgb, 255:red, 0; green, 0; blue, 0 }  ][line width=0.75]      (0, 0) circle [x radius= 1.34, y radius= 1.34]   ;
\draw [shift={(72.02,19)}, rotate = 0] [color={rgb, 255:red, 0; green, 0; blue, 0 }  ][fill={rgb, 255:red, 0; green, 0; blue, 0 }  ][line width=0.75]      (0, 0) circle [x radius= 1.34, y radius= 1.34]   ;
\draw    (88.36,19) -- (104.7,19) ;
\draw [shift={(104.7,19)}, rotate = 0] [color={rgb, 255:red, 0; green, 0; blue, 0 }  ][fill={rgb, 255:red, 0; green, 0; blue, 0 }  ][line width=0.75]      (0, 0) circle [x radius= 1.34, y radius= 1.34]   ;
\draw [shift={(88.36,19)}, rotate = 0] [color={rgb, 255:red, 0; green, 0; blue, 0 }  ][fill={rgb, 255:red, 0; green, 0; blue, 0 }  ][line width=0.75]      (0, 0) circle [x radius= 1.34, y radius= 1.34]   ;
\draw    (104.7,19) -- (121.04,19) ;
\draw [shift={(121.04,19)}, rotate = 0] [color={rgb, 255:red, 0; green, 0; blue, 0 }  ][fill={rgb, 255:red, 0; green, 0; blue, 0 }  ][line width=0.75]      (0, 0) circle [x radius= 1.34, y radius= 1.34]   ;
\draw [shift={(104.7,19)}, rotate = 0] [color={rgb, 255:red, 0; green, 0; blue, 0 }  ][fill={rgb, 255:red, 0; green, 0; blue, 0 }  ][line width=0.75]      (0, 0) circle [x radius= 1.34, y radius= 1.34]   ;
\draw    (121.04,19) -- (137.38,19) ;
\draw [shift={(137.38,19)}, rotate = 0] [color={rgb, 255:red, 0; green, 0; blue, 0 }  ][fill={rgb, 255:red, 0; green, 0; blue, 0 }  ][line width=0.75]      (0, 0) circle [x radius= 1.34, y radius= 1.34]   ;
\draw [shift={(121.04,19)}, rotate = 0] [color={rgb, 255:red, 0; green, 0; blue, 0 }  ][fill={rgb, 255:red, 0; green, 0; blue, 0 }  ][line width=0.75]      (0, 0) circle [x radius= 1.34, y radius= 1.34]   ;
\draw    (137.38,19) -- (153.72,19) ;
\draw [shift={(153.72,19)}, rotate = 0] [color={rgb, 255:red, 0; green, 0; blue, 0 }  ][fill={rgb, 255:red, 0; green, 0; blue, 0 }  ][line width=0.75]      (0, 0) circle [x radius= 1.34, y radius= 1.34]   ;
\draw [shift={(137.38,19)}, rotate = 0] [color={rgb, 255:red, 0; green, 0; blue, 0 }  ][fill={rgb, 255:red, 0; green, 0; blue, 0 }  ][line width=0.75]      (0, 0) circle [x radius= 1.34, y radius= 1.34]   ;
\draw    (153.72,19) -- (170.06,19) ;
\draw [shift={(170.06,19)}, rotate = 0] [color={rgb, 255:red, 0; green, 0; blue, 0 }  ][fill={rgb, 255:red, 0; green, 0; blue, 0 }  ][line width=0.75]      (0, 0) circle [x radius= 1.34, y radius= 1.34]   ;
\draw [shift={(153.72,19)}, rotate = 0] [color={rgb, 255:red, 0; green, 0; blue, 0 }  ][fill={rgb, 255:red, 0; green, 0; blue, 0 }  ][line width=0.75]      (0, 0) circle [x radius= 1.34, y radius= 1.34]   ;
\draw    (170.06,19) -- (186.4,19) ;
\draw [shift={(186.4,19)}, rotate = 0] [color={rgb, 255:red, 0; green, 0; blue, 0 }  ][fill={rgb, 255:red, 0; green, 0; blue, 0 }  ][line width=0.75]      (0, 0) circle [x radius= 1.34, y radius= 1.34]   ;
\draw [shift={(170.06,19)}, rotate = 0] [color={rgb, 255:red, 0; green, 0; blue, 0 }  ][fill={rgb, 255:red, 0; green, 0; blue, 0 }  ][line width=0.75]      (0, 0) circle [x radius= 1.34, y radius= 1.34]   ;
\draw    (186.4,19) -- (202.74,19) ;
\draw [shift={(202.74,19)}, rotate = 0] [color={rgb, 255:red, 0; green, 0; blue, 0 }  ][fill={rgb, 255:red, 0; green, 0; blue, 0 }  ][line width=0.75]      (0, 0) circle [x radius= 1.34, y radius= 1.34]   ;
\draw [shift={(186.4,19)}, rotate = 0] [color={rgb, 255:red, 0; green, 0; blue, 0 }  ][fill={rgb, 255:red, 0; green, 0; blue, 0 }  ][line width=0.75]      (0, 0) circle [x radius= 1.34, y radius= 1.34]   ;
\draw    (202.74,19) -- (219.08,19) ;
\draw [shift={(219.08,19)}, rotate = 0] [color={rgb, 255:red, 0; green, 0; blue, 0 }  ][fill={rgb, 255:red, 0; green, 0; blue, 0 }  ][line width=0.75]      (0, 0) circle [x radius= 1.34, y radius= 1.34]   ;
\draw [shift={(202.74,19)}, rotate = 0] [color={rgb, 255:red, 0; green, 0; blue, 0 }  ][fill={rgb, 255:red, 0; green, 0; blue, 0 }  ][line width=0.75]      (0, 0) circle [x radius= 1.34, y radius= 1.34]   ;
\draw    (43.33,58.67) -- (55.34,58.5) ;
\draw [shift={(55.34,58.5)}, rotate = 359.2] [color={rgb, 255:red, 0; green, 0; blue, 0 }  ][fill={rgb, 255:red, 0; green, 0; blue, 0 }  ][line width=0.75]      (0, 0) circle [x radius= 1.34, y radius= 1.34]   ;
\draw    (55.34,58.5) -- (71.68,58.5) ;
\draw [shift={(71.68,58.5)}, rotate = 0] [color={rgb, 255:red, 0; green, 0; blue, 0 }  ][fill={rgb, 255:red, 0; green, 0; blue, 0 }  ][line width=0.75]      (0, 0) circle [x radius= 1.34, y radius= 1.34]   ;
\draw [shift={(55.34,58.5)}, rotate = 0] [color={rgb, 255:red, 0; green, 0; blue, 0 }  ][fill={rgb, 255:red, 0; green, 0; blue, 0 }  ][line width=0.75]      (0, 0) circle [x radius= 1.34, y radius= 1.34]   ;
\draw    (71.68,58.5) -- (88.02,58.5) ;
\draw [shift={(88.02,58.5)}, rotate = 0] [color={rgb, 255:red, 0; green, 0; blue, 0 }  ][fill={rgb, 255:red, 0; green, 0; blue, 0 }  ][line width=0.75]      (0, 0) circle [x radius= 1.34, y radius= 1.34]   ;
\draw [shift={(71.68,58.5)}, rotate = 0] [color={rgb, 255:red, 0; green, 0; blue, 0 }  ][fill={rgb, 255:red, 0; green, 0; blue, 0 }  ][line width=0.75]      (0, 0) circle [x radius= 1.34, y radius= 1.34]   ;
\draw    (88.02,58.5) -- (104.36,58.5) ;
\draw [shift={(104.36,58.5)}, rotate = 0] [color={rgb, 255:red, 0; green, 0; blue, 0 }  ][fill={rgb, 255:red, 0; green, 0; blue, 0 }  ][line width=0.75]      (0, 0) circle [x radius= 1.34, y radius= 1.34]   ;
\draw [shift={(88.02,58.5)}, rotate = 0] [color={rgb, 255:red, 0; green, 0; blue, 0 }  ][fill={rgb, 255:red, 0; green, 0; blue, 0 }  ][line width=0.75]      (0, 0) circle [x radius= 1.34, y radius= 1.34]   ;
\draw    (104.36,58.5) -- (120.7,58.5) ;
\draw [shift={(120.7,58.5)}, rotate = 0] [color={rgb, 255:red, 0; green, 0; blue, 0 }  ][fill={rgb, 255:red, 0; green, 0; blue, 0 }  ][line width=0.75]      (0, 0) circle [x radius= 1.34, y radius= 1.34]   ;
\draw [shift={(104.36,58.5)}, rotate = 0] [color={rgb, 255:red, 0; green, 0; blue, 0 }  ][fill={rgb, 255:red, 0; green, 0; blue, 0 }  ][line width=0.75]      (0, 0) circle [x radius= 1.34, y radius= 1.34]   ;
\draw    (120.7,58.5) -- (137.05,58.5) ;
\draw [shift={(137.05,58.5)}, rotate = 0] [color={rgb, 255:red, 0; green, 0; blue, 0 }  ][fill={rgb, 255:red, 0; green, 0; blue, 0 }  ][line width=0.75]      (0, 0) circle [x radius= 1.34, y radius= 1.34]   ;
\draw [shift={(120.7,58.5)}, rotate = 0] [color={rgb, 255:red, 0; green, 0; blue, 0 }  ][fill={rgb, 255:red, 0; green, 0; blue, 0 }  ][line width=0.75]      (0, 0) circle [x radius= 1.34, y radius= 1.34]   ;
\draw    (137.05,58.5) -- (153.39,58.5) ;
\draw [shift={(153.39,58.5)}, rotate = 0] [color={rgb, 255:red, 0; green, 0; blue, 0 }  ][fill={rgb, 255:red, 0; green, 0; blue, 0 }  ][line width=0.75]      (0, 0) circle [x radius= 1.34, y radius= 1.34]   ;
\draw [shift={(137.05,58.5)}, rotate = 0] [color={rgb, 255:red, 0; green, 0; blue, 0 }  ][fill={rgb, 255:red, 0; green, 0; blue, 0 }  ][line width=0.75]      (0, 0) circle [x radius= 1.34, y radius= 1.34]   ;
\draw    (153.39,58.5) -- (169.73,58.5) ;
\draw [shift={(169.73,58.5)}, rotate = 0] [color={rgb, 255:red, 0; green, 0; blue, 0 }  ][fill={rgb, 255:red, 0; green, 0; blue, 0 }  ][line width=0.75]      (0, 0) circle [x radius= 1.34, y radius= 1.34]   ;
\draw [shift={(153.39,58.5)}, rotate = 0] [color={rgb, 255:red, 0; green, 0; blue, 0 }  ][fill={rgb, 255:red, 0; green, 0; blue, 0 }  ][line width=0.75]      (0, 0) circle [x radius= 1.34, y radius= 1.34]   ;
\draw    (169.73,58.5) -- (186.07,58.5) ;
\draw [shift={(186.07,58.5)}, rotate = 0] [color={rgb, 255:red, 0; green, 0; blue, 0 }  ][fill={rgb, 255:red, 0; green, 0; blue, 0 }  ][line width=0.75]      (0, 0) circle [x radius= 1.34, y radius= 1.34]   ;
\draw [shift={(169.73,58.5)}, rotate = 0] [color={rgb, 255:red, 0; green, 0; blue, 0 }  ][fill={rgb, 255:red, 0; green, 0; blue, 0 }  ][line width=0.75]      (0, 0) circle [x radius= 1.34, y radius= 1.34]   ;
\draw    (186.07,58.5) -- (202.41,58.5) ;
\draw [shift={(202.41,58.5)}, rotate = 0] [color={rgb, 255:red, 0; green, 0; blue, 0 }  ][fill={rgb, 255:red, 0; green, 0; blue, 0 }  ][line width=0.75]      (0, 0) circle [x radius= 1.34, y radius= 1.34]   ;
\draw [shift={(186.07,58.5)}, rotate = 0] [color={rgb, 255:red, 0; green, 0; blue, 0 }  ][fill={rgb, 255:red, 0; green, 0; blue, 0 }  ][line width=0.75]      (0, 0) circle [x radius= 1.34, y radius= 1.34]   ;
\draw    (202.41,58.5) -- (218.75,58.5) ;
\draw [shift={(218.75,58.5)}, rotate = 0] [color={rgb, 255:red, 0; green, 0; blue, 0 }  ][fill={rgb, 255:red, 0; green, 0; blue, 0 }  ][line width=0.75]      (0, 0) circle [x radius= 1.34, y radius= 1.34]   ;
\draw [shift={(202.41,58.5)}, rotate = 0] [color={rgb, 255:red, 0; green, 0; blue, 0 }  ][fill={rgb, 255:red, 0; green, 0; blue, 0 }  ][line width=0.75]      (0, 0) circle [x radius= 1.34, y radius= 1.34]   ;
\draw   (88,45.33) .. controls (88,44.97) and (88.3,44.67) .. (88.67,44.67) .. controls (89.03,44.67) and (89.33,44.97) .. (89.33,45.33) .. controls (89.33,45.7) and (89.03,46) .. (88.67,46) .. controls (88.3,46) and (88,45.7) .. (88,45.33) -- cycle ;
\draw  [dash pattern={on 0.84pt off 2.51pt}]  (94,45.67) -- (129.5,45.67) ;
\draw [shift={(132.5,45.67)}, rotate = 180] [fill={rgb, 255:red, 0; green, 0; blue, 0 }  ][line width=0.08]  [draw opacity=0] (5.36,-2.57) -- (0,0) -- (5.36,2.57) -- (3.56,0) -- cycle    ;
\draw   (137,45.33) .. controls (137,44.97) and (137.3,44.67) .. (137.67,44.67) .. controls (138.03,44.67) and (138.33,44.97) .. (138.33,45.33) .. controls (138.33,45.7) and (138.03,46) .. (137.67,46) .. controls (137.3,46) and (137,45.7) .. (137,45.33) -- cycle ;
\draw    (219.08,19) -- (227.08,19) ;
\draw [shift={(219.08,19)}, rotate = 0] [color={rgb, 255:red, 0; green, 0; blue, 0 }  ][fill={rgb, 255:red, 0; green, 0; blue, 0 }  ][line width=0.75]      (0, 0) circle [x radius= 1.34, y radius= 1.34]   ;
\draw    (218.75,58.5) -- (226.58,58.5) ;
\draw [shift={(218.75,58.5)}, rotate = 0] [color={rgb, 255:red, 0; green, 0; blue, 0 }  ][fill={rgb, 255:red, 0; green, 0; blue, 0 }  ][line width=0.75]      (0, 0) circle [x radius= 1.34, y radius= 1.34]   ;
\draw   (153.08,32.17) .. controls (153.08,31.8) and (153.37,31.5) .. (153.74,31.5) .. controls (154.11,31.5) and (154.41,31.8) .. (154.41,32.17) .. controls (154.41,32.53) and (154.11,32.83) .. (153.74,32.83) .. controls (153.37,32.83) and (153.08,32.53) .. (153.08,32.17) -- cycle ;
\draw  [dash pattern={on 0.84pt off 2.51pt}]  (93.76,44) -- (145.82,33.59) ;
\draw [shift={(148.76,33)}, rotate = 168.69] [fill={rgb, 255:red, 0; green, 0; blue, 0 }  ][line width=0.08]  [draw opacity=0] (5.36,-2.57) -- (0,0) -- (5.36,2.57) -- (3.56,0) -- cycle    ;

\draw (22.67,56.4) node [anchor=north west][inner sep=0.75pt]  [font=\tiny]  {$\partial ^{\prime }$};
\draw (23.67,16.4) node [anchor=north west][inner sep=0.75pt]  [font=\tiny]  {$\partial $};
\draw (226,34.2) node [anchor=north west][inner sep=0.75pt]    {$\cdots $};
\draw (6.8,33.4) node [anchor=north west][inner sep=0.75pt]    {$\cdots $};
\draw (110.2,64.9) node [anchor=north west][inner sep=0.75pt]  [font=\tiny]  {$( 0,0)$};
\draw (109.3,6.5) node [anchor=north west][inner sep=0.75pt]  [font=\tiny]  {$( 0,1)$};
\draw (62.5,41.9) node [anchor=north west][inner sep=0.75pt]  [font=\tiny]  {$( x,y)$};
\draw (115,28.57) node [anchor=north west][inner sep=0.75pt]  [font=\tiny]  {$\theta $};
\draw (107,46.9) node [anchor=north west][inner sep=0.75pt]  [font=\tiny]  {$\sigma _{3}$};
\draw (140,41.07) node [anchor=north west][inner sep=0.75pt]  [font=\tiny]  {$( x+3,y)$};
\draw (155.08,26.9) node [anchor=north west][inner sep=0.75pt]  [font=\tiny]  {$( -x,1-y)$};
\end{tikzpicture}
\caption{$\sigma_n$ and $\theta$ on the strip $\widetilde{\mathcal{S}}$ with $n=3$}
\end{figure}

We denote by $[I,J]$, where $I, J\in \widetilde{\mathcal{S}}$, the line segment in $\widetilde{\mathcal{S}}$ with endpoints $I$ and $J$. By abuse of notation, we also use $[i,j]$ to represent the line segment with endpoints $(i,0)$ and $(j,1)$. Let $\operatorname{Seg}(M)$ be the collection of line segments $\{[i,j]\;|\; i,j\in \mathbb{Z}\}$.  Regard the line segments in $\operatorname{Seg}(M)$ as sets of points. The maps $\sigma_n$ and $\theta$ naturally induce two bijections on $\operatorname{Seg}(M)$, also denoted by $\sigma_n$ and $\theta$. Precisely, 
\begin{align*}
    \sigma_n: & \operatorname{Seg}(M)\rightarrow\operatorname{Seg}(M) && [i,j]\mapsto [i+n, j+n]; \\
    \theta: & \operatorname{Seg}(M)\rightarrow\operatorname{Seg}(M) && [i,j]\mapsto [-j,-i].
\end{align*}

 Let $G$ be the group  generated by $\sigma_n$ and $\theta$. Naturally, there is a $G$-action on $\widetilde{\mathcal{S}}$. Denote by $\widetilde{\operatorname{Seg}(M)}$ the set of $G$-orbits of segments in $\operatorname{Seg}(M)$. We have the following proposition.
 \begin{prop}{}{}\label{correspondence1}
     There exists a bijection
    \begin{align*}
\phi: \widetilde{\operatorname{Seg}(M)} & \rightarrow \operatorname{ind}({\rm vect}^F\mbox{-}\mathbb{X}) \\
\widetilde{[i,j]} & \mapsto \mathsf{E}_{L_0(-i\vec{x}_3)}\langle (i+j-1)\vec{x}_3 \rangle.
\end{align*}
\end{prop}
\begin{proof}
Firstly, we claim that $\phi$ is well-defined. Assume $[i,j]\sim [s,t]$ in $\operatorname{Seg}(M)$. Then there exists an element $g$ in $G$ such that $g\cdot[i,j]=[s,t]$. Through straightforward verification, we know that $\theta^2=1_G$ and $\sigma_n^k\theta=\theta\sigma_n^{-k}$, for any $k\in \mathbb{Z}$. Therefore, we can set $g$ as $\sigma_n^k$ or $\sigma_n^k\theta$ for some $k\in \mathbb{Z}$. If $g=\sigma_n^k$, then
\begin{align*}
    \phi(\widetilde{[s,t]}) &= \phi(\widetilde{g\cdot[i,j]}) = \phi(\widetilde{[i+kn,j+kn]})\\
    &= \mathsf{E}_{L_0(-i\vec{x}_3-k\vec{c})}\langle (i+j+2kn-1)\vec{x}_3\rangle \\
    &= \mathsf{E}_{L_0(-i\vec{x}_3)}\langle (i+j-1)\vec{x}_3 \rangle \\
    &= \phi(\widetilde{[i,j]}).
\end{align*}
If $g=\sigma_n^k\theta$, then
\begin{align*}
    \phi(\widetilde{[s,t]}) &= \phi( \widetilde{g\cdot[i,j]}) = \phi(\widetilde{[-j+kn,-i+kn]}) \\
&= \mathsf{E}_{L_0(j\vec{x}_3-k\vec{c})}\langle (-i-j+2kn-1)\vec{x}_3\rangle \\
&= \mathsf{E}_{L_0(j\vec{x}_3)}\langle (-i-j-1)\vec{x}_3 \rangle \\
&=\mathsf{E}_{L_0(j\vec{x}_3)((-i-j)\vec{x}_3-\vec{x}_1)}\langle \vec{\delta}+(i+j+1)\vec{x}_3 \rangle\\
&= \mathsf{E}_{L_0(-i\vec{x}_3)}\langle (i+j-1)\vec{x}_3 \rangle \\
&= \phi(\widetilde{[i,j]}).
\end{align*}
Therefore, $\phi$ is well-defined. 

Clearly, $\phi$ is surjective. It remains to show that $\phi$ is injective. Take an element $\widetilde{[i,j]}$ in $\widetilde{\operatorname{Seg}(M)}$, we can always assume that $i,j$ satisfy $-n<i+j\leq n$. This is because, given $i+j\in \mathbb{Z}$ and $\mathbb{R}=\cup_{k\in \mathbb{Z}}((2k-1)n,(2k+1)n]$, there exists a unique $k\in \mathbb{Z}$ such that $(2k-1)n<i+j\leq (2k+1)n$. Let $[i^\prime,j^\prime]=\sigma_n^{-k}\cdot[i,j]$, then $i^\prime,j^\prime$ satisfy $-n<i^\prime+j^\prime\leq n$ and $\widetilde{[i,j]}=\widetilde{[i^\prime,j^\prime]}$.  Now, suppose $\widetilde{[i,j]}$ and $\widetilde{[s,t]}$ are elements in $\widetilde{\operatorname{Seg}(M)}$, satisfying $\phi(\widetilde{[i,j]})=\phi(\widetilde{[s,t]})$, i.e.,  \[\mathsf{E}_{L_0(-i\vec{x}_3)}\langle (i+j-1)\vec{x}_3 \rangle=\mathsf{E}_{L_0(-s\vec{x}_3)}\langle (s+t-1)\vec{x}_3 \rangle.\] 
The preceding discussion allows us to assume that $-n<i+j, s+t\leq n$. Then by Remark \ref{rem1}, $i-s=0$ and $i+j-1=s+t-1$. Consequently, $\widetilde{[i,j]}=\widetilde{[s,t]}$. Therefore, $\phi$ is a bijection.
\end{proof}

\begin{rem} 
 Let $(-i,0)$  be the marked point on the lower boundary $\partial^\prime$ of $\widetilde{\mathcal{S}}$. Consider the line segments in $\operatorname{Seg}(M)$ with $(-i,0)$ as one endpoint. 
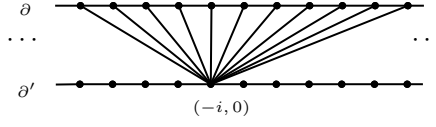
\begin{figure}[H]
    \centering
 \tikzset{every picture/.style={line width=0.75pt}}          
 \begin{tikzpicture}[x=0.75pt,y=0.75pt,yscale=-1,xscale=1]
 \draw    (26.5,53.5) -- (39.17,53.5) ;
\draw [shift={(39.17,53.5)}, rotate = 360] [color={rgb, 255:red, 0; green, 0; blue, 0 }  ][fill={rgb, 255:red, 0; green, 0; blue, 0 }  ][line width=0.75]      (0, 0) circle [x radius= 1.34, y radius= 1.34]   ;
  \draw    (39.17,53.5) -- (55.52,53.5) ;
\draw [shift={(55.52,53.5)}, rotate = 0] [color={rgb, 255:red, 0; green, 0; blue, 0 }  ][fill={rgb, 255:red, 0; green, 0; blue, 0 }  ][line width=0.75]      (0, 0) circle [x radius= 1.34, y radius= 1.34]   ;
\draw [shift={(39.17,53.5)}, rotate = 0] [color={rgb, 255:red, 0; green, 0; blue, 0 }  ][fill={rgb, 255:red, 0; green, 0; blue, 0 }  ][line width=0.75]      (0, 0) circle [x radius= 1.34, y radius= 1.34]   ;
  \draw    (55.52,53.5) -- (71.86,53.5) ;
\draw [shift={(71.86,53.5)}, rotate = 0] [color={rgb, 255:red, 0; green, 0; blue, 0 }  ][fill={rgb, 255:red, 0; green, 0; blue, 0 }  ][line width=0.75]      (0, 0) circle [x radius= 1.34, y radius= 1.34]   ;
\draw [shift={(55.52,53.5)}, rotate = 0] [color={rgb, 255:red, 0; green, 0; blue, 0 }  ][fill={rgb, 255:red, 0; green, 0; blue, 0 }  ][line width=0.75]      (0, 0) circle [x radius= 1.34, y radius= 1.34]   ;
  \draw    (71.86,53.5) -- (88.2,53.5) ;
\draw [shift={(88.2,53.5)}, rotate = 0] [color={rgb, 255:red, 0; green, 0; blue, 0 }  ][fill={rgb, 255:red, 0; green, 0; blue, 0 }  ][line width=0.75]      (0, 0) circle [x radius= 1.34, y radius= 1.34]   ;
\draw [shift={(71.86,53.5)}, rotate = 0] [color={rgb, 255:red, 0; green, 0; blue, 0 }  ][fill={rgb, 255:red, 0; green, 0; blue, 0 }  ][line width=0.75]      (0, 0) circle [x radius= 1.34, y radius= 1.34]   ;
  \draw    (88.2,53.5) -- (104.54,53.5) ;
\draw [shift={(104.54,53.5)}, rotate = 0] [color={rgb, 255:red, 0; green, 0; blue, 0 }  ][fill={rgb, 255:red, 0; green, 0; blue, 0 }  ][line width=0.75]      (0, 0) circle [x radius= 1.34, y radius= 1.34]   ;
\draw [shift={(88.2,53.5)}, rotate = 0] [color={rgb, 255:red, 0; green, 0; blue, 0 }  ][fill={rgb, 255:red, 0; green, 0; blue, 0 }  ][line width=0.75]      (0, 0) circle [x radius= 1.34, y radius= 1.34]   ;
  \draw    (104.54,53.5) -- (120.88,53.5) ;
\draw [shift={(120.88,53.5)}, rotate = 0] [color={rgb, 255:red, 0; green, 0; blue, 0 }  ][fill={rgb, 255:red, 0; green, 0; blue, 0 }  ][line width=0.75]      (0, 0) circle [x radius= 1.34, y radius= 1.34]   ;
\draw [shift={(104.54,53.5)}, rotate = 0] [color={rgb, 255:red, 0; green, 0; blue, 0 }  ][fill={rgb, 255:red, 0; green, 0; blue, 0 }  ][line width=0.75]      (0, 0) circle [x radius= 1.34, y radius= 1.34]   ;
  \draw    (120.88,53.5) -- (137.22,53.5) ;
\draw [shift={(137.22,53.5)}, rotate = 0] [color={rgb, 255:red, 0; green, 0; blue, 0 }  ][fill={rgb, 255:red, 0; green, 0; blue, 0 }  ][line width=0.75]      (0, 0) circle [x radius= 1.34, y radius= 1.34]   ;
\draw [shift={(120.88,53.5)}, rotate = 0] [color={rgb, 255:red, 0; green, 0; blue, 0 }  ][fill={rgb, 255:red, 0; green, 0; blue, 0 }  ][line width=0.75]      (0, 0) circle [x radius= 1.34, y radius= 1.34]   ;
  \draw    (137.22,53.5) -- (153.56,53.5) ;
\draw [shift={(153.56,53.5)}, rotate = 0] [color={rgb, 255:red, 0; green, 0; blue, 0 }  ][fill={rgb, 255:red, 0; green, 0; blue, 0 }  ][line width=0.75]      (0, 0) circle [x radius= 1.34, y radius= 1.34]   ;
\draw [shift={(137.22,53.5)}, rotate = 0] [color={rgb, 255:red, 0; green, 0; blue, 0 }  ][fill={rgb, 255:red, 0; green, 0; blue, 0 }  ][line width=0.75]      (0, 0) circle [x radius= 1.34, y radius= 1.34]   ;
  \draw    (153.56,53.5) -- (169.9,53.5) ;
\draw [shift={(169.9,53.5)}, rotate = 0] [color={rgb, 255:red, 0; green, 0; blue, 0 }  ][fill={rgb, 255:red, 0; green, 0; blue, 0 }  ][line width=0.75]      (0, 0) circle [x radius= 1.34, y radius= 1.34]   ;
\draw [shift={(153.56,53.5)}, rotate = 0] [color={rgb, 255:red, 0; green, 0; blue, 0 }  ][fill={rgb, 255:red, 0; green, 0; blue, 0 }  ][line width=0.75]      (0, 0) circle [x radius= 1.34, y radius= 1.34]   ;
  \draw    (169.9,53.5) -- (186.24,53.5) ;
\draw [shift={(186.24,53.5)}, rotate = 0] [color={rgb, 255:red, 0; green, 0; blue, 0 }  ][fill={rgb, 255:red, 0; green, 0; blue, 0 }  ][line width=0.75]      (0, 0) circle [x radius= 1.34, y radius= 1.34]   ;
\draw [shift={(169.9,53.5)}, rotate = 0] [color={rgb, 255:red, 0; green, 0; blue, 0 }  ][fill={rgb, 255:red, 0; green, 0; blue, 0 }  ][line width=0.75]      (0, 0) circle [x radius= 1.34, y radius= 1.34]   ;
  \draw    (186.24,53.5) -- (202.58,53.5) ;
\draw [shift={(202.58,53.5)}, rotate = 0] [color={rgb, 255:red, 0; green, 0; blue, 0 }  ][fill={rgb, 255:red, 0; green, 0; blue, 0 }  ][line width=0.75]      (0, 0) circle [x radius= 1.34, y radius= 1.34]   ;
\draw [shift={(186.24,53.5)}, rotate = 0] [color={rgb, 255:red, 0; green, 0; blue, 0 }  ][fill={rgb, 255:red, 0; green, 0; blue, 0 }  ][line width=0.75]      (0, 0) circle [x radius= 1.34, y radius= 1.34]   ;
  \draw    (26.83,93.17) -- (38.84,93) ;
\draw [shift={(38.84,93)}, rotate = 359.2] [color={rgb, 255:red, 0; green, 0; blue, 0 }  ][fill={rgb, 255:red, 0; green, 0; blue, 0 }  ][line width=0.75]      (0, 0) circle [x radius= 1.34, y radius= 1.34]   ;
  \draw    (38.84,93) -- (55.18,93) ;
\draw [shift={(55.18,93)}, rotate = 0] [color={rgb, 255:red, 0; green, 0; blue, 0 }  ][fill={rgb, 255:red, 0; green, 0; blue, 0 }  ][line width=0.75]      (0, 0) circle [x radius= 1.34, y radius= 1.34]   ;
\draw [shift={(38.84,93)}, rotate = 0] [color={rgb, 255:red, 0; green, 0; blue, 0 }  ][fill={rgb, 255:red, 0; green, 0; blue, 0 }  ][line width=0.75]      (0, 0) circle [x radius= 1.34, y radius= 1.34]   ;
  \draw    (55.18,93) -- (71.52,93) ;
\draw [shift={(71.52,93)}, rotate = 0] [color={rgb, 255:red, 0; green, 0; blue, 0 }  ][fill={rgb, 255:red, 0; green, 0; blue, 0 }  ][line width=0.75]      (0, 0) circle [x radius= 1.34, y radius= 1.34]   ;
\draw [shift={(55.18,93)}, rotate = 0] [color={rgb, 255:red, 0; green, 0; blue, 0 }  ][fill={rgb, 255:red, 0; green, 0; blue, 0 }  ][line width=0.75]      (0, 0) circle [x radius= 1.34, y radius= 1.34]   ;
  \draw    (71.52,93) -- (87.86,93) ;
\draw [shift={(87.86,93)}, rotate = 0] [color={rgb, 255:red, 0; green, 0; blue, 0 }  ][fill={rgb, 255:red, 0; green, 0; blue, 0 }  ][line width=0.75]      (0, 0) circle [x radius= 1.34, y radius= 1.34]   ;
\draw [shift={(71.52,93)}, rotate = 0] [color={rgb, 255:red, 0; green, 0; blue, 0 }  ][fill={rgb, 255:red, 0; green, 0; blue, 0 }  ][line width=0.75]      (0, 0) circle [x radius= 1.34, y radius= 1.34]   ;
  \draw    (87.86,93) -- (104.2,93) ;
\draw [shift={(104.2,93)}, rotate = 0] [color={rgb, 255:red, 0; green, 0; blue, 0 }  ][fill={rgb, 255:red, 0; green, 0; blue, 0 }  ][line width=0.75]      (0, 0) circle [x radius= 1.34, y radius= 1.34]   ;
\draw [shift={(87.86,93)}, rotate = 0] [color={rgb, 255:red, 0; green, 0; blue, 0 }  ][fill={rgb, 255:red, 0; green, 0; blue, 0 }  ][line width=0.75]      (0, 0) circle [x radius= 1.34, y radius= 1.34]   ;
  \draw    (104.2,93) -- (120.55,93) ;
\draw [shift={(120.55,93)}, rotate = 0] [color={rgb, 255:red, 0; green, 0; blue, 0 }  ][fill={rgb, 255:red, 0; green, 0; blue, 0 }  ][line width=0.75]      (0, 0) circle [x radius= 1.34, y radius= 1.34]   ;
\draw [shift={(104.2,93)}, rotate = 0] [color={rgb, 255:red, 0; green, 0; blue, 0 }  ][fill={rgb, 255:red, 0; green, 0; blue, 0 }  ][line width=0.75]      (0, 0) circle [x radius= 1.34, y radius= 1.34]   ;
  \draw    (120.55,93) -- (136.89,93) ;
\draw [shift={(136.89,93)}, rotate = 0] [color={rgb, 255:red, 0; green, 0; blue, 0 }  ][fill={rgb, 255:red, 0; green, 0; blue, 0 }  ][line width=0.75]      (0, 0) circle [x radius= 1.34, y radius= 1.34]   ;
\draw [shift={(120.55,93)}, rotate = 0] [color={rgb, 255:red, 0; green, 0; blue, 0 }  ][fill={rgb, 255:red, 0; green, 0; blue, 0 }  ][line width=0.75]      (0, 0) circle [x radius= 1.34, y radius= 1.34]   ;
  \draw    (136.89,93) -- (153.23,93) ;
\draw [shift={(153.23,93)}, rotate = 0] [color={rgb, 255:red, 0; green, 0; blue, 0 }  ][fill={rgb, 255:red, 0; green, 0; blue, 0 }  ][line width=0.75]      (0, 0) circle [x radius= 1.34, y radius= 1.34]   ;
\draw [shift={(136.89,93)}, rotate = 0] [color={rgb, 255:red, 0; green, 0; blue, 0 }  ][fill={rgb, 255:red, 0; green, 0; blue, 0 }  ][line width=0.75]      (0, 0) circle [x radius= 1.34, y radius= 1.34]   ;
  \draw    (153.23,93) -- (169.57,93) ;
\draw [shift={(169.57,93)}, rotate = 0] [color={rgb, 255:red, 0; green, 0; blue, 0 }  ][fill={rgb, 255:red, 0; green, 0; blue, 0 }  ][line width=0.75]      (0, 0) circle [x radius= 1.34, y radius= 1.34]   ;
\draw [shift={(153.23,93)}, rotate = 0] [color={rgb, 255:red, 0; green, 0; blue, 0 }  ][fill={rgb, 255:red, 0; green, 0; blue, 0 }  ][line width=0.75]      (0, 0) circle [x radius= 1.34, y radius= 1.34]   ;
  \draw    (169.57,93) -- (185.91,93) ;
\draw [shift={(185.91,93)}, rotate = 0] [color={rgb, 255:red, 0; green, 0; blue, 0 }  ][fill={rgb, 255:red, 0; green, 0; blue, 0 }  ][line width=0.75]      (0, 0) circle [x radius= 1.34, y radius= 1.34]   ;
\draw [shift={(169.57,93)}, rotate = 0] [color={rgb, 255:red, 0; green, 0; blue, 0 }  ][fill={rgb, 255:red, 0; green, 0; blue, 0 }  ][line width=0.75]      (0, 0) circle [x radius= 1.34, y radius= 1.34]   ;
  \draw    (185.91,93) -- (202.25,93) ;
\draw [shift={(202.25,93)}, rotate = 0] [color={rgb, 255:red, 0; green, 0; blue, 0 }  ][fill={rgb, 255:red, 0; green, 0; blue, 0 }  ][line width=0.75]      (0, 0) circle [x radius= 1.34, y radius= 1.34]   ;
\draw [shift={(185.91,93)}, rotate = 0] [color={rgb, 255:red, 0; green, 0; blue, 0 }  ][fill={rgb, 255:red, 0; green, 0; blue, 0 }  ][line width=0.75]      (0, 0) circle [x radius= 1.34, y radius= 1.34]   ;
  \draw    (202.58,53.5) -- (210.58,53.5) ;
\draw [shift={(202.58,53.5)}, rotate = 0] [color={rgb, 255:red, 0; green, 0; blue, 0 }  ][fill={rgb, 255:red, 0; green, 0; blue, 0 }  ][line width=0.75]      (0, 0) circle [x radius= 1.34, y radius= 1.34]   ;
  \draw    (202.25,93) -- (210.08,93) ;
\draw [shift={(202.25,93)}, rotate = 0] [color={rgb, 255:red, 0; green, 0; blue, 0 }  ][fill={rgb, 255:red, 0; green, 0; blue, 0 }  ][line width=0.75]      (0, 0) circle [x radius= 1.34, y radius= 1.34]   ;
  
\draw [color={rgb, 255:red, 0; green, 0; blue, 0 }  ,draw opacity=1 ]   (104.2,93) -- (120.88,53.5) ;
  \draw    (104.2,93) -- (186.24,53.5) ;
  \draw    (104.2,93) -- (202.58,53.5) ;
  \draw    (104.2,93) -- (169.9,53.5) ;
  \draw    (104.2,93) -- (153.56,53.5) ;
  \draw    (104.2,93) -- (137.22,53.5) ;
  \draw    (104.2,93) -- (104.54,53.5) ;
  \draw    (104.2,93) -- (39.17,53.5) ;
  \draw    (104.2,93) -- (55.52,53.5) ;
  \draw    (104.2,93) -- (71.86,53.5) ;
  \draw    (104.2,93) -- (88.2,53.5) ;
  \draw (6.17,90.9) node [anchor=north west][inner sep=0.75pt]  [font=\tiny]  {$\partial ^{\prime }$};
 \draw (7.17,50.9) node [anchor=north west][inner sep=0.75pt]  [font=\tiny]  {$\partial $};
 \draw (203,66.2) node [anchor=north west][inner sep=0.75pt]    {$\cdots $};
 \draw (0.8,66.4) node [anchor=north west][inner sep=0.75pt]    {$\cdots $};
 \draw (93.7,99.4) node [anchor=north west][inner sep=0.75pt]  [font=\tiny]  {$( -i,0)$};
\end{tikzpicture}
\end{figure}
\noindent It is easy to verify that their $G$-orbits are distinct. By Proposition \ref{correspondence1}, we have the correspondence \[\widetilde{[-i,k+i+1]}\mapsto \phi(\widetilde{[-i,k+i+1]})= \mathsf{E}_{L_0(i\vec{x}_3)}\langle k\vec{x}_3\rangle,\] where $k\in\mathbb{Z}$. Indeed, the generalized extension bundles in the set \[\{ \mathsf{E}_{L_0(i\vec{x}_3)}\langle k\vec{x}_3\rangle \mid k\in \mathbb{Z}\}\] lie on a  polyline in the valued quiver $(\Gamma ,d)$, where the polyline is obtained by reflecting a line with respect to the upper and lower horizontal lines $l_1$ and $l_2$, as illustrated in the figure below:

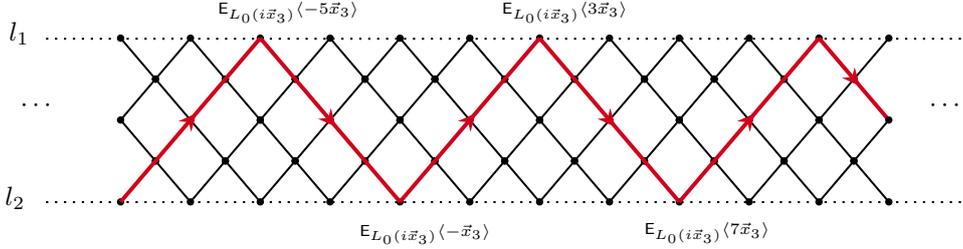
\begin{figure}[H]
    \centering
\tikzset{every picture/.style={line width=0.75pt}}          
\begin{tikzpicture}[x=0.75pt,y=0.75pt,yscale=-1,xscale=1]
 \draw    (115.1,152.15) -- (149.96,193.33) ;
  \draw    (115.1,193.33) -- (184.82,110.98) ;
  \draw    (115.1,193.33) -- (132.53,172.74) ;
\draw [shift={(132.53,172.74)}, rotate = 310.25] [color={rgb, 255:red, 0; green, 0; blue, 0 }  ][fill={rgb, 255:red, 0; green, 0; blue, 0 }  ][line width=0.75]      (0, 0) circle [x radius= 1.34, y radius= 1.34]   ;
\draw [shift={(115.1,193.33)}, rotate = 310.25] [color={rgb, 255:red, 0; green, 0; blue, 0 }  ][fill={rgb, 255:red, 0; green, 0; blue, 0 }  ][line width=0.75]      (0, 0) circle [x radius= 1.34, y radius= 1.34]   ;
  \draw    (132.53,172.74) -- (149.96,152.15) ;
\draw [shift={(149.96,152.15)}, rotate = 310.25] [color={rgb, 255:red, 0; green, 0; blue, 0 }  ][fill={rgb, 255:red, 0; green, 0; blue, 0 }  ][line width=0.75]      (0, 0) circle [x radius= 1.34, y radius= 1.34]   ;
\draw [shift={(132.53,172.74)}, rotate = 310.25] [color={rgb, 255:red, 0; green, 0; blue, 0 }  ][fill={rgb, 255:red, 0; green, 0; blue, 0 }  ][line width=0.75]      (0, 0) circle [x radius= 1.34, y radius= 1.34]   ;
  \draw    (149.96,152.15) -- (167.39,131.57) ;
\draw [shift={(167.39,131.57)}, rotate = 310.25] [color={rgb, 255:red, 0; green, 0; blue, 0 }  ][fill={rgb, 255:red, 0; green, 0; blue, 0 }  ][line width=0.75]      (0, 0) circle [x radius= 1.34, y radius= 1.34]   ;
\draw [shift={(149.96,152.15)}, rotate = 310.25] [color={rgb, 255:red, 0; green, 0; blue, 0 }  ][fill={rgb, 255:red, 0; green, 0; blue, 0 }  ][line width=0.75]      (0, 0) circle [x radius= 1.34, y radius= 1.34]   ;
  \draw    (167.39,131.57) -- (184.82,110.98) ;
\draw [shift={(184.82,110.98)}, rotate = 310.25] [color={rgb, 255:red, 0; green, 0; blue, 0 }  ][fill={rgb, 255:red, 0; green, 0; blue, 0 }  ][line width=0.75]      (0, 0) circle [x radius= 1.34, y radius= 1.34]   ;
\draw [shift={(167.39,131.57)}, rotate = 310.25] [color={rgb, 255:red, 0; green, 0; blue, 0 }  ][fill={rgb, 255:red, 0; green, 0; blue, 0 }  ][line width=0.75]      (0, 0) circle [x radius= 1.34, y radius= 1.34]   ;
  \draw    (184.82,193.33) -- (202.25,172.74) ;
\draw [shift={(202.25,172.74)}, rotate = 310.25] [color={rgb, 255:red, 0; green, 0; blue, 0 }  ][fill={rgb, 255:red, 0; green, 0; blue, 0 }  ][line width=0.75]      (0, 0) circle [x radius= 1.34, y radius= 1.34]   ;
\draw [shift={(184.82,193.33)}, rotate = 310.25] [color={rgb, 255:red, 0; green, 0; blue, 0 }  ][fill={rgb, 255:red, 0; green, 0; blue, 0 }  ][line width=0.75]      (0, 0) circle [x radius= 1.34, y radius= 1.34]   ;
  \draw    (202.25,172.74) -- (219.68,152.15) ;
\draw [shift={(219.68,152.15)}, rotate = 310.25] [color={rgb, 255:red, 0; green, 0; blue, 0 }  ][fill={rgb, 255:red, 0; green, 0; blue, 0 }  ][line width=0.75]      (0, 0) circle [x radius= 1.34, y radius= 1.34]   ;
\draw [shift={(202.25,172.74)}, rotate = 310.25] [color={rgb, 255:red, 0; green, 0; blue, 0 }  ][fill={rgb, 255:red, 0; green, 0; blue, 0 }  ][line width=0.75]      (0, 0) circle [x radius= 1.34, y radius= 1.34]   ;
  \draw    (219.68,152.15) -- (237.11,131.57) ;
\draw [shift={(237.11,131.57)}, rotate = 310.25] [color={rgb, 255:red, 0; green, 0; blue, 0 }  ][fill={rgb, 255:red, 0; green, 0; blue, 0 }  ][line width=0.75]      (0, 0) circle [x radius= 1.34, y radius= 1.34]   ;
\draw [shift={(219.68,152.15)}, rotate = 310.25] [color={rgb, 255:red, 0; green, 0; blue, 0 }  ][fill={rgb, 255:red, 0; green, 0; blue, 0 }  ][line width=0.75]      (0, 0) circle [x radius= 1.34, y radius= 1.34]   ;
  \draw    (237.11,131.57) -- (254.55,110.98) ;
\draw [shift={(254.55,110.98)}, rotate = 310.25] [color={rgb, 255:red, 0; green, 0; blue, 0 }  ][fill={rgb, 255:red, 0; green, 0; blue, 0 }  ][line width=0.75]      (0, 0) circle [x radius= 1.34, y radius= 1.34]   ;
\draw [shift={(237.11,131.57)}, rotate = 310.25] [color={rgb, 255:red, 0; green, 0; blue, 0 }  ][fill={rgb, 255:red, 0; green, 0; blue, 0 }  ][line width=0.75]      (0, 0) circle [x radius= 1.34, y radius= 1.34]   ;
  \draw    (428.85,193.33) -- (446.28,172.74) ;
\draw [shift={(446.28,172.74)}, rotate = 310.25] [color={rgb, 255:red, 0; green, 0; blue, 0 }  ][fill={rgb, 255:red, 0; green, 0; blue, 0 }  ][line width=0.75]      (0, 0) circle [x radius= 1.34, y radius= 1.34]   ;
\draw [shift={(428.85,193.33)}, rotate = 310.25] [color={rgb, 255:red, 0; green, 0; blue, 0 }  ][fill={rgb, 255:red, 0; green, 0; blue, 0 }  ][line width=0.75]      (0, 0) circle [x radius= 1.34, y radius= 1.34]   ;
  \draw    (446.28,172.74) -- (463.71,152.15) ;
\draw [shift={(463.71,152.15)}, rotate = 310.25] [color={rgb, 255:red, 0; green, 0; blue, 0 }  ][fill={rgb, 255:red, 0; green, 0; blue, 0 }  ][line width=0.75]      (0, 0) circle [x radius= 1.34, y radius= 1.34]   ;
\draw [shift={(446.28,172.74)}, rotate = 310.25] [color={rgb, 255:red, 0; green, 0; blue, 0 }  ][fill={rgb, 255:red, 0; green, 0; blue, 0 }  ][line width=0.75]      (0, 0) circle [x radius= 1.34, y radius= 1.34]   ;
  \draw    (463.71,152.15) -- (481.14,131.57) ;
\draw [shift={(481.14,131.57)}, rotate = 310.25] [color={rgb, 255:red, 0; green, 0; blue, 0 }  ][fill={rgb, 255:red, 0; green, 0; blue, 0 }  ][line width=0.75]      (0, 0) circle [x radius= 1.34, y radius= 1.34]   ;
\draw [shift={(463.71,152.15)}, rotate = 310.25] [color={rgb, 255:red, 0; green, 0; blue, 0 }  ][fill={rgb, 255:red, 0; green, 0; blue, 0 }  ][line width=0.75]      (0, 0) circle [x radius= 1.34, y radius= 1.34]   ;
  \draw    (481.14,131.57) -- (498.57,110.98) ;
\draw [shift={(498.57,110.98)}, rotate = 310.25] [color={rgb, 255:red, 0; green, 0; blue, 0 }  ][fill={rgb, 255:red, 0; green, 0; blue, 0 }  ][line width=0.75]      (0, 0) circle [x radius= 1.34, y radius= 1.34]   ;
\draw [shift={(481.14,131.57)}, rotate = 310.25] [color={rgb, 255:red, 0; green, 0; blue, 0 }  ][fill={rgb, 255:red, 0; green, 0; blue, 0 }  ][line width=0.75]      (0, 0) circle [x radius= 1.34, y radius= 1.34]   ;
  \draw    (149.96,193.33) -- (167.39,172.74) ;
\draw [shift={(167.39,172.74)}, rotate = 310.25] [color={rgb, 255:red, 0; green, 0; blue, 0 }  ][fill={rgb, 255:red, 0; green, 0; blue, 0 }  ][line width=0.75]      (0, 0) circle [x radius= 1.34, y radius= 1.34]   ;
\draw [shift={(149.96,193.33)}, rotate = 310.25] [color={rgb, 255:red, 0; green, 0; blue, 0 }  ][fill={rgb, 255:red, 0; green, 0; blue, 0 }  ][line width=0.75]      (0, 0) circle [x radius= 1.34, y radius= 1.34]   ;
  \draw    (167.39,172.74) -- (184.82,152.15) ;
\draw [shift={(184.82,152.15)}, rotate = 310.25] [color={rgb, 255:red, 0; green, 0; blue, 0 }  ][fill={rgb, 255:red, 0; green, 0; blue, 0 }  ][line width=0.75]      (0, 0) circle [x radius= 1.34, y radius= 1.34]   ;
\draw [shift={(167.39,172.74)}, rotate = 310.25] [color={rgb, 255:red, 0; green, 0; blue, 0 }  ][fill={rgb, 255:red, 0; green, 0; blue, 0 }  ][line width=0.75]      (0, 0) circle [x radius= 1.34, y radius= 1.34]   ;
  \draw    (184.82,152.15) -- (202.25,131.57) ;
\draw [shift={(202.25,131.57)}, rotate = 310.25] [color={rgb, 255:red, 0; green, 0; blue, 0 }  ][fill={rgb, 255:red, 0; green, 0; blue, 0 }  ][line width=0.75]      (0, 0) circle [x radius= 1.34, y radius= 1.34]   ;
\draw [shift={(184.82,152.15)}, rotate = 310.25] [color={rgb, 255:red, 0; green, 0; blue, 0 }  ][fill={rgb, 255:red, 0; green, 0; blue, 0 }  ][line width=0.75]      (0, 0) circle [x radius= 1.34, y radius= 1.34]   ;
  \draw    (202.25,131.57) -- (219.68,110.98) ;
\draw [shift={(219.68,110.98)}, rotate = 310.25] [color={rgb, 255:red, 0; green, 0; blue, 0 }  ][fill={rgb, 255:red, 0; green, 0; blue, 0 }  ][line width=0.75]      (0, 0) circle [x radius= 1.34, y radius= 1.34]   ;
\draw [shift={(202.25,131.57)}, rotate = 310.25] [color={rgb, 255:red, 0; green, 0; blue, 0 }  ][fill={rgb, 255:red, 0; green, 0; blue, 0 }  ][line width=0.75]      (0, 0) circle [x radius= 1.34, y radius= 1.34]   ;
  \draw    (393.99,193.33) -- (411.42,172.74) ;
\draw [shift={(411.42,172.74)}, rotate = 310.25] [color={rgb, 255:red, 0; green, 0; blue, 0 }  ][fill={rgb, 255:red, 0; green, 0; blue, 0 }  ][line width=0.75]      (0, 0) circle [x radius= 1.34, y radius= 1.34]   ;
\draw [shift={(393.99,193.33)}, rotate = 310.25] [color={rgb, 255:red, 0; green, 0; blue, 0 }  ][fill={rgb, 255:red, 0; green, 0; blue, 0 }  ][line width=0.75]      (0, 0) circle [x radius= 1.34, y radius= 1.34]   ;
  \draw    (411.42,172.74) -- (428.85,152.15) ;
\draw [shift={(428.85,152.15)}, rotate = 310.25] [color={rgb, 255:red, 0; green, 0; blue, 0 }  ][fill={rgb, 255:red, 0; green, 0; blue, 0 }  ][line width=0.75]      (0, 0) circle [x radius= 1.34, y radius= 1.34]   ;
\draw [shift={(411.42,172.74)}, rotate = 310.25] [color={rgb, 255:red, 0; green, 0; blue, 0 }  ][fill={rgb, 255:red, 0; green, 0; blue, 0 }  ][line width=0.75]      (0, 0) circle [x radius= 1.34, y radius= 1.34]   ;
  \draw    (428.85,152.15) -- (446.28,131.57) ;
\draw [shift={(446.28,131.57)}, rotate = 310.25] [color={rgb, 255:red, 0; green, 0; blue, 0 }  ][fill={rgb, 255:red, 0; green, 0; blue, 0 }  ][line width=0.75]      (0, 0) circle [x radius= 1.34, y radius= 1.34]   ;
\draw [shift={(428.85,152.15)}, rotate = 310.25] [color={rgb, 255:red, 0; green, 0; blue, 0 }  ][fill={rgb, 255:red, 0; green, 0; blue, 0 }  ][line width=0.75]      (0, 0) circle [x radius= 1.34, y radius= 1.34]   ;
  \draw    (446.28,131.57) -- (463.71,110.98) ;
\draw [shift={(463.71,110.98)}, rotate = 310.25] [color={rgb, 255:red, 0; green, 0; blue, 0 }  ][fill={rgb, 255:red, 0; green, 0; blue, 0 }  ][line width=0.75]      (0, 0) circle [x radius= 1.34, y radius= 1.34]   ;
\draw [shift={(446.28,131.57)}, rotate = 310.25] [color={rgb, 255:red, 0; green, 0; blue, 0 }  ][fill={rgb, 255:red, 0; green, 0; blue, 0 }  ][line width=0.75]      (0, 0) circle [x radius= 1.34, y radius= 1.34]   ;
  \draw    (359.13,193.33) -- (376.56,172.74) ;
\draw [shift={(376.56,172.74)}, rotate = 310.25] [color={rgb, 255:red, 0; green, 0; blue, 0 }  ][fill={rgb, 255:red, 0; green, 0; blue, 0 }  ][line width=0.75]      (0, 0) circle [x radius= 1.34, y radius= 1.34]   ;
\draw [shift={(359.13,193.33)}, rotate = 310.25] [color={rgb, 255:red, 0; green, 0; blue, 0 }  ][fill={rgb, 255:red, 0; green, 0; blue, 0 }  ][line width=0.75]      (0, 0) circle [x radius= 1.34, y radius= 1.34]   ;
  \draw    (376.56,172.74) -- (393.99,152.15) ;
\draw [shift={(393.99,152.15)}, rotate = 310.25] [color={rgb, 255:red, 0; green, 0; blue, 0 }  ][fill={rgb, 255:red, 0; green, 0; blue, 0 }  ][line width=0.75]      (0, 0) circle [x radius= 1.34, y radius= 1.34]   ;
\draw [shift={(376.56,172.74)}, rotate = 310.25] [color={rgb, 255:red, 0; green, 0; blue, 0 }  ][fill={rgb, 255:red, 0; green, 0; blue, 0 }  ][line width=0.75]      (0, 0) circle [x radius= 1.34, y radius= 1.34]   ;
  \draw    (393.99,152.15) -- (411.42,131.57) ;
\draw [shift={(411.42,131.57)}, rotate = 310.25] [color={rgb, 255:red, 0; green, 0; blue, 0 }  ][fill={rgb, 255:red, 0; green, 0; blue, 0 }  ][line width=0.75]      (0, 0) circle [x radius= 1.34, y radius= 1.34]   ;
\draw [shift={(393.99,152.15)}, rotate = 310.25] [color={rgb, 255:red, 0; green, 0; blue, 0 }  ][fill={rgb, 255:red, 0; green, 0; blue, 0 }  ][line width=0.75]      (0, 0) circle [x radius= 1.34, y radius= 1.34]   ;
  \draw    (411.42,131.57) -- (428.85,110.98) ;
\draw [shift={(428.85,110.98)}, rotate = 310.25] [color={rgb, 255:red, 0; green, 0; blue, 0 }  ][fill={rgb, 255:red, 0; green, 0; blue, 0 }  ][line width=0.75]      (0, 0) circle [x radius= 1.34, y radius= 1.34]   ;
\draw [shift={(411.42,131.57)}, rotate = 310.25] [color={rgb, 255:red, 0; green, 0; blue, 0 }  ][fill={rgb, 255:red, 0; green, 0; blue, 0 }  ][line width=0.75]      (0, 0) circle [x radius= 1.34, y radius= 1.34]   ;
  \draw    (324.27,193.33) -- (341.7,172.74) ;
\draw [shift={(341.7,172.74)}, rotate = 310.25] [color={rgb, 255:red, 0; green, 0; blue, 0 }  ][fill={rgb, 255:red, 0; green, 0; blue, 0 }  ][line width=0.75]      (0, 0) circle [x radius= 1.34, y radius= 1.34]   ;
\draw [shift={(324.27,193.33)}, rotate = 310.25] [color={rgb, 255:red, 0; green, 0; blue, 0 }  ][fill={rgb, 255:red, 0; green, 0; blue, 0 }  ][line width=0.75]      (0, 0) circle [x radius= 1.34, y radius= 1.34]   ;
  \draw    (341.7,172.74) -- (359.13,152.15) ;
\draw [shift={(359.13,152.15)}, rotate = 310.25] [color={rgb, 255:red, 0; green, 0; blue, 0 }  ][fill={rgb, 255:red, 0; green, 0; blue, 0 }  ][line width=0.75]      (0, 0) circle [x radius= 1.34, y radius= 1.34]   ;
\draw [shift={(341.7,172.74)}, rotate = 310.25] [color={rgb, 255:red, 0; green, 0; blue, 0 }  ][fill={rgb, 255:red, 0; green, 0; blue, 0 }  ][line width=0.75]      (0, 0) circle [x radius= 1.34, y radius= 1.34]   ;
  \draw    (359.13,152.15) -- (376.56,131.57) ;
\draw [shift={(376.56,131.57)}, rotate = 310.25] [color={rgb, 255:red, 0; green, 0; blue, 0 }  ][fill={rgb, 255:red, 0; green, 0; blue, 0 }  ][line width=0.75]      (0, 0) circle [x radius= 1.34, y radius= 1.34]   ;
\draw [shift={(359.13,152.15)}, rotate = 310.25] [color={rgb, 255:red, 0; green, 0; blue, 0 }  ][fill={rgb, 255:red, 0; green, 0; blue, 0 }  ][line width=0.75]      (0, 0) circle [x radius= 1.34, y radius= 1.34]   ;
  \draw    (376.56,131.57) -- (393.99,110.98) ;
\draw [shift={(393.99,110.98)}, rotate = 310.25] [color={rgb, 255:red, 0; green, 0; blue, 0 }  ][fill={rgb, 255:red, 0; green, 0; blue, 0 }  ][line width=0.75]      (0, 0) circle [x radius= 1.34, y radius= 1.34]   ;
\draw [shift={(376.56,131.57)}, rotate = 310.25] [color={rgb, 255:red, 0; green, 0; blue, 0 }  ][fill={rgb, 255:red, 0; green, 0; blue, 0 }  ][line width=0.75]      (0, 0) circle [x radius= 1.34, y radius= 1.34]   ;
  \draw    (254.55,193.33) -- (271.98,172.74) ;
\draw [shift={(271.98,172.74)}, rotate = 310.25] [color={rgb, 255:red, 0; green, 0; blue, 0 }  ][fill={rgb, 255:red, 0; green, 0; blue, 0 }  ][line width=0.75]      (0, 0) circle [x radius= 1.34, y radius= 1.34]   ;
\draw [shift={(254.55,193.33)}, rotate = 310.25] [color={rgb, 255:red, 0; green, 0; blue, 0 }  ][fill={rgb, 255:red, 0; green, 0; blue, 0 }  ][line width=0.75]      (0, 0) circle [x radius= 1.34, y radius= 1.34]   ;
  \draw    (271.98,172.74) -- (289.41,152.15) ;
\draw [shift={(289.41,152.15)}, rotate = 310.25] [color={rgb, 255:red, 0; green, 0; blue, 0 }  ][fill={rgb, 255:red, 0; green, 0; blue, 0 }  ][line width=0.75]      (0, 0) circle [x radius= 1.34, y radius= 1.34]   ;
\draw [shift={(271.98,172.74)}, rotate = 310.25] [color={rgb, 255:red, 0; green, 0; blue, 0 }  ][fill={rgb, 255:red, 0; green, 0; blue, 0 }  ][line width=0.75]      (0, 0) circle [x radius= 1.34, y radius= 1.34]   ;
  \draw    (289.41,152.15) -- (306.84,131.57) ;
\draw [shift={(306.84,131.57)}, rotate = 310.25] [color={rgb, 255:red, 0; green, 0; blue, 0 }  ][fill={rgb, 255:red, 0; green, 0; blue, 0 }  ][line width=0.75]      (0, 0) circle [x radius= 1.34, y radius= 1.34]   ;
\draw [shift={(289.41,152.15)}, rotate = 310.25] [color={rgb, 255:red, 0; green, 0; blue, 0 }  ][fill={rgb, 255:red, 0; green, 0; blue, 0 }  ][line width=0.75]      (0, 0) circle [x radius= 1.34, y radius= 1.34]   ;
  \draw    (306.84,131.57) -- (324.27,110.98) ;
\draw [shift={(324.27,110.98)}, rotate = 310.25] [color={rgb, 255:red, 0; green, 0; blue, 0 }  ][fill={rgb, 255:red, 0; green, 0; blue, 0 }  ][line width=0.75]      (0, 0) circle [x radius= 1.34, y radius= 1.34]   ;
\draw [shift={(306.84,131.57)}, rotate = 310.25] [color={rgb, 255:red, 0; green, 0; blue, 0 }  ][fill={rgb, 255:red, 0; green, 0; blue, 0 }  ][line width=0.75]      (0, 0) circle [x radius= 1.34, y radius= 1.34]   ;
  \draw    (289.41,193.33) -- (306.84,172.74) ;
\draw [shift={(306.84,172.74)}, rotate = 310.25] [color={rgb, 255:red, 0; green, 0; blue, 0 }  ][fill={rgb, 255:red, 0; green, 0; blue, 0 }  ][line width=0.75]      (0, 0) circle [x radius= 1.34, y radius= 1.34]   ;
\draw [shift={(289.41,193.33)}, rotate = 310.25] [color={rgb, 255:red, 0; green, 0; blue, 0 }  ][fill={rgb, 255:red, 0; green, 0; blue, 0 }  ][line width=0.75]      (0, 0) circle [x radius= 1.34, y radius= 1.34]   ;
  \draw    (306.84,172.74) -- (324.27,152.15) ;
\draw [shift={(324.27,152.15)}, rotate = 310.25] [color={rgb, 255:red, 0; green, 0; blue, 0 }  ][fill={rgb, 255:red, 0; green, 0; blue, 0 }  ][line width=0.75]      (0, 0) circle [x radius= 1.34, y radius= 1.34]   ;
\draw [shift={(306.84,172.74)}, rotate = 310.25] [color={rgb, 255:red, 0; green, 0; blue, 0 }  ][fill={rgb, 255:red, 0; green, 0; blue, 0 }  ][line width=0.75]      (0, 0) circle [x radius= 1.34, y radius= 1.34]   ;
  \draw    (324.27,152.15) -- (341.7,131.57) ;
\draw [shift={(341.7,131.57)}, rotate = 310.25] [color={rgb, 255:red, 0; green, 0; blue, 0 }  ][fill={rgb, 255:red, 0; green, 0; blue, 0 }  ][line width=0.75]      (0, 0) circle [x radius= 1.34, y radius= 1.34]   ;
\draw [shift={(324.27,152.15)}, rotate = 310.25] [color={rgb, 255:red, 0; green, 0; blue, 0 }  ][fill={rgb, 255:red, 0; green, 0; blue, 0 }  ][line width=0.75]      (0, 0) circle [x radius= 1.34, y radius= 1.34]   ;
  \draw    (341.7,131.57) -- (359.13,110.98) ;
\draw [shift={(359.13,110.98)}, rotate = 310.25] [color={rgb, 255:red, 0; green, 0; blue, 0 }  ][fill={rgb, 255:red, 0; green, 0; blue, 0 }  ][line width=0.75]      (0, 0) circle [x radius= 1.34, y radius= 1.34]   ;
\draw [shift={(341.7,131.57)}, rotate = 310.25] [color={rgb, 255:red, 0; green, 0; blue, 0 }  ][fill={rgb, 255:red, 0; green, 0; blue, 0 }  ][line width=0.75]      (0, 0) circle [x radius= 1.34, y radius= 1.34]   ;
  \draw    (219.68,193.33) -- (237.11,172.74) ;
\draw [shift={(237.11,172.74)}, rotate = 310.25] [color={rgb, 255:red, 0; green, 0; blue, 0 }  ][fill={rgb, 255:red, 0; green, 0; blue, 0 }  ][line width=0.75]      (0, 0) circle [x radius= 1.34, y radius= 1.34]   ;
\draw [shift={(219.68,193.33)}, rotate = 310.25] [color={rgb, 255:red, 0; green, 0; blue, 0 }  ][fill={rgb, 255:red, 0; green, 0; blue, 0 }  ][line width=0.75]      (0, 0) circle [x radius= 1.34, y radius= 1.34]   ;
  \draw    (237.11,172.74) -- (254.55,152.15) ;
\draw [shift={(254.55,152.15)}, rotate = 310.25] [color={rgb, 255:red, 0; green, 0; blue, 0 }  ][fill={rgb, 255:red, 0; green, 0; blue, 0 }  ][line width=0.75]      (0, 0) circle [x radius= 1.34, y radius= 1.34]   ;
\draw [shift={(237.11,172.74)}, rotate = 310.25] [color={rgb, 255:red, 0; green, 0; blue, 0 }  ][fill={rgb, 255:red, 0; green, 0; blue, 0 }  ][line width=0.75]      (0, 0) circle [x radius= 1.34, y radius= 1.34]   ;
  \draw    (254.55,152.15) -- (271.98,131.57) ;
\draw [shift={(271.98,131.57)}, rotate = 310.25] [color={rgb, 255:red, 0; green, 0; blue, 0 }  ][fill={rgb, 255:red, 0; green, 0; blue, 0 }  ][line width=0.75]      (0, 0) circle [x radius= 1.34, y radius= 1.34]   ;
\draw [shift={(254.55,152.15)}, rotate = 310.25] [color={rgb, 255:red, 0; green, 0; blue, 0 }  ][fill={rgb, 255:red, 0; green, 0; blue, 0 }  ][line width=0.75]      (0, 0) circle [x radius= 1.34, y radius= 1.34]   ;
  \draw    (271.98,131.57) -- (289.41,110.98) ;
\draw [shift={(289.41,110.98)}, rotate = 310.25] [color={rgb, 255:red, 0; green, 0; blue, 0 }  ][fill={rgb, 255:red, 0; green, 0; blue, 0 }  ][line width=0.75]      (0, 0) circle [x radius= 1.34, y radius= 1.34]   ;
\draw [shift={(271.98,131.57)}, rotate = 310.25] [color={rgb, 255:red, 0; green, 0; blue, 0 }  ][fill={rgb, 255:red, 0; green, 0; blue, 0 }  ][line width=0.75]      (0, 0) circle [x radius= 1.34, y radius= 1.34]   ;
  
\draw    (184.82,110.98) -- (254.55,193.33) ;
  \draw    (463.71,110.98) -- (498.57,152.15) ;
  \draw    (115.1,110.98) -- (184.82,193.33) ;
\draw [shift={(115.1,110.98)}, rotate = 49.75] [color={rgb, 255:red, 0; green, 0; blue, 0 }  ][fill={rgb, 255:red, 0; green, 0; blue, 0 }  ][line width=0.75]      (0, 0) circle [x radius= 1.34, y radius= 1.34]   ;
  \draw  [dash pattern={on 0.84pt off 2.51pt}]  (77.88,111.29) -- (537.42,111.29) ;
  \draw    (219.68,110.98) -- (289.41,193.33) ;
  \draw    (254.55,110.98) -- (324.27,193.33) ;
  \draw    (289.41,110.98) -- (359.13,193.33) ;
  \draw    (324.27,110.98) -- (393.99,193.33) ;
  \draw    (359.13,110.98) -- (428.85,193.33) ;
  \draw    (393.99,110.98) -- (463.71,193.33) ;
  \draw    (428.85,110.98) -- (498.57,193.33) ;
\draw [shift={(498.57,193.33)}, rotate = 49.75] [color={rgb, 255:red, 0; green, 0; blue, 0 }  ][fill={rgb, 255:red, 0; green, 0; blue, 0 }  ][line width=0.75]      (0, 0) circle [x radius= 1.34, y radius= 1.34]   ;
  
\draw    (149.96,110.98) -- (219.68,193.33) ;
  \draw  [dash pattern={on 0.84pt off 2.51pt}]  (77.88,193.44) -- (538.3,193.44) ;
  \draw [color={rgb, 255:red, 208; green, 2; blue, 27 }  ,draw opacity=1 ][line width=1.5]    (184.82,110.98) -- (254.55,193.33) ;
\draw [shift={(222.53,155.51)}, rotate = 229.75] [fill={rgb, 255:red, 208; green, 2; blue, 27 }  ,fill opacity=1 ][line width=0.08]  [draw opacity=0] (8.75,-4.2) -- (0,0) -- (8.75,4.2) -- (5.81,0) -- cycle    ;
  \draw [color={rgb, 255:red, 208; green, 2; blue, 27 }  ,draw opacity=1 ][line width=1.5]    (324.27,110.98) -- (393.99,193.33) ;
\draw [shift={(361.97,155.51)}, rotate = 229.75] [fill={rgb, 255:red, 208; green, 2; blue, 27 }  ,fill opacity=1 ][line width=0.08]  [draw opacity=0] (8.75,-4.2) -- (0,0) -- (8.75,4.2) -- (5.81,0) -- cycle    ;
  \draw [color={rgb, 255:red, 208; green, 2; blue, 27 }  ,draw opacity=1 ][line width=1.5]    (254.55,193.33) -- (324.27,110.98) ;
\draw [shift={(292.25,148.8)}, rotate = 130.25] [fill={rgb, 255:red, 208; green, 2; blue, 27 }  ,fill opacity=1 ][line width=0.08]  [draw opacity=0] (8.75,-4.2) -- (0,0) -- (8.75,4.2) -- (5.81,0) -- cycle    ;
  \draw [color={rgb, 255:red, 208; green, 2; blue, 27 }  ,draw opacity=1 ][line width=1.5]    (393.99,193.33) -- (463.71,110.98) ;
\draw [shift={(431.69,148.8)}, rotate = 130.25] [fill={rgb, 255:red, 208; green, 2; blue, 27 }  ,fill opacity=1 ][line width=0.08]  [draw opacity=0] (8.75,-4.2) -- (0,0) -- (8.75,4.2) -- (5.81,0) -- cycle    ;
  \draw [color={rgb, 255:red, 208; green, 2; blue, 27 }  ,draw opacity=1 ][line width=1.5]    (463.71,110.98) -- (498.57,152.15) ;
\draw [shift={(483.98,134.93)}, rotate = 229.75] [fill={rgb, 255:red, 208; green, 2; blue, 27 }  ,fill opacity=1 ][line width=0.08]  [draw opacity=0] (8.75,-4.2) -- (0,0) -- (8.75,4.2) -- (5.81,0) -- cycle    ;
  \draw [color={rgb, 255:red, 208; green, 2; blue, 27 }  ,draw opacity=1 ][line width=1.5]    (115.1,193.33) -- (184.82,110.98) ;
\draw [shift={(152.81,148.8)}, rotate = 130.25] [fill={rgb, 255:red, 208; green, 2; blue, 27 }  ,fill opacity=1 ][line width=0.08]  [draw opacity=0] (8.75,-4.2) -- (0,0) -- (8.75,4.2) -- (5.81,0) -- cycle    ;
  \draw    (132.53,131.57) -- (149.96,110.98) ;
\draw [shift={(149.96,110.98)}, rotate = 310.25] [color={rgb, 255:red, 0; green, 0; blue, 0 }  ][fill={rgb, 255:red, 0; green, 0; blue, 0 }  ][line width=0.75]      (0, 0) circle [x radius= 1.34, y radius= 1.34]   ;
\draw [shift={(132.53,131.57)}, rotate = 310.25] [color={rgb, 255:red, 0; green, 0; blue, 0 }  ][fill={rgb, 255:red, 0; green, 0; blue, 0 }  ][line width=0.75]      (0, 0) circle [x radius= 1.34, y radius= 1.34]   ;
  \draw    (115.1,152.15) -- (132.53,131.57) ;
\draw [shift={(132.53,131.57)}, rotate = 310.25] [color={rgb, 255:red, 0; green, 0; blue, 0 }  ][fill={rgb, 255:red, 0; green, 0; blue, 0 }  ][line width=0.75]      (0, 0) circle [x radius= 1.34, y radius= 1.34]   ;
\draw [shift={(115.1,152.15)}, rotate = 310.25] [color={rgb, 255:red, 0; green, 0; blue, 0 }  ][fill={rgb, 255:red, 0; green, 0; blue, 0 }  ][line width=0.75]      (0, 0) circle [x radius= 1.34, y radius= 1.34]   ;
  \draw    (463.71,193.33) -- (481.14,172.74) ;
\draw [shift={(481.14,172.74)}, rotate = 310.25] [color={rgb, 255:red, 0; green, 0; blue, 0 }  ][fill={rgb, 255:red, 0; green, 0; blue, 0 }  ][line width=0.75]      (0, 0) circle [x radius= 1.34, y radius= 1.34]   ;
\draw [shift={(463.71,193.33)}, rotate = 310.25] [color={rgb, 255:red, 0; green, 0; blue, 0 }  ][fill={rgb, 255:red, 0; green, 0; blue, 0 }  ][line width=0.75]      (0, 0) circle [x radius= 1.34, y radius= 1.34]   ;
  \draw    (481.14,172.74) -- (498.57,152.15) ;
\draw [shift={(498.57,152.15)}, rotate = 310.25] [color={rgb, 255:red, 0; green, 0; blue, 0 }  ][fill={rgb, 255:red, 0; green, 0; blue, 0 }  ][line width=0.75]      (0, 0) circle [x radius= 1.34, y radius= 1.34]   ;
\draw [shift={(481.14,172.74)}, rotate = 310.25] [color={rgb, 255:red, 0; green, 0; blue, 0 }  ][fill={rgb, 255:red, 0; green, 0; blue, 0 }  ][line width=0.75]      (0, 0) circle [x radius= 1.34, y radius= 1.34]   ;

 \draw (233.02,202.93) node [anchor=north west][inner sep=0.75pt]  [font=\tiny]  {$\mathsf{E}_{L_0( i\vec{x}_{3})} \langle -\vec{x}_{3} \rangle $};
 \draw (162.3,91.29) node [anchor=north west][inner sep=0.75pt]  [font=\tiny]  {$\mathsf{E}_{L_0( i\vec{x}_{3})} \langle -5\vec{x}_{3} \rangle $};
 \draw (304.13,91.25) node [anchor=north west][inner sep=0.75pt]  [font=\tiny]  {$\mathsf{E}_{L_0( i\vec{x}_{3})} \langle 3\vec{x}_{3} \rangle $};
 \draw (63.6,140.54) node [anchor=north west][inner sep=0.75pt]    {$\cdots $};
 \draw (517.27,140.38) node [anchor=north west][inner sep=0.75pt]    {$\cdots $};
 \draw (58,101.4) node [anchor=north west][inner sep=0.75pt]    {$l_{1}$};
 \draw (56,184.4) node [anchor=north west][inner sep=0.75pt]    {$l_{2}$};
 \draw (375.13,202.25) node [anchor=north west][inner sep=0.75pt]  [font=\tiny]  {$\mathsf{E}_{L_0( i\vec{x}_{3})} \langle 7\vec{x}_{3} \rangle $};
\end{tikzpicture}
\caption{The polyline in $(\Gamma ,d)$ for $n=4$}
\end{figure}
\noindent As the parameter $k$ increases, the endpoint $(k+i+1,1)$ of the line segment $[-i,k+i+1]$ moves rightward, and the corresponding generalized extension bundle $\mathsf{E}_{L_0(i\vec{x}_3)}\langle k\vec{x}_3\rangle$ shifts along the direction of the polyline in the valued quiver $(\Gamma ,d)$. Conversely, as $k$ decreases, the endpoint $(k+i+1,1)$ moves leftward, and the generalized extension bundle shifts in the opposite direction of the polyline.
\end{rem}

As an application of Proposition \ref{correspondence1}, we have
\begin{prop}{}{}\label{exact sequence1}
Let $k_1$ and $k_2$ be positive integers, and let $[i,j]\in \operatorname{Seg}(M)$ be a diagonal of the quadrilateral as follows:
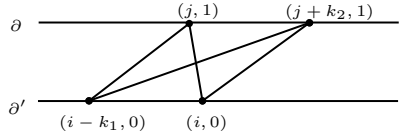
\begin{figure}[H]
    \centering

\tikzset{every picture/.style={line width=0.75pt}}          

\begin{tikzpicture}[x=0.75pt,y=0.75pt,yscale=-1,xscale=1]

\draw    (158.33,105) -- (338.08,105) ;
  
\draw    (158,144.5) -- (337.75,144.5) ;
  
\draw [color={rgb, 255:red, 0; green, 0; blue, 0 }  ,draw opacity=1 ]   (240.1,144.43) -- (233.5,105.5) ;
\draw [shift={(240.1,144.43)}, rotate = 260.38] [color={rgb, 255:red, 0; green, 0; blue, 0 }  ,draw opacity=1 ][fill={rgb, 255:red, 0; green, 0; blue, 0 }  ,fill opacity=1 ][line width=0.75]      (0, 0) circle [x radius= 1.34, y radius= 1.34]   ;
  
\draw [color={rgb, 255:red, 0; green, 0; blue, 0 }  ,draw opacity=1 ]   (293.52,105.15) -- (183.47,144.4) ;
\draw [shift={(183.47,144.4)}, rotate = 160.37] [color={rgb, 255:red, 0; green, 0; blue, 0 }  ,draw opacity=1 ][fill={rgb, 255:red, 0; green, 0; blue, 0 }  ,fill opacity=1 ][line width=0.75]      (0, 0) circle [x radius= 1.34, y radius= 1.34]   ;
\draw [shift={(293.52,105.15)}, rotate = 160.37] [color={rgb, 255:red, 0; green, 0; blue, 0 }  ,draw opacity=1 ][fill={rgb, 255:red, 0; green, 0; blue, 0 }  ,fill opacity=1 ][line width=0.75]      (0, 0) circle [x radius= 1.34, y radius= 1.34]   ;
  
\draw [color={rgb, 255:red, 0; green, 0; blue, 0 }  ,draw opacity=1 ]   (233.5,105.5) -- (183.47,144.4) ;
  
\draw [color={rgb, 255:red, 0; green, 0; blue, 0 }  ,draw opacity=1 ]   (293.52,105.15) -- (240.1,144.43) ;
  
\draw    (293.52,105.15) ;
\draw [shift={(293.52,105.15)}, rotate = 0] [color={rgb, 255:red, 0; green, 0; blue, 0 }  ][fill={rgb, 255:red, 0; green, 0; blue, 0 }  ][line width=0.75]      (0, 0) circle [x radius= 1.34, y radius= 1.34]   ;
  
\draw    (233.5,105.5) ;
\draw [shift={(233.5,105.5)}, rotate = 0] [color={rgb, 255:red, 0; green, 0; blue, 0 }  ][fill={rgb, 255:red, 0; green, 0; blue, 0 }  ][line width=0.75]      (0, 0) circle [x radius= 1.34, y radius= 1.34]   ;
  
\draw    (183.47,144.4) ;
\draw [shift={(183.47,144.4)}, rotate = 0] [color={rgb, 255:red, 0; green, 0; blue, 0 }  ][fill={rgb, 255:red, 0; green, 0; blue, 0 }  ][line width=0.75]      (0, 0) circle [x radius= 1.34, y radius= 1.34]   ;
  
\draw    (240.1,144.43) ;
\draw [shift={(240.1,144.43)}, rotate = 0] [color={rgb, 255:red, 0; green, 0; blue, 0 }  ][fill={rgb, 255:red, 0; green, 0; blue, 0 }  ][line width=0.75]      (0, 0) circle [x radius= 1.34, y radius= 1.34]   ;

\draw (141.67,142.4) node [anchor=north west][inner sep=0.75pt]  [font=\tiny]  {$\partial ^{\prime }$};
 
\draw (142.67,102.4) node [anchor=north west][inner sep=0.75pt]  [font=\tiny]  {$\partial $};
 
\draw (230.2,149.93) node [anchor=north west][inner sep=0.75pt]  [font=\tiny]  {$( i,0)$};
 
\draw (225.83,93.73) node [anchor=north west][inner sep=0.75pt]  [font=\tiny]  {$( j,1)$};
 
\draw (279.83,94.13) node [anchor=north west][inner sep=0.75pt]  [font=\tiny]  {$( j+k_{2} ,1)$};
 
\draw (167,149.93) node [anchor=north west][inner sep=0.75pt]  [font=\tiny]  {$( i-k_{1} ,0)$};

\end{tikzpicture}
\end{figure}
\noindent Then there exists an exact sequence in ${\rm vect}^F\mbox{-}\mathbb{X}$ of the following form:
\[0\longrightarrow \phi({\widetilde{[i,j]}})\longrightarrow \phi({\widetilde{[i-k_1,j]}})\oplus\phi({\widetilde{[i,j+k_2]}})\longrightarrow \phi({\widetilde{[i-k_1,j+k_2]}})\longrightarrow 0.\]
\noindent In particular, $\phi(\widetilde{[i-k_1,j]})=\phi(\widetilde{[i,j+k_2]})$ if and only if one of following holds: 
    \begin{itemize}
        \item $k_1=k_2$ and $i+j\equiv 0\; (\bmod\; n) $; 
        \item  $k_1=k_2\equiv 0\; (\bmod\; n)$.
    \end{itemize}   
\end{prop}

\begin{proof}
 Let $\vec{x}=(i+j-1)\vec{x}_3, \vec{y}=k_1\vec{x}_3$ and $\vec{z}=k_2\vec{x}_3$. We have the commutative diagram
\[\begin{tikzcd}[ampersand replacement=\&,cramped]
	{\mathsf{E}_{L_0(-i\vec{x}_3)}\langle\vec{x}\rangle} \& \cdot \& \cdots \& \cdot \& {\mathsf{E}_{L_0(\vec{y}-i\vec{x}_3)}\langle\vec{x}-\vec{y}\rangle} \\
	\cdot \& \cdot \& \cdots \& \cdot \& \cdot \\
	\vdots \& \vdots \&\& \vdots \& \vdots \\
	\cdot \& \cdot \& \cdots \& \cdot \& \cdot \\
	{\mathsf{E}_{L_0(-i\vec{x}_3)}\langle\vec{x}+\vec{z}\rangle} \& \cdot \& \cdots \& \cdot \& {\mathsf{E}_{L_0(\vec{y}-i\vec{x}_3)}\langle\vec{x}-\vec{y}+\vec{z}\rangle}
	\arrow[from=1-1, to=1-2]
	\arrow[from=1-1, to=2-1]
	\arrow[from=1-2, to=1-3]
	\arrow[from=1-2, to=2-2]
	\arrow[from=1-3, to=1-4]
	\arrow[from=1-4, to=1-5]
	\arrow[from=1-4, to=2-4]
	\arrow[from=1-5, to=2-5]
	\arrow[from=2-1, to=2-2]
	\arrow[from=2-1, to=3-1]
	\arrow[from=2-2, to=2-3]
	\arrow[from=2-2, to=3-2]
	\arrow[from=2-3, to=2-4]
	\arrow[from=2-4, to=2-5]
	\arrow[from=2-4, to=3-4]
	\arrow[from=2-5, to=3-5]
	\arrow[from=3-1, to=4-1]
	\arrow[from=3-2, to=4-2]
	\arrow[from=3-4, to=4-4]
	\arrow[from=3-5, to=4-5]
	\arrow[from=4-1, to=4-2]
	\arrow[from=4-1, to=5-1]
	\arrow[from=4-2, to=4-3]
	\arrow[from=4-2, to=5-2]
	\arrow[from=4-3, to=4-4]
	\arrow[from=4-4, to=4-5]
	\arrow[from=4-4, to=5-4]
	\arrow[from=4-5, to=5-5]
	\arrow[from=5-1, to=5-2]
	\arrow[from=5-2, to=5-3]
	\arrow[from=5-3, to=5-4]
	\arrow[from=5-4, to=5-5]
\end{tikzcd}\]
where the smallest squares are bicartesian, as obtained by Proposition \ref{PPS1}. Employing the pullback lemma, a statement about the composition of pullback/pushout diagrams, the largest square is bicartesian. Under the bijection $\phi$, the largest bicartesian square induces the required exact sequence.

Since $\phi(\widetilde{[i-k_1,j]})=\phi(\widetilde{[i,j+k_2]})$ if and only if there exists $g\in G$ such that $g\cdot[i-k_1,j]=[i,j+k_2]$ if and only if $k_1=k_2$ and $i+j\equiv 0\;(\bmod\; n)$  or  $k_1=k_2\equiv 0\;(\bmod\; n)$, the remaining assertion follows.
\end{proof}

\subsection{Geometric interpretation for the automorphism group of $(\Gamma, d)$}
For the concept of the mapping class group of a marked surface, please refer to \cite{MR2946087}. For a marked surface $\Sigma$, let $\mathcal{MG}(\Sigma)$ denote its mapping class group. Let $\mathcal{D}_{2,n}$ denote a twice-punctured disk with $n$ marked points on its boundary. As introduced in \cite{MR4289034}, given $\mathcal{D}_{2,n}$ with a triangulation (in the sense of \cite[Definition 2.6]{MR2448067}), one can construct a cylinder ${\rm Cyl}_{n,n}$ with $n$ marked points on its upper and lower boundaries. By equipping ${\rm Cyl}_{n,n}$ with specific topological and complex structures induced from $\mathcal{D}_{2,n}$, the cylinder ${\rm Cyl}_{n,n}$ can be transformed into a Riemann surface with a biholomorphism $\sigma: {\rm Cyl}_{n,n} \rightarrow {\rm Cyl}_{n,n}$. Specifically, $\sigma$ acts as an involution on the cylinder ${\rm Cyl}_{n,n}$, rotating every point about a horizontal line passing through the center of ${\rm Cyl}_{n,n}$ by 180 degrees. Moreover, the orbit space ${\rm Cyl}_{n,n}/\langle\sigma\rangle$ is homeomorphic to $\mathcal{D}_{2,n}$, inducing an orbifold structure on $\mathcal{D}_{2,n}$.

In the previous subsection, we consider $\widetilde{\mathcal{S}}$ as a set of points without specifying its topological or complex structure. By equipping specific topological and complex structures induced from ${\rm Cyl}_{n,n}$ onto $\widetilde{\mathcal{S}}$, the strip $\widetilde{\mathcal{S}}$ can naturally be regarded as the universal covering space of ${\rm Cyl}_{n,n}$ with the covering map $p: \widetilde{\mathcal{S}} \rightarrow {\rm Cyl}_{n,n}$. Moreover, the self-homeomorphisms $\sigma_n, \theta$ of $\widetilde{\mathcal{S}}$ satisfy $p \sigma_n = p$ and $p \theta = \sigma p$. In this context, the orbit space $\mathcal{S} = \widetilde{\mathcal{S}} / G$ is homeomorphic to ${\rm Cyl}_{n,n}/\langle\sigma\rangle$. Consequently, we obtain the following commutative diagram:
\[\begin{tikzcd}[ampersand replacement=\&,cramped]
	{\widetilde{\mathcal{S}}} \& {\widetilde{\mathcal{S}}/\langle\sigma_n\rangle} \& {{\rm Cyl}_{n,n}} \\
	\& {\mathcal{S}} \& {{\rm Cyl}_{n,n}/\langle\sigma\rangle} \& {\mathcal{D}_{2,n}}
	\arrow["{\pi_1}", two heads, from=1-1, to=1-2]
	\arrow["p", shift left, curve={height=-18pt}, shorten <=7pt, two heads, from=1-1, to=1-3]
	\arrow["\pi"', two heads, from=1-1, to=2-2]
	\arrow["\cong", from=1-2, to=1-3]
	\arrow["{\pi_2}", two heads, from=1-2, to=2-2]
	\arrow[two heads, from=1-3, to=2-3]
	\arrow["\cong", from=2-2, to=2-3]
	\arrow["\cong", from=2-3, to=2-4]
\end{tikzcd}\]
\noindent where $\pi,\pi_1$ and $\pi_2$ are canonical projection maps. For example, let $n=3.$
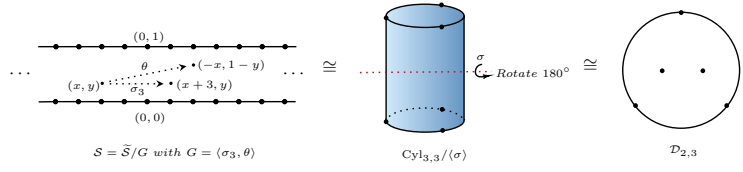
\begin{figure}[H]
    \centering

\begin{adjustbox}{scale=0.7}

\tikzset {_ubxy1tlzn/.code = {\pgfsetadditionalshadetransform{ \pgftransformshift{\pgfpoint{0 bp } { 0 bp }  }  \pgftransformrotate{0 }  \pgftransformscale{2 }  }}}
\pgfdeclarehorizontalshading{_mdgo3dp8u}{150bp}{rgb(0bp)=(0.81,0.91,0.98);
rgb(37.5bp)=(0.81,0.91,0.98);
rgb(62.5bp)=(0.39,0.58,0.76);
rgb(100bp)=(0.39,0.58,0.76)}
\tikzset{every picture/.style={line width=0.75pt}}       
\begin{tikzpicture}[x=0.75pt,y=0.75pt,yscale=-1,xscale=1]
\draw  [draw opacity=0][dash pattern={on 0.84pt off 2.51pt}] (280.26,147.63) .. controls (280.26,147.63) and (280.26,147.63) .. (280.26,147.63) .. controls (280.26,142.04) and (292.76,137.51) .. (308.18,137.51) .. controls (323.6,137.51) and (336.1,142.04) .. (336.1,147.63) -- (308.18,147.63) -- cycle ; \draw [dash pattern={on 0.84pt off 2.51pt}] [dash pattern={on 0.84pt off 2.51pt}]  (280.26,147.63) .. controls (280.26,142.04) and (292.76,137.51) .. (308.18,137.51) .. controls (323.6,137.51) and (336.1,142.04) .. (336.1,147.63) ;  
\path  [shading=_mdgo3dp8u,_ubxy1tlzn] (335.85,72.73) -- (335.85,147.96) .. controls (335.85,152.56) and (323.41,156.3) .. (308.06,156.3) .. controls (292.71,156.3) and (280.26,152.56) .. (280.26,147.96) -- (280.26,72.73) .. controls (280.26,68.13) and (292.71,64.39) .. (308.06,64.39) .. controls (323.41,64.39) and (335.85,68.13) .. (335.85,72.73) .. controls (335.85,77.34) and (323.41,81.07) .. (308.06,81.07) .. controls (292.71,81.07) and (280.26,77.34) .. (280.26,72.73) ; 
 \draw   (335.85,72.73) -- (335.85,147.96) .. controls (335.85,152.56) and (323.41,156.3) .. (308.06,156.3) .. controls (292.71,156.3) and (280.26,152.56) .. (280.26,147.96) -- (280.26,72.73) .. controls (280.26,68.13) and (292.71,64.39) .. (308.06,64.39) .. controls (323.41,64.39) and (335.85,68.13) .. (335.85,72.73) .. controls (335.85,77.34) and (323.41,81.07) .. (308.06,81.07) .. controls (292.71,81.07) and (280.26,77.34) .. (280.26,72.73) ; 
\draw  [fill={rgb, 255:red, 0; green, 0; blue, 0 }  ,fill opacity=1 ] (318.7,80.73) .. controls (318.7,80.06) and (319.23,79.51) .. (319.88,79.51) .. controls (320.53,79.51) and (321.06,80.06) .. (321.06,80.73) .. controls (321.06,81.41) and (320.53,81.95) .. (319.88,81.95) .. controls (319.23,81.95) and (318.7,81.41) .. (318.7,80.73) -- cycle ;
\draw  [fill={rgb, 255:red, 0; green, 0; blue, 0 }  ,fill opacity=1 ] (318.56,64.89) .. controls (318.56,64.22) and (319.08,63.67) .. (319.73,63.67) .. controls (320.38,63.67) and (320.91,64.22) .. (320.91,64.89) .. controls (320.91,65.57) and (320.38,66.12) .. (319.73,66.12) .. controls (319.08,66.12) and (318.56,65.57) .. (318.56,64.89) -- cycle ;
\draw  [fill={rgb, 255:red, 0; green, 0; blue, 0 }  ,fill opacity=1 ] (279.08,74.29) .. controls (279.08,73.61) and (279.61,73.06) .. (280.26,73.06) .. controls (280.91,73.06) and (281.44,73.61) .. (281.44,74.29) .. controls (281.44,74.96) and (280.91,75.51) .. (280.26,75.51) .. controls (279.61,75.51) and (279.08,74.96) .. (279.08,74.29) -- cycle ;
\draw  [fill={rgb, 255:red, 0; green, 0; blue, 0 }  ,fill opacity=1 ] (319.38,155.07) .. controls (319.38,154.4) and (319.91,153.85) .. (320.56,153.85) .. controls (321.21,153.85) and (321.73,154.4) .. (321.73,155.07) .. controls (321.73,155.75) and (321.21,156.3) .. (320.56,156.3) .. controls (319.91,156.3) and (319.38,155.75) .. (319.38,155.07) -- cycle ;
\draw  [fill={rgb, 255:red, 0; green, 0; blue, 0 }  ,fill opacity=1 ] (279.08,148.85) .. controls (279.08,148.17) and (279.61,147.63) .. (280.26,147.63) .. controls (280.91,147.63) and (281.44,148.17) .. (281.44,148.85) .. controls (281.44,149.52) and (280.91,150.07) .. (280.26,150.07) .. controls (279.61,150.07) and (279.08,149.52) .. (279.08,148.85) -- cycle ;
\draw  [fill={rgb, 255:red, 0; green, 0; blue, 0 }  ,fill opacity=1 ] (319.07,139.95) .. controls (319.07,139.28) and (319.6,138.73) .. (320.25,138.73) .. controls (320.9,138.73) and (321.43,139.28) .. (321.43,139.95) .. controls (321.43,140.63) and (320.9,141.17) .. (320.25,141.17) .. controls (319.6,141.17) and (319.07,140.63) .. (319.07,139.95) -- cycle ;
\draw  [draw opacity=0][dash pattern={on 0.84pt off 2.51pt}] (280.26,148.85) .. controls (280.26,148.84) and (280.26,148.83) .. (280.26,148.82) .. controls (280.26,143.37) and (292.67,138.96) .. (307.98,138.96) .. controls (322.85,138.96) and (334.98,143.12) .. (335.68,148.34) -- (307.98,148.82) -- cycle ; 
\draw [dash pattern={on 0.84pt off 2.51pt}] [dash pattern={on 0.84pt off 2.51pt}]  (280.26,148.85) .. controls (280.26,148.84) and (280.26,148.83) .. (280.26,148.82) .. controls (280.26,143.37) and (292.67,138.96) .. (307.98,138.96) .. controls (322.85,138.96) and (334.98,143.12) .. (335.68,148.34) ;  
\draw [color={rgb, 255:red, 208; green, 2; blue, 27 }  ,draw opacity=1 ][line width=0.75]  [dash pattern={on 0.84pt off 2.51pt}]  (261.57,113.47) -- (355.49,112.75) ;
\draw   (449,111.83) .. controls (449,88.82) and (467.65,70.17) .. (490.67,70.17) .. controls (513.68,70.17) and (532.33,88.82) .. (532.33,111.83) .. controls (532.33,134.85) and (513.68,153.5) .. (490.67,153.5) .. controls (467.65,153.5) and (449,134.85) .. (449,111.83) -- cycle ;
\draw  [fill={rgb, 255:red, 0; green, 0; blue, 0 }  ,fill opacity=1 ] (475.95,112.23) .. controls (475.95,111.56) and (476.48,111.01) .. (477.13,111.01) .. controls (477.78,111.01) and (478.31,111.56) .. (478.31,112.23) .. controls (478.31,112.91) and (477.78,113.45) .. (477.13,113.45) .. controls (476.48,113.45) and (475.95,112.91) .. (475.95,112.23) -- cycle ;
\draw  [fill={rgb, 255:red, 0; green, 0; blue, 0 }  ,fill opacity=1 ] (504.95,112.23) .. controls (504.95,111.56) and (505.48,111.01) .. (506.13,111.01) .. controls (506.78,111.01) and (507.31,111.56) .. (507.31,112.23) .. controls (507.31,112.91) and (506.78,113.45) .. (506.13,113.45) .. controls (505.48,113.45) and (504.95,112.91) .. (504.95,112.23) -- cycle ;
\draw  [fill={rgb, 255:red, 0; green, 0; blue, 0 }  ,fill opacity=1 ] (456.95,137.23) .. controls (456.95,136.56) and (457.48,136.01) .. (458.13,136.01) .. controls (458.78,136.01) and (459.31,136.56) .. (459.31,137.23) .. controls (459.31,137.91) and (458.78,138.45) .. (458.13,138.45) .. controls (457.48,138.45) and (456.95,137.91) .. (456.95,137.23) -- cycle ;
\draw  [fill={rgb, 255:red, 0; green, 0; blue, 0 }  ,fill opacity=1 ] (489.49,70.39) .. controls (489.49,69.71) and (490.02,69.17) .. (490.67,69.17) .. controls (491.32,69.17) and (491.84,69.71) .. (491.84,70.39) .. controls (491.84,71.06) and (491.32,71.61) .. (490.67,71.61) .. controls (490.02,71.61) and (489.49,71.06) .. (489.49,70.39) -- cycle ;
\draw  [fill={rgb, 255:red, 0; green, 0; blue, 0 }  ,fill opacity=1 ] (521.95,137.23) .. controls (521.95,136.56) and (522.48,136.01) .. (523.13,136.01) .. controls (523.78,136.01) and (524.31,136.56) .. (524.31,137.23) .. controls (524.31,137.91) and (523.78,138.45) .. (523.13,138.45) .. controls (522.48,138.45) and (521.95,137.91) .. (521.95,137.23) -- cycle ;
\draw    (348.74,107.63) .. controls (343.99,107.88) and (342.99,113.63) .. (344.74,115.39) .. controls (346.27,116.93) and (347.67,116.8) .. (351.8,116.07) ;
\draw [shift={(353.74,115.72)}, rotate = 169.99] [color={rgb, 255:red, 0; green, 0; blue, 0 }  ][line width=0.75]    (4.37,-1.32) .. controls (2.78,-0.56) and (1.32,-0.12) .. (0,0) .. controls (1.32,0.12) and (2.78,0.56) .. (4.37,1.32)   ;
\draw    (32.17,94.9) -- (44.84,94.9) ;
\draw [shift={(44.84,94.9)}, rotate = 0] [color={rgb, 255:red, 0; green, 0; blue, 0 }  ][fill={rgb, 255:red, 0; green, 0; blue, 0 }  ][line width=0.75]      (0, 0) circle [x radius= 1.34, y radius= 1.34]   ;
\draw    (44.84,94.9) -- (61.18,94.9) ;
\draw [shift={(61.18,94.9)}, rotate = 0] [color={rgb, 255:red, 0; green, 0; blue, 0 }  ][fill={rgb, 255:red, 0; green, 0; blue, 0 }  ][line width=0.75]      (0, 0) circle [x radius= 1.34, y radius= 1.34]   ;
\draw [shift={(44.84,94.9)}, rotate = 0] [color={rgb, 255:red, 0; green, 0; blue, 0 }  ][fill={rgb, 255:red, 0; green, 0; blue, 0 }  ][line width=0.75]      (0, 0) circle [x radius= 1.34, y radius= 1.34]   ;
\draw    (61.18,94.9) -- (77.52,94.9) ;
\draw [shift={(77.52,94.9)}, rotate = 0] [color={rgb, 255:red, 0; green, 0; blue, 0 }  ][fill={rgb, 255:red, 0; green, 0; blue, 0 }  ][line width=0.75]      (0, 0) circle [x radius= 1.34, y radius= 1.34]   ;
\draw [shift={(61.18,94.9)}, rotate = 0] [color={rgb, 255:red, 0; green, 0; blue, 0 }  ][fill={rgb, 255:red, 0; green, 0; blue, 0 }  ][line width=0.75]      (0, 0) circle [x radius= 1.34, y radius= 1.34]   ;
\draw    (77.52,94.9) -- (93.86,94.9) ;
\draw [shift={(93.86,94.9)}, rotate = 0] [color={rgb, 255:red, 0; green, 0; blue, 0 }  ][fill={rgb, 255:red, 0; green, 0; blue, 0 }  ][line width=0.75]      (0, 0) circle [x radius= 1.34, y radius= 1.34]   ;
\draw [shift={(77.52,94.9)}, rotate = 0] [color={rgb, 255:red, 0; green, 0; blue, 0 }  ][fill={rgb, 255:red, 0; green, 0; blue, 0 }  ][line width=0.75]      (0, 0) circle [x radius= 1.34, y radius= 1.34]   ;
\draw    (93.86,94.9) -- (110.2,94.9) ;
\draw [shift={(110.2,94.9)}, rotate = 0] [color={rgb, 255:red, 0; green, 0; blue, 0 }  ][fill={rgb, 255:red, 0; green, 0; blue, 0 }  ][line width=0.75]      (0, 0) circle [x radius= 1.34, y radius= 1.34]   ;
\draw [shift={(93.86,94.9)}, rotate = 0] [color={rgb, 255:red, 0; green, 0; blue, 0 }  ][fill={rgb, 255:red, 0; green, 0; blue, 0 }  ][line width=0.75]      (0, 0) circle [x radius= 1.34, y radius= 1.34]   ;
\draw    (110.2,94.9) -- (126.55,94.9) ;
\draw [shift={(126.55,94.9)}, rotate = 0] [color={rgb, 255:red, 0; green, 0; blue, 0 }  ][fill={rgb, 255:red, 0; green, 0; blue, 0 }  ][line width=0.75]      (0, 0) circle [x radius= 1.34, y radius= 1.34]   ;
\draw [shift={(110.2,94.9)}, rotate = 0] [color={rgb, 255:red, 0; green, 0; blue, 0 }  ][fill={rgb, 255:red, 0; green, 0; blue, 0 }  ][line width=0.75]      (0, 0) circle [x radius= 1.34, y radius= 1.34]   ;
\draw    (126.55,94.9) -- (142.89,94.9) ;
\draw [shift={(142.89,94.9)}, rotate = 0] [color={rgb, 255:red, 0; green, 0; blue, 0 }  ][fill={rgb, 255:red, 0; green, 0; blue, 0 }  ][line width=0.75]      (0, 0) circle [x radius= 1.34, y radius= 1.34]   ;
\draw [shift={(126.55,94.9)}, rotate = 0] [color={rgb, 255:red, 0; green, 0; blue, 0 }  ][fill={rgb, 255:red, 0; green, 0; blue, 0 }  ][line width=0.75]      (0, 0) circle [x radius= 1.34, y radius= 1.34]   ;
\draw    (142.89,94.9) -- (159.23,94.9) ;
\draw [shift={(159.23,94.9)}, rotate = 0] [color={rgb, 255:red, 0; green, 0; blue, 0 }  ][fill={rgb, 255:red, 0; green, 0; blue, 0 }  ][line width=0.75]      (0, 0) circle [x radius= 1.34, y radius= 1.34]   ;
\draw [shift={(142.89,94.9)}, rotate = 0] [color={rgb, 255:red, 0; green, 0; blue, 0 }  ][fill={rgb, 255:red, 0; green, 0; blue, 0 }  ][line width=0.75]      (0, 0) circle [x radius= 1.34, y radius= 1.34]   ;
\draw    (159.23,94.9) -- (175.57,94.9) ;
\draw [shift={(175.57,94.9)}, rotate = 0] [color={rgb, 255:red, 0; green, 0; blue, 0 }  ][fill={rgb, 255:red, 0; green, 0; blue, 0 }  ][line width=0.75]      (0, 0) circle [x radius= 1.34, y radius= 1.34]   ;
\draw [shift={(159.23,94.9)}, rotate = 0] [color={rgb, 255:red, 0; green, 0; blue, 0 }  ][fill={rgb, 255:red, 0; green, 0; blue, 0 }  ][line width=0.75]      (0, 0) circle [x radius= 1.34, y radius= 1.34]   ;
\draw    (175.57,94.9) -- (191.91,94.9) ;
\draw [shift={(191.91,94.9)}, rotate = 0] [color={rgb, 255:red, 0; green, 0; blue, 0 }  ][fill={rgb, 255:red, 0; green, 0; blue, 0 }  ][line width=0.75]      (0, 0) circle [x radius= 1.34, y radius= 1.34]   ;
\draw [shift={(175.57,94.9)}, rotate = 0] [color={rgb, 255:red, 0; green, 0; blue, 0 }  ][fill={rgb, 255:red, 0; green, 0; blue, 0 }  ][line width=0.75]      (0, 0) circle [x radius= 1.34, y radius= 1.34]   ;
\draw    (191.91,94.9) -- (208.25,94.9) ;
\draw [shift={(208.25,94.9)}, rotate = 0] [color={rgb, 255:red, 0; green, 0; blue, 0 }  ][fill={rgb, 255:red, 0; green, 0; blue, 0 }  ][line width=0.75]      (0, 0) circle [x radius= 1.34, y radius= 1.34]   ;
\draw [shift={(191.91,94.9)}, rotate = 0] [color={rgb, 255:red, 0; green, 0; blue, 0 }  ][fill={rgb, 255:red, 0; green, 0; blue, 0 }  ][line width=0.75]      (0, 0) circle [x radius= 1.34, y radius= 1.34]   ;
\draw    (32.5,134.57) -- (44.51,134.4) ;
\draw [shift={(44.51,134.4)}, rotate = 359.2] [color={rgb, 255:red, 0; green, 0; blue, 0 }  ][fill={rgb, 255:red, 0; green, 0; blue, 0 }  ][line width=0.75]      (0, 0) circle [x radius= 1.34, y radius= 1.34]   ;
\draw    (44.51,134.4) -- (60.85,134.4) ;
\draw [shift={(60.85,134.4)}, rotate = 0] [color={rgb, 255:red, 0; green, 0; blue, 0 }  ][fill={rgb, 255:red, 0; green, 0; blue, 0 }  ][line width=0.75]      (0, 0) circle [x radius= 1.34, y radius= 1.34]   ;
\draw [shift={(44.51,134.4)}, rotate = 0] [color={rgb, 255:red, 0; green, 0; blue, 0 }  ][fill={rgb, 255:red, 0; green, 0; blue, 0 }  ][line width=0.75]      (0, 0) circle [x radius= 1.34, y radius= 1.34]   ;
\draw    (60.85,134.4) -- (77.19,134.4) ;
\draw [shift={(77.19,134.4)}, rotate = 0] [color={rgb, 255:red, 0; green, 0; blue, 0 }  ][fill={rgb, 255:red, 0; green, 0; blue, 0 }  ][line width=0.75]      (0, 0) circle [x radius= 1.34, y radius= 1.34]   ;
\draw [shift={(60.85,134.4)}, rotate = 0] [color={rgb, 255:red, 0; green, 0; blue, 0 }  ][fill={rgb, 255:red, 0; green, 0; blue, 0 }  ][line width=0.75]      (0, 0) circle [x radius= 1.34, y radius= 1.34]   ;
\draw    (77.19,134.4) -- (93.53,134.4) ;
\draw [shift={(93.53,134.4)}, rotate = 0] [color={rgb, 255:red, 0; green, 0; blue, 0 }  ][fill={rgb, 255:red, 0; green, 0; blue, 0 }  ][line width=0.75]      (0, 0) circle [x radius= 1.34, y radius= 1.34]   ;
\draw [shift={(77.19,134.4)}, rotate = 0] [color={rgb, 255:red, 0; green, 0; blue, 0 }  ][fill={rgb, 255:red, 0; green, 0; blue, 0 }  ][line width=0.75]      (0, 0) circle [x radius= 1.34, y radius= 1.34]   ;
\draw    (93.53,134.4) -- (109.87,134.4) ;
\draw [shift={(109.87,134.4)}, rotate = 0] [color={rgb, 255:red, 0; green, 0; blue, 0 }  ][fill={rgb, 255:red, 0; green, 0; blue, 0 }  ][line width=0.75]      (0, 0) circle [x radius= 1.34, y radius= 1.34]   ;
\draw [shift={(93.53,134.4)}, rotate = 0] [color={rgb, 255:red, 0; green, 0; blue, 0 }  ][fill={rgb, 255:red, 0; green, 0; blue, 0 }  ][line width=0.75]      (0, 0) circle [x radius= 1.34, y radius= 1.34]   ;
\draw    (109.87,134.4) -- (126.21,134.4) ;
\draw [shift={(126.21,134.4)}, rotate = 0] [color={rgb, 255:red, 0; green, 0; blue, 0 }  ][fill={rgb, 255:red, 0; green, 0; blue, 0 }  ][line width=0.75]      (0, 0) circle [x radius= 1.34, y radius= 1.34]   ;
\draw [shift={(109.87,134.4)}, rotate = 0] [color={rgb, 255:red, 0; green, 0; blue, 0 }  ][fill={rgb, 255:red, 0; green, 0; blue, 0 }  ][line width=0.75]      (0, 0) circle [x radius= 1.34, y radius= 1.34]   ;
\draw    (126.21,134.4) -- (142.55,134.4) ;
\draw [shift={(142.55,134.4)}, rotate = 0] [color={rgb, 255:red, 0; green, 0; blue, 0 }  ][fill={rgb, 255:red, 0; green, 0; blue, 0 }  ][line width=0.75]      (0, 0) circle [x radius= 1.34, y radius= 1.34]   ;
\draw [shift={(126.21,134.4)}, rotate = 0] [color={rgb, 255:red, 0; green, 0; blue, 0 }  ][fill={rgb, 255:red, 0; green, 0; blue, 0 }  ][line width=0.75]      (0, 0) circle [x radius= 1.34, y radius= 1.34]   ;
\draw    (142.55,134.4) -- (158.89,134.4) ;
\draw [shift={(158.89,134.4)}, rotate = 0] [color={rgb, 255:red, 0; green, 0; blue, 0 }  ][fill={rgb, 255:red, 0; green, 0; blue, 0 }  ][line width=0.75]      (0, 0) circle [x radius= 1.34, y radius= 1.34]   ;
\draw [shift={(142.55,134.4)}, rotate = 0] [color={rgb, 255:red, 0; green, 0; blue, 0 }  ][fill={rgb, 255:red, 0; green, 0; blue, 0 }  ][line width=0.75]      (0, 0) circle [x radius= 1.34, y radius= 1.34]   ;
\draw    (158.89,134.4) -- (175.23,134.4) ;
\draw [shift={(175.23,134.4)}, rotate = 0] [color={rgb, 255:red, 0; green, 0; blue, 0 }  ][fill={rgb, 255:red, 0; green, 0; blue, 0 }  ][line width=0.75]      (0, 0) circle [x radius= 1.34, y radius= 1.34]   ;
\draw [shift={(158.89,134.4)}, rotate = 0] [color={rgb, 255:red, 0; green, 0; blue, 0 }  ][fill={rgb, 255:red, 0; green, 0; blue, 0 }  ][line width=0.75]      (0, 0) circle [x radius= 1.34, y radius= 1.34]   ;
\draw    (175.23,134.4) -- (191.58,134.4) ;
\draw [shift={(191.58,134.4)}, rotate = 0] [color={rgb, 255:red, 0; green, 0; blue, 0 }  ][fill={rgb, 255:red, 0; green, 0; blue, 0 }  ][line width=0.75]      (0, 0) circle [x radius= 1.34, y radius= 1.34]   ;
\draw [shift={(175.23,134.4)}, rotate = 0] [color={rgb, 255:red, 0; green, 0; blue, 0 }  ][fill={rgb, 255:red, 0; green, 0; blue, 0 }  ][line width=0.75]      (0, 0) circle [x radius= 1.34, y radius= 1.34]   ;
\draw    (191.58,134.4) -- (207.92,134.4) ;
\draw [shift={(207.92,134.4)}, rotate = 0] [color={rgb, 255:red, 0; green, 0; blue, 0 }  ][fill={rgb, 255:red, 0; green, 0; blue, 0 }  ][line width=0.75]      (0, 0) circle [x radius= 1.34, y radius= 1.34]   ;
\draw [shift={(191.58,134.4)}, rotate = 0] [color={rgb, 255:red, 0; green, 0; blue, 0 }  ][fill={rgb, 255:red, 0; green, 0; blue, 0 }  ][line width=0.75]      (0, 0) circle [x radius= 1.34, y radius= 1.34]   ;
\draw   (77.17,121.23) .. controls (77.17,120.87) and (77.47,120.57) .. (77.83,120.57) .. controls (78.2,120.57) and (78.5,120.87) .. (78.5,121.23) .. controls (78.5,121.6) and (78.2,121.9) .. (77.83,121.9) .. controls (77.47,121.9) and (77.17,121.6) .. (77.17,121.23) -- cycle ;
\draw  [dash pattern={on 0.84pt off 2.51pt}]  (83.17,121.57) -- (118.67,121.57) ;
\draw [shift={(121.67,121.57)}, rotate = 180] [fill={rgb, 255:red, 0; green, 0; blue, 0 }  ][line width=0.08]  [draw opacity=0] (5.36,-2.57) -- (0,0) -- (5.36,2.57) -- (3.56,0) -- cycle    ;
\draw   (126.17,121.23) .. controls (126.17,120.87) and (126.47,120.57) .. (126.83,120.57) .. controls (127.2,120.57) and (127.5,120.87) .. (127.5,121.23) .. controls (127.5,121.6) and (127.2,121.9) .. (126.83,121.9) .. controls (126.47,121.9) and (126.17,121.6) .. (126.17,121.23) -- cycle ;
\draw    (208.25,94.9) -- (216.25,94.9) ;
\draw [shift={(208.25,94.9)}, rotate = 0] [color={rgb, 255:red, 0; green, 0; blue, 0 }  ][fill={rgb, 255:red, 0; green, 0; blue, 0 }  ][line width=0.75]      (0, 0) circle [x radius= 1.34, y radius= 1.34]   ;
\draw    (207.92,134.4) -- (215.75,134.4) ;
\draw [shift={(207.92,134.4)}, rotate = 0] [color={rgb, 255:red, 0; green, 0; blue, 0 }  ][fill={rgb, 255:red, 0; green, 0; blue, 0 }  ][line width=0.75]      (0, 0) circle [x radius= 1.34, y radius= 1.34]   ;
\draw   (142.24,108.07) .. controls (142.24,107.7) and (142.54,107.4) .. (142.91,107.4) .. controls (143.28,107.4) and (143.58,107.7) .. (143.58,108.07) .. controls (143.58,108.43) and (143.28,108.73) .. (142.91,108.73) .. controls (142.54,108.73) and (142.24,108.43) .. (142.24,108.07) -- cycle ;
\draw  [dash pattern={on 0.84pt off 2.51pt}]  (82.92,119.9) -- (134.98,109.49) ;
\draw [shift={(137.92,108.9)}, rotate = 168.69] [fill={rgb, 255:red, 0; green, 0; blue, 0 }  ][line width=0.08]  [draw opacity=0] (5.36,-2.57) -- (0,0) -- (5.36,2.57) -- (3.56,0) -- cycle    ;
\draw (290,167.23) node [anchor=north west][inner sep=0.75pt]  [font=\tiny]  {${\rm Cyl}_{3,3} /\langle \sigma\rangle $};
\draw (343.17,99.4) node [anchor=north west][inner sep=0.75pt]  [font=\tiny]  {$\sigma $};
\draw (233.42,103.73) node [anchor=north west][inner sep=0.75pt]    {$\cong $};
\draw (419.33,102.23) node [anchor=north west][inner sep=0.75pt]    {$\cong $};
\draw (480.75,164.65) node [anchor=north west][inner sep=0.75pt]  [font=\tiny]  {$\mathcal{D}_{2,3}$};
\draw (356.91,108.28) node [anchor=north west][inner sep=0.75pt]  [font=\tiny]  {$Rotate\ 180^{\circ }$};
\draw (205,110) node [anchor=north west][inner sep=0.75pt]    {$\cdots $};
\draw (99.37,140.8) node [anchor=north west][inner sep=0.75pt]  [font=\tiny]  {$( 0,0)$};
\draw (98.47,82.4) node [anchor=north west][inner sep=0.75pt]  [font=\tiny]  {$( 0,1)$};
\draw (51.67,117.8) node [anchor=north west][inner sep=0.75pt]  [font=\tiny]  {$( x,y)$};
\draw (104.17,104.47) node [anchor=north west][inner sep=0.75pt]  [font=\tiny]  {$\theta $};
\draw (96.17,122.8) node [anchor=north west][inner sep=0.75pt]  [font=\tiny]  {$\sigma _{3}$};
\draw (129.17,116.97) node [anchor=north west][inner sep=0.75pt]  [font=\tiny]  {$( x+3,y)$};
\draw (144.24,102.8) node [anchor=north west][inner sep=0.75pt]  [font=\tiny]  {$( -x,1-y)$};
\draw (70,163.9) node [anchor=north west][inner sep=0.75pt]  [font=\tiny]  {$\mathcal{S}=\widetilde{\mathcal{S}} /G\ with\ G=\langle \sigma _{3} ,\theta \rangle \ $};
\draw (10,110) node [anchor=north west][inner sep=0.75pt]    {$\cdots $};
\end{tikzpicture}
\end{adjustbox}
\end{figure}
Therefore, the mapping class groups $\mathcal{MG}(\mathcal{S})$, $\mathcal{MG}({\rm Cyl}_{n,n}/\langle\sigma\rangle)$, and $\mathcal{MG}(\mathcal{D}_{2,n})$ are isomorphic. Recall \cite[Theorem 4.13]{MR2946087} that the mapping class group of $\mathcal{D}_{2,n}$ is represented by
 \begin{equation}\label{equ 2221}
\langle r, s \mid s^2 = r^n, rs = sr \rangle,
\end{equation}
where $s$ maps one puncture to the other and $s^2$ is isotopic to a rotation of the boundary by an angle of $2\pi$ that fixes the punctures, and $r$ is a generator such that $r^n = s^2$. Now we present a set of generators for $\mathcal{MG}(\mathcal{S})$.  Consider two elements $\widetilde{\alpha}$ and $\widetilde{\beta}$ in $\mathcal{MG}(\widetilde{\mathcal{S}})$, where
\begin{align*}
\widetilde{\alpha}: \widetilde{\mathcal{S}} & \rightarrow \widetilde{\mathcal{S}} & \qquad \widetilde{\beta}: \widetilde{\mathcal{S}} & \rightarrow \widetilde{\mathcal{S}} \\
(x,y) & \mapsto (x+ny,y) & (x,y) & \mapsto (x+2y-1,y)
\end{align*}
 The canonical projection map $\pi: \widetilde{\mathcal{S}}\rightarrow\mathcal{S}$ induces two maps $\alpha$ and $\beta$ on $\mathcal{S}$  that make the following diagrams commute: 
\[\begin{tikzcd}[ampersand replacement=\&,cramped]
	{\widetilde{\mathcal{S}}} \& {\widetilde{\mathcal{S}}} \&\& {\widetilde{\mathcal{S}}} \& {\widetilde{\mathcal{S}}} \\
	{\mathcal{S}} \& {\mathcal{S}} \&\& {\mathcal{S}} \& {\mathcal{S}}
	\arrow["{\widetilde{\alpha}}", from=1-1, to=1-2]
	\arrow["\pi"', from=1-1, to=2-1]
	\arrow["\pi", from=1-2, to=2-2]
	\arrow["{\widetilde{\beta}}", from=1-4, to=1-5]
	\arrow["\pi"', from=1-4, to=2-4]
	\arrow["\pi", from=1-5, to=2-5]
	\arrow["\alpha"', from=2-1, to=2-2]
	\arrow["\beta"', from=2-4, to=2-5]
\end{tikzcd}\]
 Moreover, one can check that
 \begin{equation}\label{equ 222}
\alpha^2=\beta^n \quad\text{and}\quad \alpha \beta=\beta \alpha.
\end{equation}
 By comparing equation \eqref{equ 2221} and equation \eqref{equ 222}, we have $\mathcal{MG}(\mathcal{S})=\langle \alpha,\beta\rangle$. Note that distinct $G$-orbits in $\widetilde{\operatorname{Seg}(M)}$ naturally represent different homotopy classes of curves in $\mathcal{S}$ relative to their endpoints. We define the action of $\mathcal{MG} (\mathcal{S})$ on $\widetilde{\operatorname{Seg}(M)}$ as follows: 
 \[x\cdot\widetilde{[i,j]}= \widetilde{[i-t,j+t+kn]},\]  
for any $x=\alpha^k\beta^t\in \mathcal{MG} (\mathcal{S})$ and  $\widetilde{[i,j]}\in \widetilde{\operatorname{Seg}(M)}$, where $0\leq k \leq 1$ and $t\in \mathbb{Z}$. 
\begin{prop}\label{group iso}
 The automorphism group $\operatorname{Aut}(\Gamma, d)=\mathbb{L}/\mathbb{Z}(\vec{x}_1-\vec{x}_2)$ of $(\Gamma, d)$ is isomorphic to the mapping class group $\mathcal{MG}(\mathcal{S})$.
\end{prop}
\begin{proof}
By comparing equation \eqref{equ 111} and equation \eqref{equ 222}, we observe that the mapping 
\[
\begin{aligned}
\psi: \mathcal{MG} (\mathcal{S}) & \rightarrow \mathbb{L}/\mathbb{Z}(\vec{x}_1-\vec{x}_2) \\
\alpha & \mapsto \vec{x}_1 \\
\beta & \mapsto \vec{x}_3,
\end{aligned}
\]
 establishes a group isomorphism from $\mathcal{MG}(\mathcal{S})$ to $\mathbb{L}/\mathbb{Z}(\vec{x}_1-\vec{x}_2)$.
\end{proof}

Recall from Proposition \ref{correspondence} that $\phi$ is a bijection from  $\widetilde{\operatorname{Seg}(M)}$ to $\operatorname{ind}({\rm vect}^F\mbox{-}\mathbb{X})$.
\begin{prop}\label{MCG and L action}
We have the following commutative diagram 
\begin{figure}[H]
    \centering
\begin{tikzcd}
{\mathcal{MG} (\mathcal{S})} \arrow[r, "\psi"]                                     & \mathbb{L}/\mathbb{Z}(\vec{x}_1-\vec{x}_2)                                                    \\
\widetilde{\operatorname{Seg}(M)} \arrow[r, "\phi"] \arrow[loop, distance=3em, in=125, out=55] & \operatorname{ind}({\rm vect}^F\mbox{-}\mathbb{X}) \arrow[loop, distance=3em, in=125, out=55]
\end{tikzcd}
\end{figure}
\noindent in the sense:
\begin{equation}\label{compatible}
    \phi(x\cdot \widetilde{[i,j]}) = \psi(x)\cdot \phi(\widetilde{[i,j]}) \quad \text{for all } x \in \mathcal{MG} (\mathcal{S}) \text{ and } \widetilde{[i,j]} \in \widetilde{\operatorname{Seg}(M)}.
\end{equation}
\end{prop}
\begin{proof}
It suffices to consider the cases when \( x = \alpha \) and \( x = \beta \). When \( x = \alpha \),
\[
\begin{aligned}
\psi(x) \cdot \phi(\widetilde{[i,j]}) &=  \mathsf{E}_{L_0(-i\vec{x}_3)} \langle (i+j-1)\vec{x}_3 \rangle(\vec{x}_1) \\
&= \mathsf{E}_{L_0(\vec{x}_1 - i\vec{x}_3)} \langle (i+j-1)\vec{x}_3 \rangle \\
&= \mathsf{E}_{L_0(-i\vec{x}_3)} \langle (i+j-1)\vec{x}_3 + \vec{c} \rangle \\
&= \phi(\widetilde{[i,j+n]}) \\
&= \phi(x\cdot \widetilde{[i,j]}).
\end{aligned}
\]

When \( x = \beta\),
\[
\begin{aligned}
\psi(x) \cdot \phi(\widetilde{[i,j]}) &=  \mathsf{E}_{L_0(-i\vec{x}_3)} \langle (i+j-1)\vec{x}_3 \rangle (\vec{x}_3)\\
&= \phi(\widetilde{[i-1,j+1]}) \\
&= \phi(x\cdot \widetilde{[i,j]}).
\end{aligned}
\]
\end{proof}

\subsection{Refinement of the bijection $\phi$}\label{Section 4.3}
 Let $\mathcal{P}=\{P_k \mid k\in \mathbb{Z}\}$, where $P_k$ is the point whose coordinate is $(\frac{kn}{2},\frac{1}{2})$. Denote by $\operatorname{Seg}_0(M)$ the set of all line segments in $\operatorname{Seg}(M)$ that pass through points in $\mathcal{P}$. We prepare the following observation, which follows directly from Proposition \ref{correspondence1}. 

 \begin{obs}
      For $[i,j] \in \operatorname{Seg}(M)$, the following conditions are equivalent.
\begin{itemize}
    \item $[i,j]\in \operatorname{Seg}_0(M)$;
    \item $i+j\equiv 0\;(\bmod\; n)$;
    \item $\phi(\widetilde{[i,j]})$ is formed by the direct sum of two line bundles.
\end{itemize}
 \end{obs}

Specifically, for each segment $[i,j]$ in $\operatorname{Seg}_0(M)$, suppose $i+j=kn$ for some $k\in \mathbb{Z}$. We have
 \[\phi(\widetilde{[i,j]})=\mathsf{E}_{L_0(-i\vec{x}_3)}\langle (kn-1)\vec{x}_3 \rangle=L_0(k\vec{x}_1-(i+1)\vec{x}_3)\oplus L_0^*(k\vec{x}_1-(i+1)\vec{x}_3).\] 
 To refine the bijection $\phi$, a natural idea is to divide the segments in $\operatorname{Seg}_0(M)$ at their midpoints, constructing a new set of segments as follows:
\[
\operatorname{Seg}_0^*(M) =\{[i,j]^+,[i,j]^-\mid [i,j]\in \operatorname{Seg}_0(M)\},
\]
where $[i,j]^+$ represents the segment with endpoints $(i,0)$ and $(\frac{i+j}{2},\frac{1}{2})$, and $[i,j]^-$ represents the segment with endpoints $(j,1)$ and $(\frac{i+j}{2},\frac{1}{2})$. 

\begin{figure}[H]
    \centering

\tikzset{every picture/.style={line width=0.75pt}}          

\begin{tikzpicture}[x=0.75pt,y=0.75pt,yscale=-1,xscale=1]

\draw    (50.33,100) -- (230.08,100) ;
  
\draw    (50,139.5) -- (229.75,139.5) ;
  
\draw [color={rgb, 255:red, 0; green, 0; blue, 0 }  ,draw opacity=1 ]   (96.29,139.47) -- (169.96,100.47) ;
\draw [shift={(169.96,100.47)}, rotate = 332.1] [color={rgb, 255:red, 0; green, 0; blue, 0 }  ,draw opacity=1 ][fill={rgb, 255:red, 0; green, 0; blue, 0 }  ,fill opacity=1 ][line width=0.75]      (0, 0) circle [x radius= 1.34, y radius= 1.34]   ;
\draw [shift={(133.13,119.97)}, rotate = 332.1] [color={rgb, 255:red, 0; green, 0; blue, 0 }  ,draw opacity=1 ][fill={rgb, 255:red, 0; green, 0; blue, 0 }  ,fill opacity=1 ][line width=0.75]      (0, 0) circle [x radius= 1.34, y radius= 1.34]   ;
\draw [shift={(96.29,139.47)}, rotate = 332.1] [color={rgb, 255:red, 0; green, 0; blue, 0 }  ,draw opacity=1 ][fill={rgb, 255:red, 0; green, 0; blue, 0 }  ,fill opacity=1 ][line width=0.75]      (0, 0) circle [x radius= 1.34, y radius= 1.34]   ;
   
\draw   (130.4,117.6) .. controls (128.13,113.53) and (124.95,112.63) .. (120.88,114.9) -- (116.16,117.53) .. controls (110.34,120.78) and (106.29,120.36) .. (104.02,116.29) .. controls (106.29,120.36) and (104.52,124.03) .. (98.7,127.28)(101.32,125.82) -- (98.7,127.28) .. controls (94.63,129.55) and (93.73,132.73) .. (96,136.8) ;
   
\draw   (136.4,122) .. controls (138.59,126.12) and (141.75,127.08) .. (145.87,124.88) -- (151.87,121.68) .. controls (157.75,118.55) and (161.79,119.04) .. (163.99,123.15) .. controls (161.79,119.04) and (163.63,115.41) .. (169.52,112.27)(166.87,113.68) -- (169.52,112.27) .. controls (173.63,110.08) and (174.59,106.92) .. (172.4,102.8) ;

\draw (33.67,137.4) node [anchor=north west][inner sep=0.75pt]  [font=\tiny]  {$\partial ^{\prime }$};
 
\draw (34.67,97.4) node [anchor=north west][inner sep=0.75pt]  [font=\tiny]  {$\partial $};
 
\draw (89,145.73) node [anchor=north west][inner sep=0.75pt]  [font=\tiny]  {$( i,0)$};
 
\draw (163.33,87.73) node [anchor=north west][inner sep=0.75pt]  [font=\tiny]  {$( j,1)$};
 
\draw (84,105.4) node [anchor=north west][inner sep=0.75pt]  [font=\tiny]  {$[ i,j]^{+}$};
 
\draw (164.4,123) node [anchor=north west][inner sep=0.75pt]  [font=\tiny]  {$[ i,j]^{-}$};
 
\draw (270.27,111.6) node [anchor=north west][inner sep=0.75pt]  [font=\normalsize]  {$i+j\equiv 0\ (\bmod\; n)$};

\end{tikzpicture}
    \caption{Dividing the segment $\operatorname{Seg}_0(M)$ into two halves}
\end{figure}
\noindent Let
$\operatorname{Seg}^*(M) = (\operatorname{Seg}(M) \setminus \operatorname{Seg}_0(M)) \cup \operatorname{Seg}_0^*(M).$  The $G$-action on $\operatorname{Seg}(M)$ induces a $G$-action on $\operatorname{Seg}^*(M)$ by \[g\circ[i,j]^*=(g\cdot [i,j])^*,\]
  for any $g\in G$, where the superscript $*$ is taken from the set $\{+,-,\text{empty}\}$. 
  
  Denote by $\widetilde{\operatorname{Seg}^*(M)}$ the set of $G$-orbits of segments in $\operatorname{Seg}^*(M)$. 
   \begin{prop}\label{correspondence}
    The bijection $\phi$ induces a bijection $\widehat{\phi}:\widetilde{\operatorname{Seg}^*(M)}\rightarrow\operatorname{ind}({\rm vect}\mbox{-}\mathbb{X})$. 
\end{prop}
  \begin{proof}
       We define the map 
   \[\widehat{\phi}:\widetilde{\operatorname{Seg}^*(M)} \rightarrow \operatorname{ind}({\rm vect}\mbox{-}\mathbb{X})\] by the formula 
   \[\widehat{\phi}(\widetilde{[i,j]})= \phi(\widetilde{[i,j]})=\mathsf{E}_{L_0(-i\vec{x}_3)}\langle (i+j-1)\vec{x}_3 \rangle,\] for $[i,j] \notin \operatorname{Seg}_0(M)$; and 
\[
\widehat{\phi}(\widetilde{[i,j]^+})= \begin{cases}
L_0(k\vec{x}_1-(i+1)\vec{x}_3), & \text{if $k\in \mathbb{Z}$ is odd},\\ 
L_0^*(k\vec{x}_1-(i+1)\vec{x}_3), & \text{if $k\in \mathbb{Z}$ is even}.
\end{cases}
\]
\[
\widehat{\phi}(\widetilde{[i,j]^-})= \begin{cases}
L_0^*(k\vec{x}_1-(i+1)\vec{x}_3), & \text{if $k\in \mathbb{Z}$ is odd},\\ 
L_0(k\vec{x}_1-(i+1)\vec{x}_3), & \text{if $k\in \mathbb{Z}$ is even}.
\end{cases}
\]
for $[i,j] \in \operatorname{Seg}_0(M)$, where $k=\frac{1}{n}(i+j)$. Since the mapping $\widehat{\phi}$ can be viewed as a slight modification of $\phi$, it is easy to see that $\widehat{\phi}$ is a well-defined bijection.
  \end{proof}
  \begin{rem}
   (i) By choosing appropriate representatives for $G$-orbits in $\widetilde{\operatorname{Seg}^*(M)}$, the bijection $\widehat{\phi}$ can be explicitly described by the following table:
		\begin{center}
\begin{tabular}{|c|c|}
\hline
\text{$G$-orbits in $\widetilde{\operatorname{Seg}^*(M)}$} & \text{Indecomposable objects in ${\rm vect}\mbox{-}\mathbb{X} $} \\
\hline
$\widetilde{[i,-i]^{+}}$ & $L_0^*(-(i+1)\vec{x}_3)$ \\
$\widetilde{[i,-i]^{-}}$ & $L_0(-(i+1)\vec{x}_3)$ \\
$\widetilde{[i,n-i]^{+}}$ & $L_0(\vec{x}_1-(i+1)\vec{x}_3)$ \\
$\widetilde{[i,n-i]^{-}}$ & $L_0^*(\vec{x}_1-(i+1)\vec{x}_3)$ \\
$\widetilde{[i,k-i]}$ & $E_{L_0(-i\vec{x}_3)}\langle (k-1)\vec{x}_3 \rangle$ \\
\hline
\end{tabular}
\end{center}
   where $i,k$ are integers with $1\leq k\leq n-1$. \\
    (ii) In the later, the bijection $\widehat{\phi}$ in Proposition \ref{correspondence} will be frequently referenced without further elaboration. By abuse of notation, for a line segment $[i,j]\in \operatorname{Seg}_0(M)$, we will denote $\widehat{\phi}(\widetilde{[i,j]^+})\oplus \widehat{\phi}(\widetilde{[i,j]^-})$ as $\widehat{\phi}(\widetilde{[i,j]})$.
\end{rem}

In the rest of this section, let $X$ be an indecomposable bundle in ${\rm vect}\mbox{-}\mathbb{X}$. Define 
\[\begin{aligned}
\varrho: \{+, -, \text{empty}\} & \rightarrow \{+, -, \text{empty}\} \\
+ & \mapsto - \\
- & \mapsto + \\
\text{empty} & \mapsto \text{empty}
\end{aligned}\]
as a permutation on the set $\{+,-,\text{empty}\}$.
\subsection{ Slopes of indecomposable bundles}
For any $\widetilde{[i,j]^*}\in \widetilde{\operatorname{Seg}^*(M)}$, since line segments in $\widetilde{[i,j]^*}$ have the same slope, we define the \emph{slope} of $\widetilde{[i,j]^*}$ by $\mu(\widetilde{[i,j]^*})=\frac{1}{j-i}$. The following observation is useful to calculate the  slope $\mu X$ of $X$. 
\begin{prop}\label{slop}
    Assume $X=\widehat{\phi}(\widetilde{[i,j]^*})$. The slope $\mu X$ of $X$ is given by 
    \[\mu X=(\frac{1}{\mu(\widetilde{[i,j]^*})}-2)\times\frac{\bar{p}}{2n}+\mu L_0,\]
    where $\bar{p}:=\operatorname{lcm}\left(2, n\right)$ and $\frac{1}{\infty}=0$ is defined. In particular, if $L_0$ is chosen as $\mathcal{O}(-\vec{\omega})$ or $\mathcal{O}(\vec{x}_3)$, we have
     \[\mu X=\frac{1}{\mu(\widetilde{[i,j]^*})}\times\frac{\bar{p}}{2n}.\]
\end{prop}
\begin{proof}
Firstly, we have $\deg L_0=\deg L_0^*$ and $\deg L_0=\mu L_0$. We discuss case by case based on the types of $*$ as follows:
\begin{itemize}
\item[\emph{Case 1:}] If $*$ is empty, then $X= \mathsf{E}_{L_0(-i\vec{x}_3)}\langle (i+j-1)\vec{x}_3 \rangle$ of rank two. By additivity of the degree function, we have
\[
\begin{aligned}
    \mu X=\operatorname{deg} X / \mathrm{rk} X=&\frac{1}{2}(\deg L_0(-i\vec{x}_3+\vec{\omega})+\deg L_0((j-1)\vec{x}_3))\\
                           =&(j-i-2)\times\frac{\bar{p}}{2n}+\mu L_0\\
                           =&(\frac{1}{\mu(\widetilde{[i,j]^*})}-2)\times\frac{\bar{p}}{2n}+\mu L_0;
\end{aligned}
 \]
\item[\emph{Case 2:}] If $*$ is $+$ or $-$, then $X$ is a line bundle satisfying
 \[
\begin{aligned}
    \mu X=\operatorname{deg} X / \mathrm{rk} X=&\deg L_0(\frac{i+j}{n}\vec{x}_1-(i+1)\vec{x}_3)\\
                           =&(\frac{i+j}{n}\times \frac{\bar{p}}{2}-(i+1)\times\frac{\bar{p}}{n})+\deg L_0\\
                           =&(j-i-2)\times\frac{\bar{p}}{2n}+\mu L_0\\
                           =&(\frac{1}{\mu(\widetilde{[i,j]^*})}-2)\times\frac{\bar{p}}{2n}+\mu L_0.
\end{aligned}
 \]
\end{itemize}
  
  The remaining part of the statement follows from the fact that 
  \[\mu\;\mathcal{O}(-\vec{\omega})=\mu\; \mathcal{O}(\vec{x}_3)=\frac{\bar{p}}{n}.\] 
 \end{proof}
 \subsection{The $\mathbb{L}$-action on ${\rm vect}\mbox{-}\mathbb{X}$}
Recall that the Picard group $\mathbb{L}$ is an abelian group generated by $\vec{x}_1, \vec{x}_2, \vec{x}_3$, and acts on ${\rm vect}\mbox{-}\mathbb{X}$ by the shift $X\mapsto X(\vec{x})$, where $\vec{x}\in \mathbb{L}$. To obtain the $G$-orbit $\widehat{\phi}^{-1}(X(\vec{x}))$ in $\widetilde{\operatorname{Seg}^*(M)}$ corresponding to $X(\vec{x})$, the following observation is crucial.
\begin{cor}{}{}\label{L-action}
    Assume $X=\widehat{\phi}(\widetilde{[i,j]^*})$. Then we have
    \begin{itemize}
        \item $X(\vec{x}_1)=\widehat{\phi}(\widetilde{[i,j+n]^{\varrho(*)}})$;
        \item $X(\vec{x}_2)=\widehat{\phi}(\widetilde{[i,j+n]^*})$;
        \item $X(\vec{x}_3)=\widehat{\phi}(\widetilde{[i-1,j+1]^*})$.
    \end{itemize}
\end{cor}
\begin{proof}
 By employing the ``folding" relationship between the Auslander-Reiten quiver $\Gamma({\rm vect}\mbox{-}\mathbb{X})$ and the  valued translation quiver $(\Gamma, d)$, the assertions directly follow from Proposition \ref{MCG and L action} and Proposition \ref{correspondence}.
\end{proof}
 Therefore, for any $\vec{x}\in \mathbb{L}$, to obtain $\widehat{\phi}^{-1}(X(\vec{x}))$, one can apply a series of operations (or their inverses) to $\widetilde{[i,j]^*}$ as shown in Corollary \ref{L-action}. For example, let $\vec{x}=-\vec{\omega}=\vec{x}_1-\vec{x}_2+\vec{x}_3$, then $\tau^{-1} X=X(-\vec{\omega})=\widehat{\phi}(\widetilde{[i-1,j+1]^{\varrho(*)}})$. 

\subsection{Vector bundle duality}
Recall that vector bundle duality $^\vee: {\rm vect}\mbox{-}\mathbb{X}\rightarrow{\rm vect}\mbox{-}\mathbb{X}, X \mapsto
\mathcal{H}om(X, \mathcal{O})$, sends line bundles to line bundles, and preserves distinguished exact sequences. We give a general version of Lemma 4.3 in \cite{MR3028577}.
\begin{lem}\label{Vector bundle duality}
    Assume $\vec{x}=l_3\vec{x}_3+l\vec{c} \in  \mathbb{L}$ is  written in normal form. Then
$$
\mathsf{E}_L\langle\vec{x}\rangle^\vee=\mathsf{E}_{L^\vee(-(\vec{x}+\vec{\omega}))}\langle\vec{x}\rangle.
$$ 
In particular, for $L=\mathcal{O}$, we have $\mathsf{E}\langle\vec{x}\rangle^\vee=\mathsf{E}\langle\vec{x}\rangle(-(\vec{x}+\vec{\omega}))$.
\end{lem}
\begin{proof}
For $l_3\neq n-1$, by Lemma 4.3 in \cite{MR3028577}, we have 
\begin{align*}
    \mathsf{E}_L\langle\vec{x}\rangle^\vee &= \mathsf{E}_{L(l\vec{x}_1)^\vee(-(l_3\vec{x}_3+\vec{\omega}))}\langle l_3\vec{x}_3\rangle \\
    &= \mathsf{E}_{L^\vee(-(l_3\vec{x}_3+l\vec{x}_1+\vec{\omega}))}\langle l_3\vec{x}_3\rangle \\
    &= \mathsf{E}_{L^\vee(-(\vec{x}+\vec{\omega}))}\langle \vec{x}\rangle;
\end{align*}
for $l_3=n-1$, we have 
\begin{align*}
    \mathsf{E}_L\langle\vec{x}\rangle^\vee &= L((l+1)\vec{x}_1-\vec{x}_3)^\vee \oplus L^*((l+1)\vec{x}_1-\vec{x}_3)^\vee \\
    &= L^\vee((l+1)\vec{x}_1-\vec{x}_3) \oplus (L^\vee)^*((l+1)\vec{x}_1-\vec{x}_3) \\
    &= \mathsf{E}_{L^\vee(-(l_3\vec{x}_3+l\vec{x}_1+\vec{\omega}))}\langle l_3\vec{x}_3\rangle \\
    &= \mathsf{E}_{L^\vee(-(\vec{x}+\vec{\omega}))}\langle \vec{x}\rangle. 
\end{align*}
\end{proof}

 The vector bundle duality $X^\vee$ of $X$ is given by
\begin{prop}\label{vector bundle duality}
    Assume $X=\widehat{\phi}(\widetilde{[i,j]^*})$. Then we have
   \[X^\vee=\widehat{\phi}(\widetilde{[j,i]^*}).\]
  Moreover, $X$ is fixed under $^\vee$ if and only if $i=j$. 
\end{prop}
\begin{proof}
     The assertion follows from Lemma \ref{Vector bundle duality} and Proposition \ref{correspondence}. Now, suppose $X=X^\vee$. Then there exists $g\in G$ such that $g\circ [i,j]^*=[j,i]^*$. Since any $g$ in $G$ is always of the form $\sigma_n^k$ or $\sigma_n^k\theta$ for some $k\in \mathbb{Z}$, we can deduce that $i=j$.
\end{proof}

\section{Intersection index and extension dimension}\label{Section 5}
In this section, we interpret the dimension of extension space between two indecomposable bundles in ${\rm vect}\mbox{-}\mathbb{X} $ is equal to the intersection index between their corresponding $G$-obrits in $\widetilde{\operatorname{Seg}^*(M)}$.

\subsection{Positive intersections}
 Let  $[I, J]$ and $[S, T]$ be line segments in $\widetilde{\mathcal{S}}$.
\begin{defn}
    A point $x$ of intersection of $[I, J]$ and $[S, T]$ is
called \emph{positive intersection} of $[I, J]$ and $[S, T]$  if the following conditions are satisfied:
\begin{itemize}
    \item the point $x$ is not an endpoint of $[I, J]$ or $[S, T]$;
    \item the reciprocal of the slope of $[S, T ]$ is less than that of $[I, J]$.
\end{itemize}
\noindent  If $x$ is a positive intersection of $[I, J]$ and $[S, T]$, then the picture looks like this:
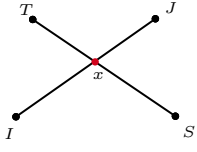
\begin{figure}[H]

\tikzset{every picture/.style={line width=0.75pt}}          

\begin{tikzpicture}[x=0.75pt,y=0.75pt,yscale=-1,xscale=1]

\draw [color={rgb, 255:red, 0; green, 0; blue, 0 }  ,draw opacity=1 ]   (70.47,79.95) -- (140.07,30.65) ;
\draw [shift={(140.07,30.65)}, rotate = 324.69] [color={rgb, 255:red, 0; green, 0; blue, 0 }  ,draw opacity=1 ][fill={rgb, 255:red, 0; green, 0; blue, 0 }  ,fill opacity=1 ][line width=0.75]      (0, 0) circle [x radius= 1.34, y radius= 1.34]   ;
\draw [shift={(70.47,79.95)}, rotate = 324.69] [color={rgb, 255:red, 0; green, 0; blue, 0 }  ,draw opacity=1 ][fill={rgb, 255:red, 0; green, 0; blue, 0 }  ,fill opacity=1 ][line width=0.75]      (0, 0) circle [x radius= 1.34, y radius= 1.34]   ;
  
\draw [color={rgb, 255:red, 0; green, 0; blue, 0 }  ,draw opacity=1 ]   (78.83,31) -- (150.07,79.65) ;
\draw [shift={(150.07,79.65)}, rotate = 34.33] [color={rgb, 255:red, 0; green, 0; blue, 0 }  ,draw opacity=1 ][fill={rgb, 255:red, 0; green, 0; blue, 0 }  ,fill opacity=1 ][line width=0.75]      (0, 0) circle [x radius= 1.34, y radius= 1.34]   ;
\draw [shift={(78.83,31)}, rotate = 34.33] [color={rgb, 255:red, 0; green, 0; blue, 0 }  ,draw opacity=1 ][fill={rgb, 255:red, 0; green, 0; blue, 0 }  ,fill opacity=1 ][line width=0.75]      (0, 0) circle [x radius= 1.34, y radius= 1.34]   ;
  
\draw    (150.07,79.65) ;
\draw [shift={(150.07,79.65)}, rotate = 0] [color={rgb, 255:red, 0; green, 0; blue, 0 }  ][fill={rgb, 255:red, 0; green, 0; blue, 0 }  ][line width=0.75]      (0, 0) circle [x radius= 1.34, y radius= 1.34]   ;
\draw [shift={(150.07,79.65)}, rotate = 0] [color={rgb, 255:red, 0; green, 0; blue, 0 }  ][fill={rgb, 255:red, 0; green, 0; blue, 0 }  ][line width=0.75]      (0, 0) circle [x radius= 1.34, y radius= 1.34]   ;
  
\draw    (78.83,31) ;
\draw [shift={(78.83,31)}, rotate = 0] [color={rgb, 255:red, 0; green, 0; blue, 0 }  ][fill={rgb, 255:red, 0; green, 0; blue, 0 }  ][line width=0.75]      (0, 0) circle [x radius= 1.34, y radius= 1.34]   ;
\draw [shift={(78.83,31)}, rotate = 0] [color={rgb, 255:red, 0; green, 0; blue, 0 }  ][fill={rgb, 255:red, 0; green, 0; blue, 0 }  ][line width=0.75]      (0, 0) circle [x radius= 1.34, y radius= 1.34]   ;
  
\draw    (70.47,79.95) ;
\draw [shift={(70.47,79.95)}, rotate = 0] [color={rgb, 255:red, 0; green, 0; blue, 0 }  ][fill={rgb, 255:red, 0; green, 0; blue, 0 }  ][line width=0.75]      (0, 0) circle [x radius= 1.34, y radius= 1.34]   ;
\draw [shift={(70.47,79.95)}, rotate = 0] [color={rgb, 255:red, 0; green, 0; blue, 0 }  ][fill={rgb, 255:red, 0; green, 0; blue, 0 }  ][line width=0.75]      (0, 0) circle [x radius= 1.34, y radius= 1.34]   ;
  
\draw    (140.07,30.65) ;
\draw [shift={(140.07,30.65)}, rotate = 0] [color={rgb, 255:red, 0; green, 0; blue, 0 }  ][fill={rgb, 255:red, 0; green, 0; blue, 0 }  ][line width=0.75]      (0, 0) circle [x radius= 1.34, y radius= 1.34]   ;
\draw [shift={(140.07,30.65)}, rotate = 0] [color={rgb, 255:red, 0; green, 0; blue, 0 }  ][fill={rgb, 255:red, 0; green, 0; blue, 0 }  ][line width=0.75]      (0, 0) circle [x radius= 1.34, y radius= 1.34]   ;
  
\draw  [color={rgb, 255:red, 208; green, 2; blue, 27 }  ,draw opacity=1 ][fill={rgb, 255:red, 208; green, 2; blue, 27 }  ,fill opacity=1 ] (108.67,52.33) .. controls (108.67,51.6) and (109.26,51) .. (110,51) .. controls (110.74,51) and (111.33,51.6) .. (111.33,52.33) .. controls (111.33,53.07) and (110.74,53.67) .. (110,53.67) .. controls (109.26,53.67) and (108.67,53.07) .. (108.67,52.33) -- cycle ;

\draw (63.3,84.2) node [anchor=north west][inner sep=0.75pt]  [font=\tiny]  {$I$};
 
\draw (70.53,21.7) node [anchor=north west][inner sep=0.75pt]  [font=\tiny]  {$T$};
 
\draw (152.07,83.05) node [anchor=north west][inner sep=0.75pt]  [font=\tiny]  {$S$};
 
\draw (143.1,20.9) node [anchor=north west][inner sep=0.75pt]  [font=\tiny]  {$J$};
 
\draw (107.27,56.03) node [anchor=north west][inner sep=0.75pt]  [font=\tiny]  {$x$};

\end{tikzpicture}

    \caption{Positive intersection}
\end{figure}
\noindent Denote the set of the positive
intersections of $[I,J]$ and $[S,T]$ by $[I,J]\cap^+ [S,T]$.
\end{defn}

We show that for each positive intersection of $[I, J]$ and $[S, T]$, there exists an exact sequence associated with this positive intersection. 
\begin{prop}{}{}\label{intersection and exact seq}
   Assume	$[I, J], [S, T] \in \widetilde{\operatorname{Seg}^*(M)}$ and $[I, J] \cap^+ [S, T] \neq \emptyset$, as illustrated below
     
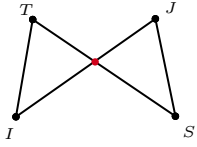
\begin{figure}[H]
\tikzset{every picture/.style={line width=0.75pt}}          

\begin{tikzpicture}[x=0.75pt,y=0.75pt,yscale=-1,xscale=1]

\draw [color={rgb, 255:red, 0; green, 0; blue, 0 }  ,draw opacity=1 ]   (140.07,30.65) -- (150.07,79.65) ;
  
\draw [color={rgb, 255:red, 0; green, 0; blue, 0 }  ,draw opacity=1 ]   (78.83,31) -- (70.47,79.95) ;
  
\draw [color={rgb, 255:red, 0; green, 0; blue, 0 }  ,draw opacity=1 ]   (70.47,79.95) -- (140.07,30.65) ;
\draw [shift={(140.07,30.65)}, rotate = 324.69] [color={rgb, 255:red, 0; green, 0; blue, 0 }  ,draw opacity=1 ][fill={rgb, 255:red, 0; green, 0; blue, 0 }  ,fill opacity=1 ][line width=0.75]      (0, 0) circle [x radius= 1.34, y radius= 1.34]   ;
\draw [shift={(70.47,79.95)}, rotate = 324.69] [color={rgb, 255:red, 0; green, 0; blue, 0 }  ,draw opacity=1 ][fill={rgb, 255:red, 0; green, 0; blue, 0 }  ,fill opacity=1 ][line width=0.75]      (0, 0) circle [x radius= 1.34, y radius= 1.34]   ;
  
\draw [color={rgb, 255:red, 0; green, 0; blue, 0 }  ,draw opacity=1 ]   (78.83,31) -- (150.07,79.65) ;
\draw [shift={(150.07,79.65)}, rotate = 34.33] [color={rgb, 255:red, 0; green, 0; blue, 0 }  ,draw opacity=1 ][fill={rgb, 255:red, 0; green, 0; blue, 0 }  ,fill opacity=1 ][line width=0.75]      (0, 0) circle [x radius= 1.34, y radius= 1.34]   ;
\draw [shift={(78.83,31)}, rotate = 34.33] [color={rgb, 255:red, 0; green, 0; blue, 0 }  ,draw opacity=1 ][fill={rgb, 255:red, 0; green, 0; blue, 0 }  ,fill opacity=1 ][line width=0.75]      (0, 0) circle [x radius= 1.34, y radius= 1.34]   ;
  
\draw    (150.07,79.65) ;
\draw [shift={(150.07,79.65)}, rotate = 0] [color={rgb, 255:red, 0; green, 0; blue, 0 }  ][fill={rgb, 255:red, 0; green, 0; blue, 0 }  ][line width=0.75]      (0, 0) circle [x radius= 1.34, y radius= 1.34]   ;
\draw [shift={(150.07,79.65)}, rotate = 0] [color={rgb, 255:red, 0; green, 0; blue, 0 }  ][fill={rgb, 255:red, 0; green, 0; blue, 0 }  ][line width=0.75]      (0, 0) circle [x radius= 1.34, y radius= 1.34]   ;
  
\draw    (78.83,31) ;
\draw [shift={(78.83,31)}, rotate = 0] [color={rgb, 255:red, 0; green, 0; blue, 0 }  ][fill={rgb, 255:red, 0; green, 0; blue, 0 }  ][line width=0.75]      (0, 0) circle [x radius= 1.34, y radius= 1.34]   ;
\draw [shift={(78.83,31)}, rotate = 0] [color={rgb, 255:red, 0; green, 0; blue, 0 }  ][fill={rgb, 255:red, 0; green, 0; blue, 0 }  ][line width=0.75]      (0, 0) circle [x radius= 1.34, y radius= 1.34]   ;
  
\draw    (70.47,79.95) ;
\draw [shift={(70.47,79.95)}, rotate = 0] [color={rgb, 255:red, 0; green, 0; blue, 0 }  ][fill={rgb, 255:red, 0; green, 0; blue, 0 }  ][line width=0.75]      (0, 0) circle [x radius= 1.34, y radius= 1.34]   ;
\draw [shift={(70.47,79.95)}, rotate = 0] [color={rgb, 255:red, 0; green, 0; blue, 0 }  ][fill={rgb, 255:red, 0; green, 0; blue, 0 }  ][line width=0.75]      (0, 0) circle [x radius= 1.34, y radius= 1.34]   ;
  
\draw    (140.07,30.65) ;
\draw [shift={(140.07,30.65)}, rotate = 0] [color={rgb, 255:red, 0; green, 0; blue, 0 }  ][fill={rgb, 255:red, 0; green, 0; blue, 0 }  ][line width=0.75]      (0, 0) circle [x radius= 1.34, y radius= 1.34]   ;
\draw [shift={(140.07,30.65)}, rotate = 0] [color={rgb, 255:red, 0; green, 0; blue, 0 }  ][fill={rgb, 255:red, 0; green, 0; blue, 0 }  ][line width=0.75]      (0, 0) circle [x radius= 1.34, y radius= 1.34]   ;
  
\draw  [color={rgb, 255:red, 208; green, 2; blue, 27 }  ,draw opacity=1 ][fill={rgb, 255:red, 208; green, 2; blue, 27 }  ,fill opacity=1 ] (108.67,52.33) .. controls (108.67,51.6) and (109.26,51) .. (110,51) .. controls (110.74,51) and (111.33,51.6) .. (111.33,52.33) .. controls (111.33,53.07) and (110.74,53.67) .. (110,53.67) .. controls (109.26,53.67) and (108.67,53.07) .. (108.67,52.33) -- cycle ;

\draw (63.3,84.2) node [anchor=north west][inner sep=0.75pt]  [font=\tiny]  {$I$};
 
\draw (70.53,21.7) node [anchor=north west][inner sep=0.75pt]  [font=\tiny]  {$T$};
 
\draw (152.07,83.05) node [anchor=north west][inner sep=0.75pt]  [font=\tiny]  {$S$};
 
\draw (143.1,20.9) node [anchor=north west][inner sep=0.75pt]  [font=\tiny]  {$J$};
\end{tikzpicture}

\label{intersection and exact sequence}
\end{figure}
\noindent Then there exists an exact sequence of the following form
\begin{equation}\label{exact and int}
    \begin{tikzcd}[ampersand replacement=\&,cramped,sep=small]
	0 \& {\widehat{\phi}(\widetilde{[S,T]})} \& {\widehat{\phi}(\widetilde{[I,T]})\oplus\widehat{\phi}(\widetilde{[S,J]})} \& {\widehat{\phi}(\widetilde{[I,J]})} \& 0
	\arrow[from=1-1, to=1-2]
	\arrow[from=1-2, to=1-3]
	\arrow[from=1-3, to=1-4]
	\arrow[from=1-4, to=1-5].
\end{tikzcd}
\end{equation}
\end{prop}
\begin{proof}
Since the degree shift automorphisms are exact, by Corollary \ref{L-action}, it suffices to consider the following cases:
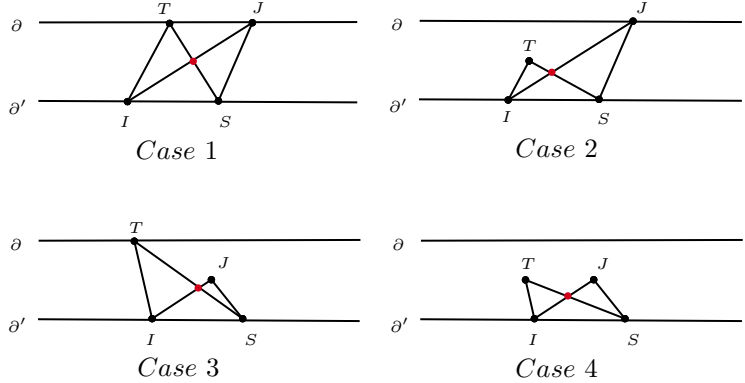
\begin{figure}[H]
    \centering

\tikzset{every picture/.style={line width=0.75pt}}          

\begin{tikzpicture}[x=0.75pt,y=0.75pt,yscale=-1,xscale=1]

\draw [color={rgb, 255:red, 0; green, 0; blue, 0 }  ,draw opacity=1 ]   (121.67,189.79) -- (106.07,170.19) ;
  
\draw    (20.03,40.77) -- (178.87,40.59) ;
  
\draw    (19.7,80.27) -- (178.87,80.99) ;
  
\draw [color={rgb, 255:red, 0; green, 0; blue, 0 }  ,draw opacity=1 ]   (109.58,80.27) -- (126.53,40.93) ;
  
\draw [color={rgb, 255:red, 0; green, 0; blue, 0 }  ,draw opacity=1 ]   (64.2,80.6) -- (85.2,41.27) ;
  
\draw    (210.03,39.97) -- (370.87,40.59) ;
  
\draw    (209.7,79.47) -- (370.87,80.19) ;
  
\draw [color={rgb, 255:red, 0; green, 0; blue, 0 }  ,draw opacity=1 ]   (299.57,79.47) -- (316.53,40.13) ;
  
\draw [color={rgb, 255:red, 0; green, 0; blue, 0 }  ,draw opacity=1 ]   (254.2,79.8) -- (264.7,60.13) ;
\draw [shift={(264.7,60.13)}, rotate = 298.1] [color={rgb, 255:red, 0; green, 0; blue, 0 }  ,draw opacity=1 ][fill={rgb, 255:red, 0; green, 0; blue, 0 }  ,fill opacity=1 ][line width=0.75]      (0, 0) circle [x radius= 1.34, y radius= 1.34]   ;
  
\draw    (19.83,150.57) -- (180.47,150.19) ;
  
\draw    (19.5,190.07) -- (180.07,190.99) ;
  
\draw [color={rgb, 255:red, 0; green, 0; blue, 0 }  ,draw opacity=1 ]   (76.47,189.79) -- (67.6,150.8) ;
  
\draw [color={rgb, 255:red, 0; green, 0; blue, 0 }  ,draw opacity=1 ]   (67.6,150.8) -- (121.67,189.79) ;
\draw [shift={(121.67,189.79)}, rotate = 35.79] [color={rgb, 255:red, 0; green, 0; blue, 0 }  ,draw opacity=1 ][fill={rgb, 255:red, 0; green, 0; blue, 0 }  ,fill opacity=1 ][line width=0.75]      (0, 0) circle [x radius= 1.34, y radius= 1.34]   ;
\draw [shift={(67.6,150.8)}, rotate = 35.79] [color={rgb, 255:red, 0; green, 0; blue, 0 }  ,draw opacity=1 ][fill={rgb, 255:red, 0; green, 0; blue, 0 }  ,fill opacity=1 ][line width=0.75]      (0, 0) circle [x radius= 1.34, y radius= 1.34]   ;
  
\draw [color={rgb, 255:red, 0; green, 0; blue, 0 }  ,draw opacity=1 ]   (64.2,80.6) -- (126.53,40.93) ;
\draw [shift={(126.53,40.93)}, rotate = 327.53] [color={rgb, 255:red, 0; green, 0; blue, 0 }  ,draw opacity=1 ][fill={rgb, 255:red, 0; green, 0; blue, 0 }  ,fill opacity=1 ][line width=0.75]      (0, 0) circle [x radius= 1.34, y radius= 1.34]   ;
\draw [shift={(64.2,80.6)}, rotate = 327.53] [color={rgb, 255:red, 0; green, 0; blue, 0 }  ,draw opacity=1 ][fill={rgb, 255:red, 0; green, 0; blue, 0 }  ,fill opacity=1 ][line width=0.75]      (0, 0) circle [x radius= 1.34, y radius= 1.34]   ;
  
\draw [color={rgb, 255:red, 0; green, 0; blue, 0 }  ,draw opacity=1 ]   (85.2,41.27) -- (109.58,80.27) ;
\draw [shift={(109.58,80.27)}, rotate = 57.99] [color={rgb, 255:red, 0; green, 0; blue, 0 }  ,draw opacity=1 ][fill={rgb, 255:red, 0; green, 0; blue, 0 }  ,fill opacity=1 ][line width=0.75]      (0, 0) circle [x radius= 1.34, y radius= 1.34]   ;
\draw [shift={(85.2,41.27)}, rotate = 57.99] [color={rgb, 255:red, 0; green, 0; blue, 0 }  ,draw opacity=1 ][fill={rgb, 255:red, 0; green, 0; blue, 0 }  ,fill opacity=1 ][line width=0.75]      (0, 0) circle [x radius= 1.34, y radius= 1.34]   ;
  
\draw [color={rgb, 255:red, 0; green, 0; blue, 0 }  ,draw opacity=1 ]   (264.7,60.13) -- (299.57,79.47) ;
\draw [shift={(299.57,79.47)}, rotate = 29] [color={rgb, 255:red, 0; green, 0; blue, 0 }  ,draw opacity=1 ][fill={rgb, 255:red, 0; green, 0; blue, 0 }  ,fill opacity=1 ][line width=0.75]      (0, 0) circle [x radius= 1.34, y radius= 1.34]   ;
\draw [shift={(264.7,60.13)}, rotate = 29] [color={rgb, 255:red, 0; green, 0; blue, 0 }  ,draw opacity=1 ][fill={rgb, 255:red, 0; green, 0; blue, 0 }  ,fill opacity=1 ][line width=0.75]      (0, 0) circle [x radius= 1.34, y radius= 1.34]   ;
  
\draw [color={rgb, 255:red, 0; green, 0; blue, 0 }  ,draw opacity=1 ]   (254.2,79.8) -- (316.53,40.13) ;
\draw [shift={(316.53,40.13)}, rotate = 327.53] [color={rgb, 255:red, 0; green, 0; blue, 0 }  ,draw opacity=1 ][fill={rgb, 255:red, 0; green, 0; blue, 0 }  ,fill opacity=1 ][line width=0.75]      (0, 0) circle [x radius= 1.34, y radius= 1.34]   ;
\draw [shift={(254.2,79.8)}, rotate = 327.53] [color={rgb, 255:red, 0; green, 0; blue, 0 }  ,draw opacity=1 ][fill={rgb, 255:red, 0; green, 0; blue, 0 }  ,fill opacity=1 ][line width=0.75]      (0, 0) circle [x radius= 1.34, y radius= 1.34]   ;
  
\draw [color={rgb, 255:red, 0; green, 0; blue, 0 }  ,draw opacity=1 ]   (106.07,170.19) -- (76.47,189.79) ;
\draw [shift={(76.47,189.79)}, rotate = 146.49] [color={rgb, 255:red, 0; green, 0; blue, 0 }  ,draw opacity=1 ][fill={rgb, 255:red, 0; green, 0; blue, 0 }  ,fill opacity=1 ][line width=0.75]      (0, 0) circle [x radius= 1.34, y radius= 1.34]   ;
\draw [shift={(106.07,170.19)}, rotate = 146.49] [color={rgb, 255:red, 0; green, 0; blue, 0 }  ,draw opacity=1 ][fill={rgb, 255:red, 0; green, 0; blue, 0 }  ,fill opacity=1 ][line width=0.75]      (0, 0) circle [x radius= 1.34, y radius= 1.34]   ;
  
\draw [color={rgb, 255:red, 0; green, 0; blue, 0 }  ,draw opacity=1 ]   (312.5,189.79) -- (296.9,170.19) ;
  
\draw    (210.67,150.57) -- (371.3,150.19) ;
  
\draw    (210.33,190.07) -- (370.9,190.99) ;
  
\draw [color={rgb, 255:red, 0; green, 0; blue, 0 }  ,draw opacity=1 ]   (267.3,189.79) -- (262.87,170.29) ;
  
\draw [color={rgb, 255:red, 0; green, 0; blue, 0 }  ,draw opacity=1 ]   (262.87,170.29) -- (312.5,189.79) ;
\draw [shift={(312.5,189.79)}, rotate = 21.44] [color={rgb, 255:red, 0; green, 0; blue, 0 }  ,draw opacity=1 ][fill={rgb, 255:red, 0; green, 0; blue, 0 }  ,fill opacity=1 ][line width=0.75]      (0, 0) circle [x radius= 1.34, y radius= 1.34]   ;
\draw [shift={(262.87,170.29)}, rotate = 21.44] [color={rgb, 255:red, 0; green, 0; blue, 0 }  ,draw opacity=1 ][fill={rgb, 255:red, 0; green, 0; blue, 0 }  ,fill opacity=1 ][line width=0.75]      (0, 0) circle [x radius= 1.34, y radius= 1.34]   ;
  
\draw [color={rgb, 255:red, 0; green, 0; blue, 0 }  ,draw opacity=1 ]   (296.9,170.19) -- (267.3,189.79) ;
\draw [shift={(267.3,189.79)}, rotate = 146.49] [color={rgb, 255:red, 0; green, 0; blue, 0 }  ,draw opacity=1 ][fill={rgb, 255:red, 0; green, 0; blue, 0 }  ,fill opacity=1 ][line width=0.75]      (0, 0) circle [x radius= 1.34, y radius= 1.34]   ;
\draw [shift={(296.9,170.19)}, rotate = 146.49] [color={rgb, 255:red, 0; green, 0; blue, 0 }  ,draw opacity=1 ][fill={rgb, 255:red, 0; green, 0; blue, 0 }  ,fill opacity=1 ][line width=0.75]      (0, 0) circle [x radius= 1.34, y radius= 1.34]   ;
  
\draw    (109.58,80.27) ;
\draw [shift={(109.58,80.27)}, rotate = 0] [color={rgb, 255:red, 0; green, 0; blue, 0 }  ][fill={rgb, 255:red, 0; green, 0; blue, 0 }  ][line width=0.75]      (0, 0) circle [x radius= 1.34, y radius= 1.34]   ;
\draw [shift={(109.58,80.27)}, rotate = 0] [color={rgb, 255:red, 0; green, 0; blue, 0 }  ][fill={rgb, 255:red, 0; green, 0; blue, 0 }  ][line width=0.75]      (0, 0) circle [x radius= 1.34, y radius= 1.34]   ;
  
\draw    (67.6,150.8) ;
\draw [shift={(67.6,150.8)}, rotate = 0] [color={rgb, 255:red, 0; green, 0; blue, 0 }  ][fill={rgb, 255:red, 0; green, 0; blue, 0 }  ][line width=0.75]      (0, 0) circle [x radius= 1.34, y radius= 1.34]   ;
\draw [shift={(67.6,150.8)}, rotate = 0] [color={rgb, 255:red, 0; green, 0; blue, 0 }  ][fill={rgb, 255:red, 0; green, 0; blue, 0 }  ][line width=0.75]      (0, 0) circle [x radius= 1.34, y radius= 1.34]   ;
  
\draw    (76.47,189.79) ;
\draw [shift={(76.47,189.79)}, rotate = 0] [color={rgb, 255:red, 0; green, 0; blue, 0 }  ][fill={rgb, 255:red, 0; green, 0; blue, 0 }  ][line width=0.75]      (0, 0) circle [x radius= 1.34, y radius= 1.34]   ;
\draw [shift={(76.47,189.79)}, rotate = 0] [color={rgb, 255:red, 0; green, 0; blue, 0 }  ][fill={rgb, 255:red, 0; green, 0; blue, 0 }  ][line width=0.75]      (0, 0) circle [x radius= 1.34, y radius= 1.34]   ;
  
\draw    (121.67,189.79) ;
\draw [shift={(121.67,189.79)}, rotate = 0] [color={rgb, 255:red, 0; green, 0; blue, 0 }  ][fill={rgb, 255:red, 0; green, 0; blue, 0 }  ][line width=0.75]      (0, 0) circle [x radius= 1.34, y radius= 1.34]   ;
\draw [shift={(121.67,189.79)}, rotate = 0] [color={rgb, 255:red, 0; green, 0; blue, 0 }  ][fill={rgb, 255:red, 0; green, 0; blue, 0 }  ][line width=0.75]      (0, 0) circle [x radius= 1.34, y radius= 1.34]   ;
  
\draw    (106.07,170.19) ;
\draw [shift={(106.07,170.19)}, rotate = 0] [color={rgb, 255:red, 0; green, 0; blue, 0 }  ][fill={rgb, 255:red, 0; green, 0; blue, 0 }  ][line width=0.75]      (0, 0) circle [x radius= 1.34, y radius= 1.34]   ;
\draw [shift={(106.07,170.19)}, rotate = 0] [color={rgb, 255:red, 0; green, 0; blue, 0 }  ][fill={rgb, 255:red, 0; green, 0; blue, 0 }  ][line width=0.75]      (0, 0) circle [x radius= 1.34, y radius= 1.34]   ;
  
\draw    (312.5,189.79) ;
\draw [shift={(312.5,189.79)}, rotate = 0] [color={rgb, 255:red, 0; green, 0; blue, 0 }  ][fill={rgb, 255:red, 0; green, 0; blue, 0 }  ][line width=0.75]      (0, 0) circle [x radius= 1.34, y radius= 1.34]   ;
\draw [shift={(312.5,189.79)}, rotate = 0] [color={rgb, 255:red, 0; green, 0; blue, 0 }  ][fill={rgb, 255:red, 0; green, 0; blue, 0 }  ][line width=0.75]      (0, 0) circle [x radius= 1.34, y radius= 1.34]   ;
  
\draw    (267.3,189.79) ;
\draw [shift={(267.3,189.79)}, rotate = 0] [color={rgb, 255:red, 0; green, 0; blue, 0 }  ][fill={rgb, 255:red, 0; green, 0; blue, 0 }  ][line width=0.75]      (0, 0) circle [x radius= 1.34, y radius= 1.34]   ;
\draw [shift={(267.3,189.79)}, rotate = 0] [color={rgb, 255:red, 0; green, 0; blue, 0 }  ][fill={rgb, 255:red, 0; green, 0; blue, 0 }  ][line width=0.75]      (0, 0) circle [x radius= 1.34, y radius= 1.34]   ;
  
\draw    (296.9,170.19) ;
\draw [shift={(296.9,170.19)}, rotate = 0] [color={rgb, 255:red, 0; green, 0; blue, 0 }  ][fill={rgb, 255:red, 0; green, 0; blue, 0 }  ][line width=0.75]      (0, 0) circle [x radius= 1.34, y radius= 1.34]   ;
\draw [shift={(296.9,170.19)}, rotate = 0] [color={rgb, 255:red, 0; green, 0; blue, 0 }  ][fill={rgb, 255:red, 0; green, 0; blue, 0 }  ][line width=0.75]      (0, 0) circle [x radius= 1.34, y radius= 1.34]   ;
  
\draw    (262.87,170.29) ;
\draw [shift={(262.87,170.29)}, rotate = 0] [color={rgb, 255:red, 0; green, 0; blue, 0 }  ][fill={rgb, 255:red, 0; green, 0; blue, 0 }  ][line width=0.75]      (0, 0) circle [x radius= 1.34, y radius= 1.34]   ;
\draw [shift={(262.87,170.29)}, rotate = 0] [color={rgb, 255:red, 0; green, 0; blue, 0 }  ][fill={rgb, 255:red, 0; green, 0; blue, 0 }  ][line width=0.75]      (0, 0) circle [x radius= 1.34, y radius= 1.34]   ;
  
\draw    (316.53,40.13) ;
\draw [shift={(316.53,40.13)}, rotate = 0] [color={rgb, 255:red, 0; green, 0; blue, 0 }  ][fill={rgb, 255:red, 0; green, 0; blue, 0 }  ][line width=0.75]      (0, 0) circle [x radius= 1.34, y radius= 1.34]   ;
\draw [shift={(316.53,40.13)}, rotate = 0] [color={rgb, 255:red, 0; green, 0; blue, 0 }  ][fill={rgb, 255:red, 0; green, 0; blue, 0 }  ][line width=0.75]      (0, 0) circle [x radius= 1.34, y radius= 1.34]   ;
  
\draw    (299.57,79.47) ;
\draw [shift={(299.57,79.47)}, rotate = 0] [color={rgb, 255:red, 0; green, 0; blue, 0 }  ][fill={rgb, 255:red, 0; green, 0; blue, 0 }  ][line width=0.75]      (0, 0) circle [x radius= 1.34, y radius= 1.34]   ;
\draw [shift={(299.57,79.47)}, rotate = 0] [color={rgb, 255:red, 0; green, 0; blue, 0 }  ][fill={rgb, 255:red, 0; green, 0; blue, 0 }  ][line width=0.75]      (0, 0) circle [x radius= 1.34, y radius= 1.34]   ;
  
\draw    (264.7,60.13) ;
\draw [shift={(264.7,60.13)}, rotate = 0] [color={rgb, 255:red, 0; green, 0; blue, 0 }  ][fill={rgb, 255:red, 0; green, 0; blue, 0 }  ][line width=0.75]      (0, 0) circle [x radius= 1.34, y radius= 1.34]   ;
\draw [shift={(264.7,60.13)}, rotate = 0] [color={rgb, 255:red, 0; green, 0; blue, 0 }  ][fill={rgb, 255:red, 0; green, 0; blue, 0 }  ][line width=0.75]      (0, 0) circle [x radius= 1.34, y radius= 1.34]   ;
  
\draw    (254.2,79.8) ;
\draw [shift={(254.2,79.8)}, rotate = 0] [color={rgb, 255:red, 0; green, 0; blue, 0 }  ][fill={rgb, 255:red, 0; green, 0; blue, 0 }  ][line width=0.75]      (0, 0) circle [x radius= 1.34, y radius= 1.34]   ;
\draw [shift={(254.2,79.8)}, rotate = 0] [color={rgb, 255:red, 0; green, 0; blue, 0 }  ][fill={rgb, 255:red, 0; green, 0; blue, 0 }  ][line width=0.75]      (0, 0) circle [x radius= 1.34, y radius= 1.34]   ;
  
\draw    (85.2,41.27) ;
\draw [shift={(85.2,41.27)}, rotate = 0] [color={rgb, 255:red, 0; green, 0; blue, 0 }  ][fill={rgb, 255:red, 0; green, 0; blue, 0 }  ][line width=0.75]      (0, 0) circle [x radius= 1.34, y radius= 1.34]   ;
\draw [shift={(85.2,41.27)}, rotate = 0] [color={rgb, 255:red, 0; green, 0; blue, 0 }  ][fill={rgb, 255:red, 0; green, 0; blue, 0 }  ][line width=0.75]      (0, 0) circle [x radius= 1.34, y radius= 1.34]   ;
  
\draw    (64.2,80.6) ;
\draw [shift={(64.2,80.6)}, rotate = 0] [color={rgb, 255:red, 0; green, 0; blue, 0 }  ][fill={rgb, 255:red, 0; green, 0; blue, 0 }  ][line width=0.75]      (0, 0) circle [x radius= 1.34, y radius= 1.34]   ;
\draw [shift={(64.2,80.6)}, rotate = 0] [color={rgb, 255:red, 0; green, 0; blue, 0 }  ][fill={rgb, 255:red, 0; green, 0; blue, 0 }  ][line width=0.75]      (0, 0) circle [x radius= 1.34, y radius= 1.34]   ;
  
\draw    (126.53,40.93) ;
\draw [shift={(126.53,40.93)}, rotate = 0] [color={rgb, 255:red, 0; green, 0; blue, 0 }  ][fill={rgb, 255:red, 0; green, 0; blue, 0 }  ][line width=0.75]      (0, 0) circle [x radius= 1.34, y radius= 1.34]   ;
\draw [shift={(126.53,40.93)}, rotate = 0] [color={rgb, 255:red, 0; green, 0; blue, 0 }  ][fill={rgb, 255:red, 0; green, 0; blue, 0 }  ][line width=0.75]      (0, 0) circle [x radius= 1.34, y radius= 1.34]   ;
  
\draw  [color={rgb, 255:red, 208; green, 2; blue, 27 }  ,draw opacity=1 ][fill={rgb, 255:red, 208; green, 2; blue, 27 }  ,fill opacity=1 ] (282.67,178.33) .. controls (282.67,177.6) and (283.26,177) .. (284,177) .. controls (284.74,177) and (285.33,177.6) .. (285.33,178.33) .. controls (285.33,179.07) and (284.74,179.67) .. (284,179.67) .. controls (283.26,179.67) and (282.67,179.07) .. (282.67,178.33) -- cycle ;
  
\draw  [color={rgb, 255:red, 208; green, 2; blue, 27 }  ,draw opacity=1 ][fill={rgb, 255:red, 208; green, 2; blue, 27 }  ,fill opacity=1 ] (98.17,174.33) .. controls (98.17,173.6) and (98.76,173) .. (99.5,173) .. controls (100.24,173) and (100.83,173.6) .. (100.83,174.33) .. controls (100.83,175.07) and (100.24,175.67) .. (99.5,175.67) .. controls (98.76,175.67) and (98.17,175.07) .. (98.17,174.33) -- cycle ;
  
\draw  [color={rgb, 255:red, 208; green, 2; blue, 27 }  ,draw opacity=1 ][fill={rgb, 255:red, 208; green, 2; blue, 27 }  ,fill opacity=1 ] (274.67,65.83) .. controls (274.67,65.1) and (275.26,64.5) .. (276,64.5) .. controls (276.74,64.5) and (277.33,65.1) .. (277.33,65.83) .. controls (277.33,66.57) and (276.74,67.17) .. (276,67.17) .. controls (275.26,67.17) and (274.67,66.57) .. (274.67,65.83) -- cycle ;
  
\draw  [color={rgb, 255:red, 208; green, 2; blue, 27 }  ,draw opacity=1 ][fill={rgb, 255:red, 208; green, 2; blue, 27 }  ,fill opacity=1 ] (95.67,60.33) .. controls (95.67,59.6) and (96.26,59) .. (97,59) .. controls (97.74,59) and (98.33,59.6) .. (98.33,60.33) .. controls (98.33,61.07) and (97.74,61.67) .. (97,61.67) .. controls (96.26,61.67) and (95.67,61.07) .. (95.67,60.33) -- cycle ;

\draw (3.37,78.17) node [anchor=north west][inner sep=0.75pt]  [font=\tiny]  {$\partial ^{\prime }$};
 
\draw (4.37,38.17) node [anchor=north west][inner sep=0.75pt]  [font=\tiny]  {$\partial $};
 
\draw (59.1,86.5) node [anchor=north west][inner sep=0.75pt]  [font=\tiny]  {$I$};
 
\draw (77.53,29.5) node [anchor=north west][inner sep=0.75pt]  [font=\tiny]  {$T$};
 
\draw (193.37,77.37) node [anchor=north west][inner sep=0.75pt]  [font=\tiny]  {$\partial ^{\prime }$};
 
\draw (194.37,37.37) node [anchor=north west][inner sep=0.75pt]  [font=\tiny]  {$\partial $};
 
\draw (3.17,187.97) node [anchor=north west][inner sep=0.75pt]  [font=\tiny]  {$\partial ^{\prime }$};
 
\draw (4.17,147.97) node [anchor=north west][inner sep=0.75pt]  [font=\tiny]  {$\partial $};
 
\draw (108.33,86.3) node [anchor=north west][inner sep=0.75pt]  [font=\tiny]  {$S$};
 
\draw (124.7,29.3) node [anchor=north west][inner sep=0.75pt]  [font=\tiny]  {$J$};
 
\draw (259.93,48.3) node [anchor=north west][inner sep=0.75pt]  [font=\tiny]  {$T$};
 
\draw (297.13,84.7) node [anchor=north west][inner sep=0.75pt]  [font=\tiny]  {$S$};
 
\draw (249.5,84.1) node [anchor=north west][inner sep=0.75pt]  [font=\tiny]  {$I$};
 
\draw (316.3,29.3) node [anchor=north west][inner sep=0.75pt]  [font=\tiny]  {$J$};
 
\draw (71.9,196.1) node [anchor=north west][inner sep=0.75pt]  [font=\tiny]  {$I$};
 
\draw (121.13,195.9) node [anchor=north west][inner sep=0.75pt]  [font=\tiny]  {$S$};
 
\draw (107.53,157.5) node [anchor=north west][inner sep=0.75pt]  [font=\tiny]  {$J$};
 
\draw (63.5,139.7) node [anchor=north west][inner sep=0.75pt]  [font=\tiny]  {$T$};
 
\draw (66.8,98.83) node [anchor=north west][inner sep=0.75pt]    {$Case\ 1$};
 
\draw (256.4,98.03) node [anchor=north west][inner sep=0.75pt]    {$Case\ 2$};
 
\draw (67.6,208.03) node [anchor=north west][inner sep=0.75pt]    {$Case\ 3$};
 
\draw (194,187.97) node [anchor=north west][inner sep=0.75pt]  [font=\tiny]  {$\partial ^{\prime }$};
 
\draw (195,147.97) node [anchor=north west][inner sep=0.75pt]  [font=\tiny]  {$\partial $};
 
\draw (262.73,196.1) node [anchor=north west][inner sep=0.75pt]  [font=\tiny]  {$I$};
 
\draw (311.97,195.9) node [anchor=north west][inner sep=0.75pt]  [font=\tiny]  {$S$};
 
\draw (296.77,157.9) node [anchor=north west][inner sep=0.75pt]  [font=\tiny]  {$J$};
 
\draw (259.13,157.7) node [anchor=north west][inner sep=0.75pt]  [font=\tiny]  {$T$};
 
\draw (256.43,208.43) node [anchor=north west][inner sep=0.75pt]    {$Case\ 4$};

\end{tikzpicture}

\end{figure}
\begin{itemize}
\item[\emph{Case 1:}]  $[I,J]$ and $[S,T]$ belong to $\operatorname{Seg}(M) \setminus \operatorname{Seg}_0(M)$. The required exact sequence can be obtained by Proposition \ref{exact sequence1}.
\item[\emph{Case 2:}] $[I,J]$ belongs to $\operatorname{Seg}(M) \setminus \operatorname{Seg}_0(M)$, and $[S,T]$ belongs to $\operatorname{Seg}_0^*(M)$. Assume $T=P_k\in \mathcal{P}$, $[S,J]=[s,j]$, $[I,J]=[i,j]$.  Without loss of generality, we may assume $k$ is even. Then  \begin{align*}
\hspace{3em}\widehat{\phi}(\widetilde{[I,J]}) &= \mathsf{E}_{L_0(-i\vec{x}_3)}\langle (i+j-1)\vec{x}_3 \rangle, & \widehat{\phi}(\widetilde{[S,J]}) &= \mathsf{E}_{L_0(-s\vec{x}_3)}\langle (s+j-1)\vec{x}_3 \rangle, \\
\hspace{3em}\widehat{\phi}(\widetilde{[S,T]}) &= L_0^*(k\vec{x}_1-(s+1)\vec{x}_3), & \widehat{\phi}(\widetilde{[I,T]}) &= L_0^*(k\vec{x}_1-(i+1)\vec{x}_3).
\end{align*}
Let $L=L_0(k\vec{x}_1-s\vec{x}_3)$, $\vec{x}=(i+j+kn-1)\vec{x}_3$, and $\vec{y}=(s-i)\vec{x}_3$. Since $[I, J] \cap^+ [S, T] \neq \emptyset$, we have $s>i$ and $i+j+kn>0$. By Proposition \ref{bicartesian square}, we have the bicartesian square  
 \[\begin{tikzcd}[ampersand replacement=\&,cramped,sep=small]
	{L(\vec{\omega})} \& {L(\vec{\omega})(\vec{y})} \\
	{\mathsf{E}_{L}\langle \vec{x}+\vec{y}\rangle} \& {\mathsf{E}_{L(\vec{y})}\langle \vec{x}\rangle}
	\arrow[from=1-1, to=1-2]
	\arrow[from=1-2, to=2-2]
	\arrow[from=1-1, to=2-1]
	\arrow[from=2-1, to=2-2]
\end{tikzcd}\]
which induces the exact sequence as required.
\item[\emph{Case 3:}]  $[I,J]$ belongs to $\operatorname{Seg}_0^*(M)$, and $[S,T]$ belongs to $\operatorname{Seg}(M) \setminus \operatorname{Seg}_0(M)$. By Proposition \ref{vector bundle duality}, this case can be transformed into  Case 2 through vector bundle duality $^\vee$.  
\item[\emph{Case 4:}] $[I,J]$ and $[S,T]$ belong to $\operatorname{Seg}_0^*(M)$. Assume $T=P_k$, $J=P_{k^\prime}$, $I=(i,0)$, and $S=(s,0)$. Without loss of generality, we may assume $k$ and $k^\prime$ are even. We have the following short exact sequence 
$0\rightarrow L_0^*(k\vec{x}_1-(s+1)\vec{x}_3)\rightarrow {L_0^*(k\vec{x}_1-(i+1)\vec{x}_3)\oplus L_0^*(k^\prime\vec{x}_1-(s+1)\vec{x}_3)}\rightarrow {L_0^*(k^\prime\vec{x}_1-(i+1)\vec{x}_3)}\rightarrow 0$ in ${\rm coh}\mbox{-}\mathbb{X}$. Note that 
\begin{align*}
\hspace{3em}\widehat{\phi}(\widetilde{[I,J]}) &= L_0^*(k^\prime\vec{x}_1-(i+1)\vec{x}_3), & \widehat{\phi}(\widetilde{[S,J]}) &= L_0^*(k^\prime\vec{x}_1-(s+1)\vec{x}_3), \\
\hspace{3em}\widehat{\phi}(\widetilde{[S,T]}) &= L_0^*(k\vec{x}_1-(s+1)\vec{x}_3), & \widehat{\phi}(\widetilde{[I,T]}) &= L_0^*(k\vec{x}_1-(i+1)\vec{x}_3).
\end{align*}
\end{itemize}
This finishes the proof.
\end{proof}

   Let $[i,j]$ be a line segment in $ \operatorname{Seg}(M) \setminus \operatorname{Seg}_0(M)$. Then $\widehat{\phi}([i,j])$ is an extension bundle and  $[i,j]$ can form a triangle with any point in $\mathcal{P}$. For each such triangle, we label the other two edges as green and blue, in ascending order of the reciprocal of their slopes, while  $[i,j]$ is colored red. 
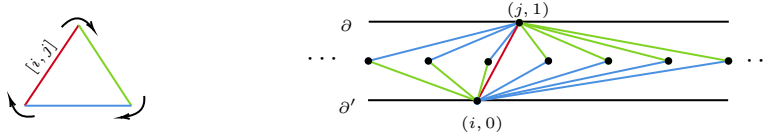
\begin{figure}[H]
    \centering

\tikzset{every picture/.style={line width=0.75pt}}          

\begin{tikzpicture}[x=0.75pt,y=0.75pt,yscale=-1,xscale=1]

\draw    (251.33,21) -- (431.08,21) ;
  
\draw    (251,60.5) -- (430.75,60.5) ;
  
\draw [color={rgb, 255:red, 74; green, 144; blue, 226 }  ,draw opacity=1 ]   (311,41.17) -- (326.5,21.5) ;
\draw [shift={(326.5,21.5)}, rotate = 308.24] [color={rgb, 255:red, 74; green, 144; blue, 226 }  ,draw opacity=1 ][fill={rgb, 255:red, 74; green, 144; blue, 226 }  ,fill opacity=1 ][line width=0.75]      (0, 0) circle [x radius= 1.34, y radius= 1.34]   ;
\draw [shift={(311,41.17)}, rotate = 308.24] [color={rgb, 255:red, 74; green, 144; blue, 226 }  ,draw opacity=1 ][fill={rgb, 255:red, 74; green, 144; blue, 226 }  ,fill opacity=1 ][line width=0.75]      (0, 0) circle [x radius= 1.34, y radius= 1.34]   ;
  
\draw [color={rgb, 255:red, 208; green, 2; blue, 27 }  ,draw opacity=1 ]   (305.5,60.83) -- (326.5,21.5) ;
\draw [shift={(305.5,60.83)}, rotate = 298.1] [color={rgb, 255:red, 208; green, 2; blue, 27 }  ,draw opacity=1 ][fill={rgb, 255:red, 208; green, 2; blue, 27 }  ,fill opacity=1 ][line width=0.75]      (0, 0) circle [x radius= 1.34, y radius= 1.34]   ;
  
\draw [color={rgb, 255:red, 126; green, 211; blue, 33 }  ,draw opacity=1 ]   (305.5,60.83) -- (311,41.17) ;
\draw [shift={(311,41.17)}, rotate = 285.62] [color={rgb, 255:red, 126; green, 211; blue, 33 }  ,draw opacity=1 ][fill={rgb, 255:red, 126; green, 211; blue, 33 }  ,fill opacity=1 ][line width=0.75]      (0, 0) circle [x radius= 1.34, y radius= 1.34]   ;
\draw [shift={(305.5,60.83)}, rotate = 285.62] [color={rgb, 255:red, 126; green, 211; blue, 33 }  ,draw opacity=1 ][fill={rgb, 255:red, 126; green, 211; blue, 33 }  ,fill opacity=1 ][line width=0.75]      (0, 0) circle [x radius= 1.34, y radius= 1.34]   ;
  
\draw [color={rgb, 255:red, 74; green, 144; blue, 226 }  ,draw opacity=1 ]   (281.12,40.75) -- (326.5,21.5) ;
\draw [shift={(281.12,40.75)}, rotate = 337.01] [color={rgb, 255:red, 74; green, 144; blue, 226 }  ,draw opacity=1 ][fill={rgb, 255:red, 74; green, 144; blue, 226 }  ,fill opacity=1 ][line width=0.75]      (0, 0) circle [x radius= 1.34, y radius= 1.34]   ;
  
\draw [color={rgb, 255:red, 126; green, 211; blue, 33 }  ,draw opacity=1 ]   (281.12,40.75) -- (305.5,60.83) ;
  
\draw [color={rgb, 255:red, 126; green, 211; blue, 33 }  ,draw opacity=1 ]   (251.17,40.75) -- (305.5,60.83) ;
  
\draw [color={rgb, 255:red, 74; green, 144; blue, 226 }  ,draw opacity=1 ]   (371,40.75) -- (305.5,60.83) ;
\draw [shift={(305.5,60.83)}, rotate = 162.95] [color={rgb, 255:red, 74; green, 144; blue, 226 }  ,draw opacity=1 ][fill={rgb, 255:red, 74; green, 144; blue, 226 }  ,fill opacity=1 ][line width=0.75]      (0, 0) circle [x radius= 1.34, y radius= 1.34]   ;
\draw [shift={(371,40.75)}, rotate = 162.95] [color={rgb, 255:red, 74; green, 144; blue, 226 }  ,draw opacity=1 ][fill={rgb, 255:red, 74; green, 144; blue, 226 }  ,fill opacity=1 ][line width=0.75]      (0, 0) circle [x radius= 1.34, y radius= 1.34]   ;
  
\draw [color={rgb, 255:red, 74; green, 144; blue, 226 }  ,draw opacity=1 ]   (400.96,40.75) -- (305.5,60.83) ;
\draw [shift={(305.5,60.83)}, rotate = 168.12] [color={rgb, 255:red, 74; green, 144; blue, 226 }  ,draw opacity=1 ][fill={rgb, 255:red, 74; green, 144; blue, 226 }  ,fill opacity=1 ][line width=0.75]      (0, 0) circle [x radius= 1.34, y radius= 1.34]   ;
\draw [shift={(400.96,40.75)}, rotate = 168.12] [color={rgb, 255:red, 74; green, 144; blue, 226 }  ,draw opacity=1 ][fill={rgb, 255:red, 74; green, 144; blue, 226 }  ,fill opacity=1 ][line width=0.75]      (0, 0) circle [x radius= 1.34, y radius= 1.34]   ;
  
\draw [color={rgb, 255:red, 74; green, 144; blue, 226 }  ,draw opacity=1 ]   (341.04,40.75) -- (305.5,60.83) ;
\draw [shift={(305.5,60.83)}, rotate = 150.53] [color={rgb, 255:red, 74; green, 144; blue, 226 }  ,draw opacity=1 ][fill={rgb, 255:red, 74; green, 144; blue, 226 }  ,fill opacity=1 ][line width=0.75]      (0, 0) circle [x radius= 1.34, y radius= 1.34]   ;
\draw [shift={(341.04,40.75)}, rotate = 150.53] [color={rgb, 255:red, 74; green, 144; blue, 226 }  ,draw opacity=1 ][fill={rgb, 255:red, 74; green, 144; blue, 226 }  ,fill opacity=1 ][line width=0.75]      (0, 0) circle [x radius= 1.34, y radius= 1.34]   ;
  
\draw [color={rgb, 255:red, 74; green, 144; blue, 226 }  ,draw opacity=1 ]   (430.92,40.75) -- (305.5,60.83) ;
\draw [shift={(305.5,60.83)}, rotate = 170.9] [color={rgb, 255:red, 74; green, 144; blue, 226 }  ,draw opacity=1 ][fill={rgb, 255:red, 74; green, 144; blue, 226 }  ,fill opacity=1 ][line width=0.75]      (0, 0) circle [x radius= 1.34, y radius= 1.34]   ;
\draw [shift={(430.92,40.75)}, rotate = 170.9] [color={rgb, 255:red, 74; green, 144; blue, 226 }  ,draw opacity=1 ][fill={rgb, 255:red, 74; green, 144; blue, 226 }  ,fill opacity=1 ][line width=0.75]      (0, 0) circle [x radius= 1.34, y radius= 1.34]   ;
  
\draw [color={rgb, 255:red, 126; green, 211; blue, 33 }  ,draw opacity=1 ]   (326.5,21.5) -- (371,40.75) ;
\draw [shift={(371,40.75)}, rotate = 23.39] [color={rgb, 255:red, 126; green, 211; blue, 33 }  ,draw opacity=1 ][fill={rgb, 255:red, 126; green, 211; blue, 33 }  ,fill opacity=1 ][line width=0.75]      (0, 0) circle [x radius= 1.34, y radius= 1.34]   ;
\draw [shift={(326.5,21.5)}, rotate = 23.39] [color={rgb, 255:red, 126; green, 211; blue, 33 }  ,draw opacity=1 ][fill={rgb, 255:red, 126; green, 211; blue, 33 }  ,fill opacity=1 ][line width=0.75]      (0, 0) circle [x radius= 1.34, y radius= 1.34]   ;
  
\draw [color={rgb, 255:red, 126; green, 211; blue, 33 }  ,draw opacity=1 ]   (326.5,21.5) -- (400.96,40.75) ;
\draw [shift={(400.96,40.75)}, rotate = 14.5] [color={rgb, 255:red, 126; green, 211; blue, 33 }  ,draw opacity=1 ][fill={rgb, 255:red, 126; green, 211; blue, 33 }  ,fill opacity=1 ][line width=0.75]      (0, 0) circle [x radius= 1.34, y radius= 1.34]   ;
\draw [shift={(326.5,21.5)}, rotate = 14.5] [color={rgb, 255:red, 126; green, 211; blue, 33 }  ,draw opacity=1 ][fill={rgb, 255:red, 126; green, 211; blue, 33 }  ,fill opacity=1 ][line width=0.75]      (0, 0) circle [x radius= 1.34, y radius= 1.34]   ;
  
\draw [color={rgb, 255:red, 126; green, 211; blue, 33 }  ,draw opacity=1 ]   (326.5,21.5) -- (341.04,40.75) ;
\draw [shift={(326.5,21.5)}, rotate = 52.93] [color={rgb, 255:red, 126; green, 211; blue, 33 }  ,draw opacity=1 ][fill={rgb, 255:red, 126; green, 211; blue, 33 }  ,fill opacity=1 ][line width=0.75]      (0, 0) circle [x radius= 1.34, y radius= 1.34]   ;
  
\draw [color={rgb, 255:red, 126; green, 211; blue, 33 }  ,draw opacity=1 ]   (326.5,21.5) -- (430.92,40.75) ;
\draw [shift={(430.92,40.75)}, rotate = 10.45] [color={rgb, 255:red, 126; green, 211; blue, 33 }  ,draw opacity=1 ][fill={rgb, 255:red, 126; green, 211; blue, 33 }  ,fill opacity=1 ][line width=0.75]      (0, 0) circle [x radius= 1.34, y radius= 1.34]   ;
\draw [shift={(326.5,21.5)}, rotate = 10.45] [color={rgb, 255:red, 126; green, 211; blue, 33 }  ,draw opacity=1 ][fill={rgb, 255:red, 126; green, 211; blue, 33 }  ,fill opacity=1 ][line width=0.75]      (0, 0) circle [x radius= 1.34, y radius= 1.34]   ;
  
\draw [color={rgb, 255:red, 74; green, 144; blue, 226 }  ,draw opacity=1 ]   (251.17,40.75) -- (326.5,21.5) ;
\draw [shift={(251.17,40.75)}, rotate = 345.67] [color={rgb, 255:red, 74; green, 144; blue, 226 }  ,draw opacity=1 ][fill={rgb, 255:red, 74; green, 144; blue, 226 }  ,fill opacity=1 ][line width=0.75]      (0, 0) circle [x radius= 1.34, y radius= 1.34]   ;
  
\draw    (326.5,21.5) ;
\draw [shift={(326.5,21.5)}, rotate = 0] [color={rgb, 255:red, 0; green, 0; blue, 0 }  ][fill={rgb, 255:red, 0; green, 0; blue, 0 }  ][line width=0.75]      (0, 0) circle [x radius= 1.34, y radius= 1.34]   ;
  
\draw    (430.92,40.75) ;
\draw [shift={(430.92,40.75)}, rotate = 0] [color={rgb, 255:red, 0; green, 0; blue, 0 }  ][fill={rgb, 255:red, 0; green, 0; blue, 0 }  ][line width=0.75]      (0, 0) circle [x radius= 1.34, y radius= 1.34]   ;
  
\draw    (341.04,40.75) ;
\draw [shift={(341.04,40.75)}, rotate = 0] [color={rgb, 255:red, 0; green, 0; blue, 0 }  ][fill={rgb, 255:red, 0; green, 0; blue, 0 }  ][line width=0.75]      (0, 0) circle [x radius= 1.34, y radius= 1.34]   ;
  
\draw    (400.96,40.75) ;
\draw [shift={(400.96,40.75)}, rotate = 0] [color={rgb, 255:red, 0; green, 0; blue, 0 }  ][fill={rgb, 255:red, 0; green, 0; blue, 0 }  ][line width=0.75]      (0, 0) circle [x radius= 1.34, y radius= 1.34]   ;
  
\draw    (371,40.75) ;
\draw [shift={(371,40.75)}, rotate = 0] [color={rgb, 255:red, 0; green, 0; blue, 0 }  ][fill={rgb, 255:red, 0; green, 0; blue, 0 }  ][line width=0.75]      (0, 0) circle [x radius= 1.34, y radius= 1.34]   ;
  
\draw    (311,41.17) ;
\draw [shift={(311,41.17)}, rotate = 0] [color={rgb, 255:red, 0; green, 0; blue, 0 }  ][fill={rgb, 255:red, 0; green, 0; blue, 0 }  ][line width=0.75]      (0, 0) circle [x radius= 1.34, y radius= 1.34]   ;
  
\draw    (281.12,40.75) ;
\draw [shift={(281.12,40.75)}, rotate = 0] [color={rgb, 255:red, 0; green, 0; blue, 0 }  ][fill={rgb, 255:red, 0; green, 0; blue, 0 }  ][line width=0.75]      (0, 0) circle [x radius= 1.34, y radius= 1.34]   ;
  
\draw    (251.17,40.75) ;
\draw [shift={(251.17,40.75)}, rotate = 0] [color={rgb, 255:red, 0; green, 0; blue, 0 }  ][fill={rgb, 255:red, 0; green, 0; blue, 0 }  ][line width=0.75]      (0, 0) circle [x radius= 1.34, y radius= 1.34]   ;
  
\draw    (305.5,60.83) ;
\draw [shift={(305.5,60.83)}, rotate = 0] [color={rgb, 255:red, 0; green, 0; blue, 0 }  ][fill={rgb, 255:red, 0; green, 0; blue, 0 }  ][line width=0.75]      (0, 0) circle [x radius= 1.34, y radius= 1.34]   ;
  
\draw [color={rgb, 255:red, 208; green, 2; blue, 27 }  ,draw opacity=1 ]   (106.58,22.71) -- (79.59,63.26) ;
  
\draw [color={rgb, 255:red, 126; green, 211; blue, 33 }  ,draw opacity=1 ]   (106.58,22.71) -- (132.86,63.26) ;
  
\draw [color={rgb, 255:red, 74; green, 144; blue, 226 }  ,draw opacity=1 ]   (132.86,63.26) -- (79.59,63.26) ;
  
\draw    (98.06,23.76) .. controls (103.32,17.94) and (108.27,17.8) .. (112.92,23.62) ;
\draw [shift={(114.04,25.16)}, rotate = 235.89] [color={rgb, 255:red, 0; green, 0; blue, 0 }  ][line width=0.75]    (4.37,-1.32) .. controls (2.78,-0.56) and (1.32,-0.12) .. (0,0) .. controls (1.32,0.12) and (2.78,0.56) .. (4.37,1.32)   ;
  
\draw    (139.26,58.36) .. controls (139.91,65.12) and (136.36,67.73) .. (129.42,68.92) ;
\draw [shift={(127.54,69.2)}, rotate = 352.68] [color={rgb, 255:red, 0; green, 0; blue, 0 }  ][line width=0.75]    (4.37,-1.32) .. controls (2.78,-0.56) and (1.32,-0.12) .. (0,0) .. controls (1.32,0.12) and (2.78,0.56) .. (4.37,1.32)   ;
  
\draw    (84.56,68.5) .. controls (76.22,70.02) and (74.33,65.46) .. (72.81,61.25) ;
\draw [shift={(72.13,59.41)}, rotate = 68.65] [color={rgb, 255:red, 0; green, 0; blue, 0 }  ][line width=0.75]    (4.37,-1.32) .. controls (2.78,-0.56) and (1.32,-0.12) .. (0,0) .. controls (1.32,0.12) and (2.78,0.56) .. (4.37,1.32)   ;

\draw (234.67,58.4) node [anchor=north west][inner sep=0.75pt]  [font=\tiny]  {$\partial ^{\prime }$};
 
\draw (235.67,18.4) node [anchor=north west][inner sep=0.75pt]  [font=\tiny]  {$\partial $};
 
\draw (296,66.73) node [anchor=north west][inner sep=0.75pt]  [font=\tiny]  {$( i,0)$};
 
\draw (318.83,9.73) node [anchor=north west][inner sep=0.75pt]  [font=\tiny]  {$( j,1)$};
 
\draw (438,36.2) node [anchor=north west][inner sep=0.75pt]    {$\cdots $};
 
\draw (218.8,35.4) node [anchor=north west][inner sep=0.75pt]    {$\cdots $};
 
\draw (77.15,43.88) node [anchor=north west][inner sep=0.75pt]  [font=\tiny,rotate=-306.57]  {$[ i,j]$};

\end{tikzpicture}

\end{figure}

\begin{prop}{}{}\label{triangle exact}
For each colored triangle depicted above, there exists an exact sequence of the following form:
\[\begin{tikzcd}[ampersand replacement=\&,cramped,column sep=small]
				0 \& {\widehat{\phi}(\widetilde{\text{\color{green}{GRN Seg}}})} \& {\widehat{\phi}(\widetilde{\text{\color{red}{RED Seg}}})} \& {\widehat{\phi}(\widetilde{\text{\color{blue}{BLU Seg}}})} \& 0
				\arrow[from=1-1, to=1-2]
				\arrow[from=1-2, to=1-3]
				\arrow[from=1-3, to=1-4]
				\arrow[from=1-4, to=1-5]
			\end{tikzcd}\]
\end{prop}
\begin{proof}
For any $\vec{x}=a\vec{x}_3, \vec{y}=b\vec{x}_3$ with $a\in \mathbb{N}, b\in \mathbb{Z}_+$, we have the following commutative diagram:
\[\begin{tikzcd}[ampersand replacement=\&,cramped,sep=small]
{L_0(\vec{\omega})} \& {L_0(\vec{\omega})(\vec{y})} \& 0 \\
{\mathsf{E}_{L_0}\langle \vec{x}+\vec{y}\rangle} \& {\mathsf{E}_{L_0(\vec{y})}\langle \vec{x}\rangle} \& {L_0(\vec{x}+\vec{y})}
\arrow[from=1-1, to=1-2]
\arrow[from=1-2, to=2-2]
\arrow[from=1-1, to=2-1]
\arrow[from=2-1, to=2-2]
\arrow[from=2-2, to=2-3]
\arrow[from=1-2, to=1-3]
\arrow[from=1-3, to=2-3]
\end{tikzcd}\]
where the left bicartesian square is given by Proposition \ref{bicartesian square}, and the right bicartesian square is induced by the almost-split sequence. As a consequence, we have the exact sequence:
\begin{equation}\label{1-2-1}
\begin{tikzcd}[ampersand replacement=\&,cramped,sep=small]
0 \& {L_0(\vec{\omega})} \& {\mathsf{E}_{L_0}\langle \vec{x}+\vec{y}\rangle} \& {L_0(\vec{x}+\vec{y})} \& 0.
\arrow[from=1-2, to=1-3]
\arrow[from=1-3, to=1-4]
\arrow[from=1-4, to=1-5]
\arrow[from=1-1, to=1-2]
\end{tikzcd}
\end{equation}
For any $P_k \in \mathcal{P}$, we first assume $kn < i+j$,  meaning $P_k$ lies to the left of the line segment $[i,j]$. Let $[I,P_k] = [i,kn-i]^+$ and $[P_k,J] = [kn-j,j]^-$. Then the reciprocals of the slopes of $[I,P_k]$, $[i,j]$, and $[P_k,J]$ form an increasing sequence. Moreover, $[I,P_k]$, $[i,j]$, and $[P_k,J]$ form a triangle with base $[i,j]$, where $P_k$ is the vertex relative to the base $[i,j]$. Without loss of generality, we may assume that $k$ is odd. Note that
\begin{align*}
\widehat{\phi}(\widetilde{[I,P_k]}) &=L_0(k\vec{x}_1-(i+1)\vec{x}_3),\\
\widehat{\phi}(\widetilde{[P_k,J]}) &=L_0^*(-k\vec{x}_1+(j-1)\vec{x}_3), \\
\widehat{\phi}(\widetilde{[i,j]}) &=E_{L_0(-i\vec{x}_3)}\langle (i+j-1)\vec{x}_3\rangle.
\end{align*}
Replacing $L$ and  $\vec{x}+\vec{y}$ in \eqref{1-2-1} with $L_0(k\vec{x}_1-i\vec{x}_3)$ and $(i+j-kn-1)\vec{x}_3$, respectively, we have the required exact sequence
\[\begin{tikzcd}[ampersand replacement=\&,cramped,sep=small]
	0 \& {\widehat{\phi}(\widetilde{[I,P_k]})} \& {\widehat{\phi}(\widetilde{[i,j]})} \& {\widehat{\phi}(\widetilde{[P_k,J]})} \& 0.
	\arrow[from=1-1, to=1-2]
	\arrow[from=1-2, to=1-3]
	\arrow[from=1-3, to=1-4]
	\arrow[from=1-4, to=1-5]
\end{tikzcd}\]
For the case where $kn>i+j$, the required exact sequence can be obtained in a similar way.
\end{proof}

 As a consequence, we provide the sufficient and necessary condition for the middle term of an extension between two line bundles to potentially be an extension bundle.

\begin{prop}{}{}
 For $\vec{x},\vec{y}\in \mathbb{L}$, there exists $\xi\in \Ext^1(L(\vec{y}),L(\vec{x}))$  such that its middle term is an extension bundle if and only if 
 \[\vec{y}-\vec{x}=\vec{x}_1+\vec{x}_2+k\vec{x}_3+l\vec{c}\] 
 in normal form with $1\leq k\leq n-1$ and  $l\geq -1$.
\end{prop}
\begin{proof}
The sufficiency follows from the existence of the extensions follows from the Proposition \ref{triangle exact}. For the necessity, assume there exists an extension of $L(\vec{x})$ by $L(\vec{y})$ such that its middle term is an indecomposable bundle of rank two. By Proposition \ref{Grothendieck group}, this middle term can always be written as $E_{L(\vec{z})}\langle i\vec{x}_3\rangle$, where $0\leq i\leq n-2$. We express 
    \[\vec{x}=\mbox{$\sum_{i=1}^3$} j_i\vec{x}_i+j\vec{c},\quad\vec{y}=\mbox{$\sum_{i=1}^3$} k_i\vec{x}_i+k\vec{c}, \quad\vec{z}=\mbox{$\sum_{i=1}^3$} l_i\vec{x}_i+l\vec{c}\] 
    in normal forms. Hence, in the Grothendieck group $K_0(\mathbb{X})$, we have 
\[[E_{L(\vec{z})}\langle i\vec{x}_3\rangle] = [L(\vec{x})]+[L(\vec{y})]=[L(\vec{z}+\vec{\omega})]+[L(\vec{z}+i\vec{x}_3)].\]
Notice that $\vec{\omega}=\vec{c}-\mbox{$\sum_{i=1}^3$} \vec{x}_i$. By Proposition 2.2 in \cite{dong2024two}, we have 
$\{j_i,k_i\}=\{l_i-1 \;(\bmod \; 2),l_i\}$ for $i=1,2$ and 
$\{j_3,k_3\}=\{l_3-1 \;(\bmod\; n),l_3+i\;(\bmod \;n)\}$. Hence,  $j_i+k_i=(l_i-1)+l_i=1 \bmod 2$. Moreover, we claim that $j_3\neq k_3$, otherwise, depending on the value of $j_3$, we have either $j_3=n-1$ implying $l_3=0$ and $i=n-1$, or $j_3\neq n-1$ implying $l_3=j_3+1$ and $i=-1 \bmod n$, both of which contradict $0\leq i \leq n-2$. Therefore, $j_i\neq k_i$ for $i=1,2,3$. It follows that $\vec{y}-\vec{x}=\vec{x}_1+\vec{x}_2+k\vec{x}_3+l\vec{c}$ in normal form with $1\leq k\leq n-1$. Finally, by Serre duality, $\Ext^1(L(\vec{y}),L(\vec{x}))=D\Hom(L(\vec{x}),L(\vec{y}+\vec{\omega}))=D S_{\vec{y}+\vec{\omega}-\vec{x}}=D  S_{(k-1)\vec{x}_3+(l+1)\vec{c}}\neq 0$ if and only if the integer $l$ is at least $-1$. 
\end{proof}

\subsection{Intersection indices as dimensions of extension groups}
	Assume that $\widetilde{[i,j]^*},\widetilde{[s,t]^\star}\in \widetilde{\operatorname{Seg}^*(M)}$. We denote by $[i,j]^* \cap^+ \widetilde{[s,t]}$ the set of positive intersections of all line segments in the $G$-orbit $\widetilde{[s,t]}\in \widetilde{\operatorname{Seg}(M)}$ and the line segment $[i,j]^*$. 
\begin{defn}{}{}\label{positive intersections }  
     The \emph{intersection index}  $I_{\text{int}}(\widetilde{[i,j]^*},\widetilde{[s,t]^\star})$ is defined as follows:
   \begin{itemize}
    \item If the following conditions hold:
    \begin{itemize}
        \item [(a)] $*=\varrho(\star)$ and $* \neq \text{empty}$,
        \item [(b)] $i+j+s+t \equiv 0\bmod 2n$,
           \item [(c)] The reciprocal of the slope of $[i,j]$ is less than that of $[s,t]$,
    \end{itemize}
    then:
    \[I_{\text{int}}(\widetilde{[i,j]^*},\widetilde{[s,t]^\star}) = \lvert [i,j]^* \cap^+ \widetilde{[s,t]} \rvert + 1.\]
    \item Otherwise:
    \[I_{\text{int}}(\widetilde{[i,j]^*},\widetilde{[s,t]^\star}) = \lvert [i,j]^* \cap^+ \widetilde{[s,t]} \rvert.\]
\end{itemize}
\end{defn}
\begin{rem}
       One can check that $I_{\text{int}}(\widetilde{[i,j]^*},\widetilde{[s,t]^\star})$ is independent of the choices of $[i,j]$ and $[s,t]$. In general, $I_{\text{int}}(\widetilde{[i,j]^*},\widetilde{[s,t]^\star})\neq I_{\text{int}}(\widetilde{[s,t]^\star},\widetilde{[i,j]^*})$.
\end{rem}

To prove Theorem \ref{dimension and positive intersection}, we need some preparations. 

A path $M_0 \rightarrow M_1 \rightarrow \cdots \rightarrow M_s$ in the Auslander-Reiten quiver $\Gamma(\text{vect-}\mathbb{X})$ is called a \emph{sectional path} if $\tau M_{i+1} \neq M_{i-1}$ for all $i=1, \ldots, s-1$. Let $\Sigma_{\rightarrow}(M)$ be the set of all indecomposable objects that can be reached from $M$ by a sectional path, and let $\Sigma_{\leftarrow}(M)$ be the set of all indecomposable objects from which one can reach $M$ by a sectional path.

	\begin{prop}{}{}\label{Hom and Ext}
		Let $E_L\langle\vec{x}\rangle$ be the extension bundle. For any indecomposable bundle $X$,
		\begin{itemize}
			\item[(1)] if ${\rm Ext^{1}}(X,E_L\langle\vec{x}\rangle)\neq 0$, then ${\rm Hom}(X,L(\vec{x}))= 0$ and $${\rm dim_{k}Ext^{1}}(X,E_L\langle\vec{x}\rangle) ={\rm dim_{k}Ext^{1}}(X,L(\vec{\omega}))+{\rm dim_{k}Ext^{1}}(X,L(\vec{x}));$$
			\item[(2)] if ${\rm Ext^{1}}(E_L\langle\vec{x}\rangle, X)\neq 0$, then ${\rm Hom}(L(\vec{\omega}),X)= 0$ and  $${\rm dim_{k}Ext^{1}}(E_L\langle\vec{x}\rangle,X) ={\rm dim_{k}Ext^{1}}(L(\vec{\omega}),X)+{\rm dim_{k}Ext^{1}}(L(\vec{x}),X).$$
		\end{itemize}
	\end{prop}
	\begin{proof}
		We only prove (1); the proof of (2) is similar. If ${\rm Ext^{1}}(X,E_L\langle\vec{x}\rangle)\neq 0$, then by Serre duality, \[{\rm Hom}(E_L\langle\vec{x}\rangle,X(\vec{\omega}))\cong D{\rm Ext}^{1}(X,E_L\langle\vec{x}\rangle)\neq 0.\]
   Assume ${\rm Hom}(X,L(\vec{x}))\neq 0$. According to Lemma 4.3 and Corollary 4.4 in \cite{MR3313495}, there exist integers $k,t \in \mathbb{N}$ such that $\tau^k X$ lies within $\Sigma_{\rightarrow}(\tau^{-1} E_L\langle\vec{x}\rangle)$ and $\tau^{-t} X$ lies within $\Sigma_{\leftarrow}(L(\vec{x}))$. However, this is not possible because $\Sigma_{\rightarrow}(\tau^{-1} E_L\langle\vec{x}\rangle)\cap \Sigma_{\leftarrow}(L(\vec{x}))=\emptyset$, as illustrated in the figure below: 
  \[\begin{adjustbox}{scale=0.5}
\tikzset{every picture/.style={line width=0.75pt}}          
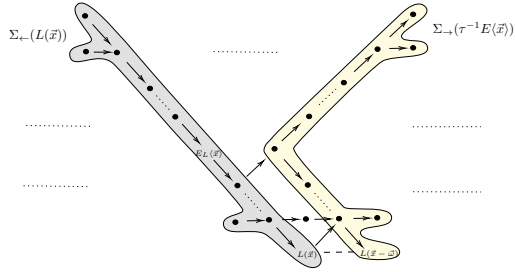
\begin{tikzpicture}[x=0.75pt,y=0.75pt,yscale=-1,xscale=1]
\draw  [fill={rgb, 255:red, 248; green, 231; blue, 28 }  ,fill opacity=0.13 ] (533.6,61.25) .. controls (534.6,72.92) and (499.93,68.25) .. (488.6,74.58) .. controls (477.27,80.92) and (402.6,156.08) .. (399.93,161.92) .. controls (397.27,167.75) and (445.27,216.25) .. (449.6,221.25) .. controls (453.93,226.25) and (496.6,218.58) .. (497.6,231.58) .. controls (498.6,244.58) and (465.27,238.25) .. (466.6,247.92) .. controls (467.93,257.58) and (504.27,254.92) .. (505.6,266.25) .. controls (506.93,277.58) and (469.27,274.58) .. (465.27,271.58) .. controls (461.27,268.58) and (448.27,255.25) .. (443.93,249.25) .. controls (439.6,243.25) and (370.93,172.25) .. (370.27,165.25) .. controls (369.6,158.25) and (377.1,154) .. (379.6,150.25) .. controls (382.1,146.5) and (500.73,32.81) .. (502.6,31.25) .. controls (504.48,29.69) and (521.6,9.92) .. (531.27,19.92) .. controls (540.93,29.92) and (505.6,42.92) .. (507.6,50.25) .. controls (509.6,57.58) and (532.6,49.58) .. (533.6,61.25) -- cycle ; 
\draw  [fill={rgb, 255:red, 155; green, 155; blue, 155 }  ,fill opacity=0.3 ] (184.93,16.92) .. controls (196.93,9.25) and (234.6,59.58) .. (276.27,106.58) .. controls (317.93,153.58) and (377.93,221.25) .. (384.6,228.58) .. controls (391.27,235.92) and (438.27,268.92) .. (424.93,280.25) .. controls (411.6,291.58) and (381.93,253.25) .. (369.93,247.92) .. controls (357.93,242.58) and (332.6,251.58) .. (329.93,239.92) .. controls (327.27,228.25) and (351.17,229.92) .. (348.93,223.92) .. controls (346.69,217.91) and (296.44,162.69) .. (287.93,153.92) .. controls (279.43,145.14) and (235.6,93.25) .. (220.6,77.92) .. controls (205.6,62.58) and (180.27,84.25) .. (176.6,70.25) .. controls (172.93,56.25) and (199.27,59.92) .. (197.93,52.58) .. controls (196.6,45.25) and (172.93,24.58) .. (184.93,16.92) -- cycle ;
\draw    (387.92,154.97) -- (403.44,140.44) ;
\draw [shift={(404.9,139.08)}, rotate = 136.89] [color={rgb, 255:red, 0; green, 0; blue, 0 }  ][line width=0.75]    (6.56,-1.97) .. controls (4.17,-0.84) and (1.99,-0.18) .. (0,0) .. controls (1.99,0.18) and (4.17,0.84) .. (6.56,1.97)   ;
\draw  [dash pattern={on 0.84pt off 2.51pt}]  (425.15,118.58) -- (442.13,102.68) ;
\draw    (460.13,85.35) -- (475.65,70.82) ;
\draw [shift={(477.11,69.45)}, rotate = 136.89] [color={rgb, 255:red, 0; green, 0; blue, 0 }  ][line width=0.75]    (6.56,-1.97) .. controls (4.17,-0.84) and (1.99,-0.18) .. (0,0) .. controls (1.99,0.18) and (4.17,0.84) .. (6.56,1.97)   ;
\draw    (390.17,173.47) -- (405.88,190.98) ;
\draw [shift={(407.22,192.47)}, rotate = 228.09] [color={rgb, 255:red, 0; green, 0; blue, 0 }  ][line width=0.75]    (6.56,-1.97) .. controls (4.17,-0.84) and (1.99,-0.18) .. (0,0) .. controls (1.99,0.18) and (4.17,0.84) .. (6.56,1.97)   ; 
\draw  [dash pattern={on 0.84pt off 2.51pt}]  (420.8,207.47) -- (437.85,226.46) ;
\draw    (451.27,242.13) -- (466.98,259.63) ;
\draw [shift={(468.32,261.12)}, rotate = 228.09] [color={rgb, 255:red, 0; green, 0; blue, 0 }  ][line width=0.75]    (6.56,-1.97) .. controls (4.17,-0.84) and (1.99,-0.18) .. (0,0) .. controls (1.99,0.18) and (4.17,0.84) .. (6.56,1.97)   ;
   
\draw  [fill={rgb, 255:red, 0; green, 0; blue, 0 }  ,fill opacity=1 ] (412.73,130.23) .. controls (412.73,128.98) and (413.87,127.97) .. (415.27,127.97) .. controls (416.67,127.97) and (417.8,128.98) .. (417.8,130.23) .. controls (417.8,131.49) and (416.67,132.5) .. (415.27,132.5) .. controls (413.87,132.5) and (412.73,131.49) .. (412.73,130.23) -- cycle ;
   
\draw  [fill={rgb, 255:red, 0; green, 0; blue, 0 }  ,fill opacity=1 ] (447.4,96.23) .. controls (447.4,94.98) and (448.53,93.97) .. (449.93,93.97) .. controls (451.33,93.97) and (452.47,94.98) .. (452.47,96.23) .. controls (452.47,97.49) and (451.33,98.5) .. (449.93,98.5) .. controls (448.53,98.5) and (447.4,97.49) .. (447.4,96.23) -- cycle ;
   
\draw  [fill={rgb, 255:red, 0; green, 0; blue, 0 }  ,fill opacity=1 ] (481.07,62.57) .. controls (481.07,61.31) and (482.2,60.3) .. (483.6,60.3) .. controls (485,60.3) and (486.13,61.31) .. (486.13,62.57) .. controls (486.13,63.82) and (485,64.83) .. (483.6,64.83) .. controls (482.2,64.83) and (481.07,63.82) .. (481.07,62.57) -- cycle ;
   
\draw  [fill={rgb, 255:red, 0; green, 0; blue, 0 }  ,fill opacity=1 ] (411.4,199.23) .. controls (411.4,197.98) and (412.53,196.97) .. (413.93,196.97) .. controls (415.33,196.97) and (416.47,197.98) .. (416.47,199.23) .. controls (416.47,200.49) and (415.33,201.5) .. (413.93,201.5) .. controls (412.53,201.5) and (411.4,200.49) .. (411.4,199.23) -- cycle ;
   
\draw  [fill={rgb, 255:red, 0; green, 0; blue, 0 }  ,fill opacity=1 ] (442.57,233.07) .. controls (442.57,231.81) and (443.7,230.8) .. (445.1,230.8) .. controls (446.5,230.8) and (447.63,231.81) .. (447.63,233.07) .. controls (447.63,234.32) and (446.5,235.33) .. (445.1,235.33) .. controls (443.7,235.33) and (442.57,234.32) .. (442.57,233.07) -- cycle ;
  
\draw    (459.53,232.9) -- (473.4,232.61) ;
\draw [shift={(475.4,232.57)}, rotate = 178.8] [color={rgb, 255:red, 0; green, 0; blue, 0 }  ][line width=0.75]    (6.56,-1.97) .. controls (4.17,-0.84) and (1.99,-0.18) .. (0,0) .. controls (1.99,0.18) and (4.17,0.84) .. (6.56,1.97)   ;
   
\draw  [fill={rgb, 255:red, 0; green, 0; blue, 0 }  ,fill opacity=1 ] (480.73,232.23) .. controls (480.73,230.98) and (481.87,229.97) .. (483.27,229.97) .. controls (484.67,229.97) and (485.8,230.98) .. (485.8,232.23) .. controls (485.8,233.49) and (484.67,234.5) .. (483.27,234.5) .. controls (481.87,234.5) and (480.73,233.49) .. (480.73,232.23) -- cycle ;
  
\draw  [dash pattern={on 4.5pt off 4.5pt}]  (430.07,266.97) -- (462.27,266.58) ;
  
\draw    (495.63,50.35) -- (511.15,35.82) ;
\draw [shift={(512.61,34.45)}, rotate = 136.89] [color={rgb, 255:red, 0; green, 0; blue, 0 }  ][line width=0.75]    (6.56,-1.97) .. controls (4.17,-0.84) and (1.99,-0.18) .. (0,0) .. controls (1.99,0.18) and (4.17,0.84) .. (6.56,1.97)   ;
   
\draw  [fill={rgb, 255:red, 0; green, 0; blue, 0 }  ,fill opacity=1 ] (516.57,27.57) .. controls (516.57,26.31) and (517.7,25.3) .. (519.1,25.3) .. controls (520.5,25.3) and (521.63,26.31) .. (521.63,27.57) .. controls (521.63,28.82) and (520.5,29.83) .. (519.1,29.83) .. controls (517.7,29.83) and (516.57,28.82) .. (516.57,27.57) -- cycle ;
  
\draw    (495.37,61.9) -- (509.23,61.61) ;
\draw [shift={(511.23,61.57)}, rotate = 178.8] [color={rgb, 255:red, 0; green, 0; blue, 0 }  ][line width=0.75]    (6.56,-1.97) .. controls (4.17,-0.84) and (1.99,-0.18) .. (0,0) .. controls (1.99,0.18) and (4.17,0.84) .. (6.56,1.97)   ;
   
\draw  [fill={rgb, 255:red, 0; green, 0; blue, 0 }  ,fill opacity=1 ] (516.57,61.23) .. controls (516.57,59.98) and (517.7,58.97) .. (519.1,58.97) .. controls (520.5,58.97) and (521.63,59.98) .. (521.63,61.23) .. controls (521.63,62.49) and (520.5,63.5) .. (519.1,63.5) .. controls (517.7,63.5) and (516.57,62.49) .. (516.57,61.23) -- cycle ;
  
\draw    (421.6,260.25) -- (438.25,242.42) ;
\draw [shift={(439.61,240.95)}, rotate = 133.03] [color={rgb, 255:red, 0; green, 0; blue, 0 }  ][line width=0.75]    (6.56,-1.97) .. controls (4.17,-0.84) and (1.99,-0.18) .. (0,0) .. controls (1.99,0.18) and (4.17,0.84) .. (6.56,1.97)   ;
  
\draw    (352.91,190.67) -- (368.44,176.14) ;
\draw [shift={(369.9,174.77)}, rotate = 136.89] [color={rgb, 255:red, 0; green, 0; blue, 0 }  ][line width=0.75]    (6.56,-1.97) .. controls (4.17,-0.84) and (1.99,-0.18) .. (0,0) .. controls (1.99,0.18) and (4.17,0.84) .. (6.56,1.97)   ;
  
\draw    (319.97,174.02) -- (335.68,191.53) ;
\draw [shift={(337.02,193.01)}, rotate = 228.09] [color={rgb, 255:red, 0; green, 0; blue, 0 }  ][line width=0.75]    (6.56,-1.97) .. controls (4.17,-0.84) and (1.99,-0.18) .. (0,0) .. controls (1.99,0.18) and (4.17,0.84) .. (6.56,1.97)   ;
  
\draw  [dash pattern={on 0.84pt off 2.51pt}]  (351.27,208.01) -- (368.32,227.01) ;
  
\draw    (382.73,242.01) -- (398.45,259.51) ;
\draw [shift={(399.78,261)}, rotate = 228.09] [color={rgb, 255:red, 0; green, 0; blue, 0 }  ][line width=0.75]    (6.56,-1.97) .. controls (4.17,-0.84) and (1.99,-0.18) .. (0,0) .. controls (1.99,0.18) and (4.17,0.84) .. (6.56,1.97)   ;
  
\draw    (288.52,137.93) -- (304.23,155.44) ;
\draw [shift={(305.57,156.93)}, rotate = 228.09] [color={rgb, 255:red, 0; green, 0; blue, 0 }  ][line width=0.75]    (6.56,-1.97) .. controls (4.17,-0.84) and (1.99,-0.18) .. (0,0) .. controls (1.99,0.18) and (4.17,0.84) .. (6.56,1.97)   ;
   
\draw  [fill={rgb, 255:red, 0; green, 0; blue, 0 }  ,fill opacity=1 ] (279.13,130.65) .. controls (279.13,129.39) and (280.27,128.38) .. (281.67,128.38) .. controls (283.07,128.38) and (284.2,129.39) .. (284.2,130.65) .. controls (284.2,131.9) and (283.07,132.91) .. (281.67,132.91) .. controls (280.27,132.91) and (279.13,131.9) .. (279.13,130.65) -- cycle ;
   
\draw  [fill={rgb, 255:red, 0; green, 0; blue, 0 }  ,fill opacity=1 ] (253.87,102.11) .. controls (253.87,100.86) and (255,99.85) .. (256.4,99.85) .. controls (257.8,99.85) and (258.93,100.86) .. (258.93,102.11) .. controls (258.93,103.37) and (257.8,104.38) .. (256.4,104.38) .. controls (255,104.38) and (253.87,103.37) .. (253.87,102.11) -- cycle ;
   
\draw  [fill={rgb, 255:red, 0; green, 0; blue, 0 }  ,fill opacity=1 ] (189.07,29.11) .. controls (189.07,27.86) and (190.2,26.85) .. (191.6,26.85) .. controls (193,26.85) and (194.13,27.86) .. (194.13,29.11) .. controls (194.13,30.37) and (193,31.38) .. (191.6,31.38) .. controls (190.2,31.38) and (189.07,30.37) .. (189.07,29.11) -- cycle ;
   
\draw  [fill={rgb, 255:red, 0; green, 0; blue, 0 }  ,fill opacity=1 ] (341.2,199.78) .. controls (341.2,198.53) and (342.33,197.51) .. (343.73,197.51) .. controls (345.13,197.51) and (346.27,198.53) .. (346.27,199.78) .. controls (346.27,201.03) and (345.13,202.05) .. (343.73,202.05) .. controls (342.33,202.05) and (341.2,201.03) .. (341.2,199.78) -- cycle ;
   
\draw  [fill={rgb, 255:red, 0; green, 0; blue, 0 }  ,fill opacity=1 ] (372.7,234.28) .. controls (372.7,233.03) and (373.83,232.01) .. (375.23,232.01) .. controls (376.63,232.01) and (377.77,233.03) .. (377.77,234.28) .. controls (377.77,235.53) and (376.63,236.55) .. (375.23,236.55) .. controls (373.83,236.55) and (372.7,235.53) .. (372.7,234.28) -- cycle ;
  
\draw    (388.33,234.11) -- (402.2,233.82) ;
\draw [shift={(404.2,233.78)}, rotate = 178.8] [color={rgb, 255:red, 0; green, 0; blue, 0 }  ][line width=0.75]    (6.56,-1.97) .. controls (4.17,-0.84) and (1.99,-0.18) .. (0,0) .. controls (1.99,0.18) and (4.17,0.84) .. (6.56,1.97)   ;
   
\draw  [fill={rgb, 255:red, 0; green, 0; blue, 0 }  ,fill opacity=1 ] (409.53,233.45) .. controls (409.53,232.19) and (410.67,231.18) .. (412.07,231.18) .. controls (413.47,231.18) and (414.6,232.19) .. (414.6,233.45) .. controls (414.6,234.7) and (413.47,235.71) .. (412.07,235.71) .. controls (410.67,235.71) and (409.53,234.7) .. (409.53,233.45) -- cycle ;
   
\draw  [fill={rgb, 255:red, 0; green, 0; blue, 0 }  ,fill opacity=1 ] (189.97,65.05) .. controls (189.97,63.79) and (191.1,62.78) .. (192.5,62.78) .. controls (193.9,62.78) and (195.03,63.79) .. (195.03,65.05) .. controls (195.03,66.3) and (193.9,67.31) .. (192.5,67.31) .. controls (191.1,67.31) and (189.97,66.3) .. (189.97,65.05) -- cycle ;
  
\draw    (200.27,65.58) -- (214.5,65.64) ;
\draw [shift={(216.5,65.65)}, rotate = 180.22] [color={rgb, 255:red, 0; green, 0; blue, 0 }  ][line width=0.75]    (6.56,-1.97) .. controls (4.17,-0.84) and (1.99,-0.18) .. (0,0) .. controls (1.99,0.18) and (4.17,0.84) .. (6.56,1.97)   ;
   
\draw  [fill={rgb, 255:red, 0; green, 0; blue, 0 }  ,fill opacity=1 ] (220.83,65.98) .. controls (220.83,64.73) and (221.97,63.71) .. (223.37,63.71) .. controls (224.77,63.71) and (225.9,64.73) .. (225.9,65.98) .. controls (225.9,67.23) and (224.77,68.25) .. (223.37,68.25) .. controls (221.97,68.25) and (220.83,67.23) .. (220.83,65.98) -- cycle ;
  
\draw    (420.53,233.77) -- (434.4,233.48) ;
\draw [shift={(436.4,233.43)}, rotate = 178.8] [color={rgb, 255:red, 0; green, 0; blue, 0 }  ][line width=0.75]    (6.56,-1.97) .. controls (4.17,-0.84) and (1.99,-0.18) .. (0,0) .. controls (1.99,0.18) and (4.17,0.84) .. (6.56,1.97)   ;
   
\draw  [fill={rgb, 255:red, 0; green, 0; blue, 0 }  ,fill opacity=1 ] (339.13,236.65) .. controls (339.13,235.39) and (340.27,234.38) .. (341.67,234.38) .. controls (343.07,234.38) and (344.2,235.39) .. (344.2,236.65) .. controls (344.2,237.9) and (343.07,238.91) .. (341.67,238.91) .. controls (340.27,238.91) and (339.13,237.9) .. (339.13,236.65) -- cycle ;
  
\draw    (353.13,235.3) -- (367,235.01) ;
\draw [shift={(369,234.97)}, rotate = 178.8] [color={rgb, 255:red, 0; green, 0; blue, 0 }  ][line width=0.75]    (6.56,-1.97) .. controls (4.17,-0.84) and (1.99,-0.18) .. (0,0) .. controls (1.99,0.18) and (4.17,0.84) .. (6.56,1.97)   ;
  
\draw  [dash pattern={on 0.84pt off 2.51pt}]  (275.55,123.38) -- (262,108.33) ;
  
\draw    (233.3,75.74) -- (249.01,93.25) ;
\draw [shift={(250.35,94.73)}, rotate = 228.09] [color={rgb, 255:red, 0; green, 0; blue, 0 }  ][line width=0.75]    (6.56,-1.97) .. controls (4.17,-0.84) and (1.99,-0.18) .. (0,0) .. controls (1.99,0.18) and (4.17,0.84) .. (6.56,1.97)   ;
  
\draw    (199.03,37.61) -- (214.75,55.11) ;
\draw [shift={(216.08,56.6)}, rotate = 228.09] [color={rgb, 255:red, 0; green, 0; blue, 0 }  ][line width=0.75]    (6.56,-1.97) .. controls (4.17,-0.84) and (1.99,-0.18) .. (0,0) .. controls (1.99,0.18) and (4.17,0.84) .. (6.56,1.97)   ;
   
\draw  [fill={rgb, 255:red, 0; green, 0; blue, 0 }  ,fill opacity=1 ] (378.2,162.83) .. controls (378.2,161.58) and (379.33,160.57) .. (380.73,160.57) .. controls (382.13,160.57) and (383.27,161.58) .. (383.27,162.83) .. controls (383.27,164.09) and (382.13,165.1) .. (380.73,165.1) .. controls (379.33,165.1) and (378.2,164.09) .. (378.2,162.83) -- cycle ;
  
\draw  [dash pattern={on 0.84pt off 2.51pt}]  (319.85,65) -- (386.6,65.92) ;
  
\draw  [dash pattern={on 0.84pt off 2.51pt}]  (132.85,139.33) -- (199.6,140.25) ;
  
\draw  [dash pattern={on 0.84pt off 2.51pt}]  (518.85,140.67) -- (586.6,140.58) ;
  
\draw  [dash pattern={on 0.84pt off 2.51pt}]  (130.18,200) -- (196.93,200.92) ;
  
\draw  [dash pattern={on 0.84pt off 2.51pt}]  (519.18,206) -- (586.93,206.25) ;

\draw (299.66,162.26) node [anchor=north west][inner sep=0.75pt]  [font=\tiny]  {$E_{L} \langle \vec{x} \rangle $};
 
\draw (464.32,263.19) node [anchor=north west][inner sep=0.75pt]  [font=\tiny]  {$L(\vec{x} -\vec{\omega })$};
 
\draw (401.78,264.4) node [anchor=north west][inner sep=0.75pt]  [font=\tiny]  {$L(\vec{x})$};
 
\draw (114.18,40.4) node [anchor=north west][inner sep=0.75pt]  [font=\normalsize]  {$\Sigma _{\leftarrow } (L(\vec{x}))$};
 
\draw (537.18,33.73) node [anchor=north west][inner sep=0.75pt]    {$\Sigma _{\rightarrow } (\tau ^{-1} E\langle \vec{x} \rangle) $};

\end{tikzpicture}

\end{adjustbox}\]
It follows that ${\rm Hom}(X,L(\vec{x}))= 0$. Therefore, applying ${\rm Hom}(X,-)$ to the short exact sequence
\[\begin{tikzcd}[ampersand replacement=\&,cramped,sep=small]
	0 \& {L(\vec{\omega})} \& {E_L\langle\vec{x}\rangle} \& {L(\vec{x})} \& 0
	\arrow["{{{\eta}}}", from=1-2, to=1-3]
	\arrow["{{{{\pi}}}}", from=1-3, to=1-4]
	\arrow[from=1-4, to=1-5]
	\arrow[from=1-1, to=1-2]
,\end{tikzcd}\]
we obtain an exact sequence:
		\[\begin{tikzcd}[ampersand replacement=\&,cramped,sep=small]
	0 \& {{\rm Ext^{1}}(X,L(\vec{\omega}))} \& {{\rm Ext^{1}}(X,Y)} \& {{\rm Ext^{1}}(X,L(\vec{x}))} \& 0
	\arrow[from=1-2, to=1-3]
	\arrow[from=1-3, to=1-4]
	\arrow[from=1-4, to=1-5]
	\arrow[from=1-1, to=1-2]
\end{tikzcd},\]
   and then the remaining part of (1) follows.
	\end{proof}

	\begin{thm}{}{}\label{dimension and positive intersection}
 Let $X$ and $Y$ be two indecomposable bundles in ${\rm vect}\mbox{-}\mathbb{X} $. Then
		$${\rm dim_{k}Ext^{1}}(X,Y)=I_{\rm{int}}(\widehat{\phi}^{-1}(X),\widehat{\phi}^{-1}(Y)).$$
	\end{thm}
\begin{proof}
   Each indecomposable bundle in ${\rm vect}\mbox{-}\mathbb{X} $ is either a line bundle of rank one or an extension bundle of rank two. We list the cases according to the ranks of $X$ and $Y$ as follows.
\begin{itemize}
\setlength{\itemindent}{0em}
\item[\emph{Case 1:}]  $\mathrm{rk} X=\mathrm{rk} Y=1$. Then $X=L_0(\vec{x})$ and $Y=L_0(\vec{y})$ for some $\vec{x},\,\vec{y}\in\mathbb{L}.$ Assume $\vec{x}=l_1(\vec{x}_1-\vec{x}_2)+l_2\vec{x}_2+l \vec{x}_3$ and $\vec{y}=k_1(\vec{x}_1-\vec{x}_2)+k_2\vec{x}_2+k \vec{x}_3$, where $l_1,l_2,k_1,k_2\in \{0,1\}$ and $l ,k \in \mathbb{Z}$.
		According to \S \ref{Section 2},
		we have
		$$\Ext^{1}(L_0(\vec{x}),L_0(\vec{y}))\cong D\Hom(L_0(\vec{y}), L_0(\vec{x}+\vec{\omega}))\cong DS_{\vec{x}+\vec{\omega}-\vec{y}}.$$
		Since  $\vec{x}+\vec{\omega}-\vec{y}=(l_1-k_1+1)\vec{x}_1+(l_2-l_1-k_2+k_1-1)\vec{x}_2+(l -k -1)\vec{x}_3$, we get
		\begin{equation}
			{\rm dim_{k}Ext^{1}}(X,Y)=
			\left\{
			\begin{array}{ll}
   \vspace{0.5em}
				\max\{0,\lfloor\frac{l -k -1}{n}\rfloor\}   & \;\text{if}\;\, l_2=k_2\ \text{and}\ l_1=k_1;\nonumber \\	
    \vspace{0.5em}
				\max\{0,\lfloor\frac{l -k -1}{n}\rfloor+1\} & \;\text{if}\;\,l_2=k_2 \ \text{and}\ l_1\neq k_1;      \\	
    \vspace{0.5em}
				\max\{0,\lfloor\frac{l -k -1}{n}\rfloor+1\} & \;\text{if}\;\,l_2>k_2;                                \\	 
				\max\{0,\lfloor\frac{l -k -1}{n}\rfloor\}   & \;\text{if}\;\,l_2<k_2.
			\end{array}
			\right.\end{equation}
\begin{itemize}[leftmargin=0em]
    \item [(i)] If $l_1=k_1$ and $l_2=k_2$, without loss of generality, we may assume $l_1=k_1=1$ and $l_2=k_2=0$. In this case, we have 
\[I_{\text{int}}(\widehat{\phi}^{-1}(X),\widehat{\phi}^{-1}(Y))=I_{\text{int}}(\widetilde{[-l-1,l+1]^+},\widetilde{[-k-1,k+1]^+}),\] and it equals the number of integers $m$ satisfying $-k-1 +mn > -l-1 $ and $mn < 0$, as illustrated in the figure below:
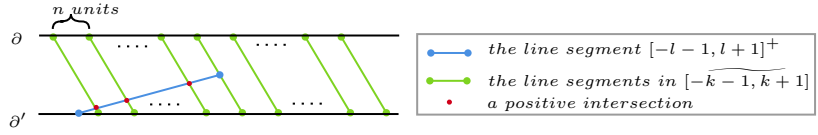
\begin{figure}[H]
    \centering

\tikzset{every picture/.style={line width=0.75pt}}          

\begin{tikzpicture}[x=0.75pt,y=0.75pt,yscale=-1,xscale=1]

\draw [color={rgb, 255:red, 126; green, 211; blue, 33 }  ,draw opacity=1 ]   (392.86,180.47) -- (415.34,218.55) ;
\draw [shift={(415.34,218.55)}, rotate = 59.44] [color={rgb, 255:red, 126; green, 211; blue, 33 }  ,draw opacity=1 ][fill={rgb, 255:red, 126; green, 211; blue, 33 }  ,fill opacity=1 ][line width=0.75]      (0, 0) circle [x radius= 1.34, y radius= 1.34]   ;
\draw [shift={(392.86,180.47)}, rotate = 59.44] [color={rgb, 255:red, 126; green, 211; blue, 33 }  ,draw opacity=1 ][fill={rgb, 255:red, 126; green, 211; blue, 33 }  ,fill opacity=1 ][line width=0.75]      (0, 0) circle [x radius= 1.34, y radius= 1.34]   ;
  
\draw [color={rgb, 255:red, 126; green, 211; blue, 33 }  ,draw opacity=1 ]   (374.88,180.47) -- (397.37,218.55) ;
\draw [shift={(397.37,218.55)}, rotate = 59.44] [color={rgb, 255:red, 126; green, 211; blue, 33 }  ,draw opacity=1 ][fill={rgb, 255:red, 126; green, 211; blue, 33 }  ,fill opacity=1 ][line width=0.75]      (0, 0) circle [x radius= 1.34, y radius= 1.34]   ;
\draw [shift={(374.88,180.47)}, rotate = 59.44] [color={rgb, 255:red, 126; green, 211; blue, 33 }  ,draw opacity=1 ][fill={rgb, 255:red, 126; green, 211; blue, 33 }  ,fill opacity=1 ][line width=0.75]      (0, 0) circle [x radius= 1.34, y radius= 1.34]   ;
  
\draw [color={rgb, 255:red, 126; green, 211; blue, 33 }  ,draw opacity=1 ]   (338.93,180.47) -- (361.42,218.55) ;
\draw [shift={(361.42,218.55)}, rotate = 59.44] [color={rgb, 255:red, 126; green, 211; blue, 33 }  ,draw opacity=1 ][fill={rgb, 255:red, 126; green, 211; blue, 33 }  ,fill opacity=1 ][line width=0.75]      (0, 0) circle [x radius= 1.34, y radius= 1.34]   ;
\draw [shift={(338.93,180.47)}, rotate = 59.44] [color={rgb, 255:red, 126; green, 211; blue, 33 }  ,draw opacity=1 ][fill={rgb, 255:red, 126; green, 211; blue, 33 }  ,fill opacity=1 ][line width=0.75]      (0, 0) circle [x radius= 1.34, y radius= 1.34]   ;
  
\draw [color={rgb, 255:red, 126; green, 211; blue, 33 }  ,draw opacity=1 ]   (320.96,180.47) -- (343.44,218.55) ;
\draw [shift={(343.44,218.55)}, rotate = 59.44] [color={rgb, 255:red, 126; green, 211; blue, 33 }  ,draw opacity=1 ][fill={rgb, 255:red, 126; green, 211; blue, 33 }  ,fill opacity=1 ][line width=0.75]      (0, 0) circle [x radius= 1.34, y radius= 1.34]   ;
\draw [shift={(320.96,180.47)}, rotate = 59.44] [color={rgb, 255:red, 126; green, 211; blue, 33 }  ,draw opacity=1 ][fill={rgb, 255:red, 126; green, 211; blue, 33 }  ,fill opacity=1 ][line width=0.75]      (0, 0) circle [x radius= 1.34, y radius= 1.34]   ;
  
\draw [color={rgb, 255:red, 126; green, 211; blue, 33 }  ,draw opacity=1 ]   (267.03,180.47) -- (289.52,218.55) ;
\draw [shift={(289.52,218.55)}, rotate = 59.44] [color={rgb, 255:red, 126; green, 211; blue, 33 }  ,draw opacity=1 ][fill={rgb, 255:red, 126; green, 211; blue, 33 }  ,fill opacity=1 ][line width=0.75]      (0, 0) circle [x radius= 1.34, y radius= 1.34]   ;
\draw [shift={(267.03,180.47)}, rotate = 59.44] [color={rgb, 255:red, 126; green, 211; blue, 33 }  ,draw opacity=1 ][fill={rgb, 255:red, 126; green, 211; blue, 33 }  ,fill opacity=1 ][line width=0.75]      (0, 0) circle [x radius= 1.34, y radius= 1.34]   ;
  
\draw [color={rgb, 255:red, 126; green, 211; blue, 33 }  ,draw opacity=1 ]   (302.98,180.47) -- (325.47,218.55) ;
\draw [shift={(325.47,218.55)}, rotate = 59.44] [color={rgb, 255:red, 126; green, 211; blue, 33 }  ,draw opacity=1 ][fill={rgb, 255:red, 126; green, 211; blue, 33 }  ,fill opacity=1 ][line width=0.75]      (0, 0) circle [x radius= 1.34, y radius= 1.34]   ;
\draw [shift={(302.98,180.47)}, rotate = 59.44] [color={rgb, 255:red, 126; green, 211; blue, 33 }  ,draw opacity=1 ][fill={rgb, 255:red, 126; green, 211; blue, 33 }  ,fill opacity=1 ][line width=0.75]      (0, 0) circle [x radius= 1.34, y radius= 1.34]   ;
  
\draw [color={rgb, 255:red, 126; green, 211; blue, 33 }  ,draw opacity=1 ]   (249.06,180.47) -- (271.54,218.55) ;
\draw [shift={(271.54,218.55)}, rotate = 59.44] [color={rgb, 255:red, 126; green, 211; blue, 33 }  ,draw opacity=1 ][fill={rgb, 255:red, 126; green, 211; blue, 33 }  ,fill opacity=1 ][line width=0.75]      (0, 0) circle [x radius= 1.34, y radius= 1.34]   ;
\draw [shift={(249.06,180.47)}, rotate = 59.44] [color={rgb, 255:red, 126; green, 211; blue, 33 }  ,draw opacity=1 ][fill={rgb, 255:red, 126; green, 211; blue, 33 }  ,fill opacity=1 ][line width=0.75]      (0, 0) circle [x radius= 1.34, y radius= 1.34]   ;
  
\draw    (242,179.71) -- (421.75,179.71) ;
  
\draw    (241.67,219.21) -- (421.42,219.21) ;
   
\draw   (267,176.38) .. controls (266.99,173.88) and (265.74,172.63) .. (263.25,172.64) -- (263.25,172.64) .. controls (259.69,172.65) and (257.91,171.4) .. (257.9,168.91) .. controls (257.91,171.4) and (256.13,172.65) .. (252.57,172.66)(254.17,172.66) -- (252.57,172.66) .. controls (250.08,172.67) and (248.83,173.92) .. (248.83,176.41) ;
  
\draw [color={rgb, 255:red, 74; green, 144; blue, 226 }  ,draw opacity=1 ]   (261.8,218.78) -- (332.2,199.51) ;
\draw [shift={(332.2,199.51)}, rotate = 344.69] [color={rgb, 255:red, 74; green, 144; blue, 226 }  ,draw opacity=1 ][fill={rgb, 255:red, 74; green, 144; blue, 226 }  ,fill opacity=1 ][line width=0.75]      (0, 0) circle [x radius= 1.34, y radius= 1.34]   ;
\draw [shift={(261.8,218.78)}, rotate = 344.69] [color={rgb, 255:red, 74; green, 144; blue, 226 }  ,draw opacity=1 ][fill={rgb, 255:red, 74; green, 144; blue, 226 }  ,fill opacity=1 ][line width=0.75]      (0, 0) circle [x radius= 1.34, y radius= 1.34]   ;
  
\draw  [dash pattern={on 0.84pt off 2.51pt}]  (350.6,184.38) -- (365.8,184.38) ;
  
\draw  [dash pattern={on 0.84pt off 2.51pt}]  (297.4,214.38) -- (312.6,214.38) ;
  
\draw  [dash pattern={on 0.84pt off 2.51pt}]  (368.6,213.98) -- (383.8,213.98) ;
  
\draw  [dash pattern={on 0.84pt off 2.51pt}]  (281.4,185.18) -- (296.6,185.18) ;
  
\draw  [color={rgb, 255:red, 208; green, 2; blue, 27 }  ,draw opacity=1 ][fill={rgb, 255:red, 208; green, 2; blue, 27 }  ,fill opacity=1 ] (269.73,215.87) .. controls (269.73,215.45) and (270.08,215.11) .. (270.5,215.11) .. controls (270.92,215.11) and (271.27,215.45) .. (271.27,215.87) .. controls (271.27,216.3) and (270.92,216.64) .. (270.5,216.64) .. controls (270.08,216.64) and (269.73,216.3) .. (269.73,215.87) -- cycle ;
  
\draw  [color={rgb, 255:red, 208; green, 2; blue, 27 }  ,draw opacity=1 ][fill={rgb, 255:red, 208; green, 2; blue, 27 }  ,fill opacity=1 ] (316.23,203.87) .. controls (316.23,203.45) and (316.58,203.11) .. (317,203.11) .. controls (317.42,203.11) and (317.77,203.45) .. (317.77,203.87) .. controls (317.77,204.3) and (317.42,204.64) .. (317,204.64) .. controls (316.58,204.64) and (316.23,204.3) .. (316.23,203.87) -- cycle ;
  
\draw  [color={rgb, 255:red, 208; green, 2; blue, 27 }  ,draw opacity=1 ][fill={rgb, 255:red, 208; green, 2; blue, 27 }  ,fill opacity=1 ] (284.98,212.37) .. controls (284.98,211.95) and (285.33,211.61) .. (285.75,211.61) .. controls (286.17,211.61) and (286.52,211.95) .. (286.52,212.37) .. controls (286.52,212.8) and (286.17,213.14) .. (285.75,213.14) .. controls (285.33,213.14) and (284.98,212.8) .. (284.98,212.37) -- cycle ;
  
\draw [color={rgb, 255:red, 255; green, 255; blue, 255 }  ,draw opacity=1 ]   (197,182) -- (197.33,219) ;
  
\draw  [color={rgb, 255:red, 208; green, 2; blue, 27 }  ,draw opacity=1 ][fill={rgb, 255:red, 208; green, 2; blue, 27 }  ,fill opacity=1 ] (446.04,213.21) .. controls (446.04,212.78) and (446.38,212.44) .. (446.81,212.44) .. controls (447.23,212.44) and (447.57,212.78) .. (447.57,213.21) .. controls (447.57,213.63) and (447.23,213.97) .. (446.81,213.97) .. controls (446.38,213.97) and (446.04,213.63) .. (446.04,213.21) -- cycle ;
  
\draw [color={rgb, 255:red, 74; green, 144; blue, 226 }  ,draw opacity=1 ]   (437.56,188.5) -- (456.06,188.5) ;
\draw [shift={(456.06,188.5)}, rotate = 0] [color={rgb, 255:red, 74; green, 144; blue, 226 }  ,draw opacity=1 ][fill={rgb, 255:red, 74; green, 144; blue, 226 }  ,fill opacity=1 ][line width=0.75]      (0, 0) circle [x radius= 1.34, y radius= 1.34]   ;
\draw [shift={(437.56,188.5)}, rotate = 0] [color={rgb, 255:red, 74; green, 144; blue, 226 }  ,draw opacity=1 ][fill={rgb, 255:red, 74; green, 144; blue, 226 }  ,fill opacity=1 ][line width=0.75]      (0, 0) circle [x radius= 1.34, y radius= 1.34]   ;
  
\draw [color={rgb, 255:red, 126; green, 211; blue, 33 }  ,draw opacity=1 ]   (437.06,202.5) -- (455.56,202.5) ;
\draw [shift={(455.56,202.5)}, rotate = 0] [color={rgb, 255:red, 126; green, 211; blue, 33 }  ,draw opacity=1 ][fill={rgb, 255:red, 126; green, 211; blue, 33 }  ,fill opacity=1 ][line width=0.75]      (0, 0) circle [x radius= 1.34, y radius= 1.34]   ;
\draw [shift={(437.06,202.5)}, rotate = 0] [color={rgb, 255:red, 126; green, 211; blue, 33 }  ,draw opacity=1 ][fill={rgb, 255:red, 126; green, 211; blue, 33 }  ,fill opacity=1 ][line width=0.75]      (0, 0) circle [x radius= 1.34, y radius= 1.34]   ;

\draw  [color={rgb, 255:red, 155; green, 155; blue, 155 }  ,draw opacity=1 ] (430.72,179) -- (629.33,179) -- (629.33,219.5) -- (430.72,219.5) -- cycle ;
\draw (225.33,217.11) node [anchor=north west][inner sep=0.75pt]  [font=\tiny]  {$\partial ^{\prime }$};
 \draw (226.33,177.11) node [anchor=north west][inner sep=0.75pt]  [font=\tiny]  {$\partial $};
 \draw (247.07,162.01) node [anchor=north west][inner sep=0.75pt]  [font=\tiny]  {$n\ units$};
 \draw (464.22,209.4) node [anchor=north west][inner sep=0.75pt]  [font=\tiny]  {$a\ positive\ intersection$};
 \draw (464.72,179.73) node [anchor=north west][inner sep=0.75pt]  [font=\tiny]  {$the\ line\ segment\ [ -l-1,l+1]^{+}$};
 \draw (464.72,193.73) node [anchor=north west][inner sep=0.75pt]  [font=\tiny]  {$the\ line\ segments\ in\ \widetilde{[ -k-1,k+1]}$};
\end{tikzpicture}
\end{figure}
\noindent  Consequently, we have
		\[I_{\text{int}}(\widehat{\phi}^{-1}(X),\widehat{\phi}^{-1}(Y))=\max\{0,\lfloor\frac{l -k -1}{n}\rfloor\}={\rm dim_{k}Ext^{1}}(X,Y).\]
		
\item [(ii)] If $l_1\neq k_1$ and $l_2=k_2$, without loss of generality, we may assume $l_1=1, k_1=0$ and $l_2=k_2=0$. In this case, we have
  \[I_{\text{int}}(\widehat{\phi}^{-1}(X),\widehat{\phi}^{-1}(Y))=I_{\text{int}}(\widetilde{[-l-1,l+1]^+},\widetilde{[-k-1,k+1]^-}).\]
  Since $+=\varrho(-)$ and $(-l-1)+(l+1)+(k+1)+(-k-1)\equiv 0\bmod 2n$, by Definition \ref{positive intersections }, the intersection index equals $1$ plus the number of integers $m$ satisfying $-k-1 +mn > -l-1 $ and $mn < 0 $. It follows that
		\[I_{\text{int}}(\widehat{\phi}^{-1}(X),\widehat{\phi}^{-1}(Y))=\max\{0,\lfloor\frac{l -k -1}{n}\rfloor+1\}={\rm dim_{k}Ext^{1}}(X,Y).\]

\item [(iii)] If $l_2>k_2$, then $l_2=1$ and $k_2=0$. Without loss of generality, We may assume $l_1=k_1=1$. In this case, we have
\[I_{\text{int}}(\widehat{\phi}^{-1}(X),\widehat{\phi}^{-1}(Y))=I_{\text{int}}(\widetilde{[-l-1,n+l+1]^+},\widetilde{[-k-1,k+1]^+}),\] and it equals the number of integers $m$ satisfying $-k-1 +mn > -l-1 $ and $mn < \frac{n}{2} $, as illustrated in the figure below:
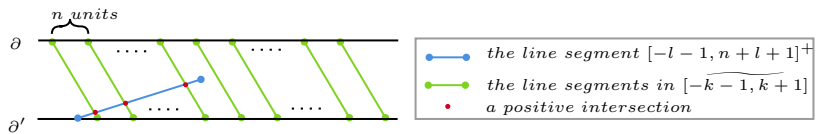
\begin{figure}[H]
    \centering

\tikzset{every picture/.style={line width=0.75pt}}          

\begin{tikzpicture}[x=0.75pt,y=0.75pt,yscale=-1,xscale=1]

\draw  [color={rgb, 255:red, 208; green, 2; blue, 27 }  ,draw opacity=1 ][fill={rgb, 255:red, 208; green, 2; blue, 27 }  ,fill opacity=1 ] (440.04,92.21) .. controls (440.04,91.78) and (440.38,91.44) .. (440.81,91.44) .. controls (441.23,91.44) and (441.57,91.78) .. (441.57,92.21) .. controls (441.57,92.63) and (441.23,92.97) .. (440.81,92.97) .. controls (440.38,92.97) and (440.04,92.63) .. (440.04,92.21) -- cycle ;
  
\draw [color={rgb, 255:red, 74; green, 144; blue, 226 }  ,draw opacity=1 ]   (431.56,67.5) -- (450.06,67.5) ;
\draw [shift={(450.06,67.5)}, rotate = 0] [color={rgb, 255:red, 74; green, 144; blue, 226 }  ,draw opacity=1 ][fill={rgb, 255:red, 74; green, 144; blue, 226 }  ,fill opacity=1 ][line width=0.75]      (0, 0) circle [x radius= 1.34, y radius= 1.34]   ;
\draw [shift={(431.56,67.5)}, rotate = 0] [color={rgb, 255:red, 74; green, 144; blue, 226 }  ,draw opacity=1 ][fill={rgb, 255:red, 74; green, 144; blue, 226 }  ,fill opacity=1 ][line width=0.75]      (0, 0) circle [x radius= 1.34, y radius= 1.34]   ;
  
\draw [color={rgb, 255:red, 126; green, 211; blue, 33 }  ,draw opacity=1 ]   (431.06,81.5) -- (449.56,81.5) ;
\draw [shift={(449.56,81.5)}, rotate = 0] [color={rgb, 255:red, 126; green, 211; blue, 33 }  ,draw opacity=1 ][fill={rgb, 255:red, 126; green, 211; blue, 33 }  ,fill opacity=1 ][line width=0.75]      (0, 0) circle [x radius= 1.34, y radius= 1.34]   ;
\draw [shift={(431.06,81.5)}, rotate = 0] [color={rgb, 255:red, 126; green, 211; blue, 33 }  ,draw opacity=1 ][fill={rgb, 255:red, 126; green, 211; blue, 33 }  ,fill opacity=1 ][line width=0.75]      (0, 0) circle [x radius= 1.34, y radius= 1.34]   ;

\draw  [color={rgb, 255:red, 155; green, 155; blue, 155 }  ,draw opacity=1 ] (424.72,58) -- (623.33,58) -- (623.33,98.5) -- (424.72,98.5) -- cycle ;
  
\draw [color={rgb, 255:red, 255; green, 255; blue, 255 }  ,draw opacity=1 ]   (191,61) -- (191.33,98) ;
  
\draw [color={rgb, 255:red, 126; green, 211; blue, 33 }  ,draw opacity=1 ]   (387.19,59.76) -- (409.67,97.84) ;
\draw [shift={(409.67,97.84)}, rotate = 59.44] [color={rgb, 255:red, 126; green, 211; blue, 33 }  ,draw opacity=1 ][fill={rgb, 255:red, 126; green, 211; blue, 33 }  ,fill opacity=1 ][line width=0.75]      (0, 0) circle [x radius= 1.34, y radius= 1.34]   ;
\draw [shift={(387.19,59.76)}, rotate = 59.44] [color={rgb, 255:red, 126; green, 211; blue, 33 }  ,draw opacity=1 ][fill={rgb, 255:red, 126; green, 211; blue, 33 }  ,fill opacity=1 ][line width=0.75]      (0, 0) circle [x radius= 1.34, y radius= 1.34]   ;
  
\draw [color={rgb, 255:red, 126; green, 211; blue, 33 }  ,draw opacity=1 ]   (369.22,59.76) -- (391.7,97.84) ;
\draw [shift={(391.7,97.84)}, rotate = 59.44] [color={rgb, 255:red, 126; green, 211; blue, 33 }  ,draw opacity=1 ][fill={rgb, 255:red, 126; green, 211; blue, 33 }  ,fill opacity=1 ][line width=0.75]      (0, 0) circle [x radius= 1.34, y radius= 1.34]   ;
\draw [shift={(369.22,59.76)}, rotate = 59.44] [color={rgb, 255:red, 126; green, 211; blue, 33 }  ,draw opacity=1 ][fill={rgb, 255:red, 126; green, 211; blue, 33 }  ,fill opacity=1 ][line width=0.75]      (0, 0) circle [x radius= 1.34, y radius= 1.34]   ;
  
\draw [color={rgb, 255:red, 126; green, 211; blue, 33 }  ,draw opacity=1 ]   (333.27,59.76) -- (355.75,97.84) ;
\draw [shift={(355.75,97.84)}, rotate = 59.44] [color={rgb, 255:red, 126; green, 211; blue, 33 }  ,draw opacity=1 ][fill={rgb, 255:red, 126; green, 211; blue, 33 }  ,fill opacity=1 ][line width=0.75]      (0, 0) circle [x radius= 1.34, y radius= 1.34]   ;
\draw [shift={(333.27,59.76)}, rotate = 59.44] [color={rgb, 255:red, 126; green, 211; blue, 33 }  ,draw opacity=1 ][fill={rgb, 255:red, 126; green, 211; blue, 33 }  ,fill opacity=1 ][line width=0.75]      (0, 0) circle [x radius= 1.34, y radius= 1.34]   ;
  
\draw [color={rgb, 255:red, 126; green, 211; blue, 33 }  ,draw opacity=1 ]   (315.29,59.76) -- (337.77,97.84) ;
\draw [shift={(337.77,97.84)}, rotate = 59.44] [color={rgb, 255:red, 126; green, 211; blue, 33 }  ,draw opacity=1 ][fill={rgb, 255:red, 126; green, 211; blue, 33 }  ,fill opacity=1 ][line width=0.75]      (0, 0) circle [x radius= 1.34, y radius= 1.34]   ;
\draw [shift={(315.29,59.76)}, rotate = 59.44] [color={rgb, 255:red, 126; green, 211; blue, 33 }  ,draw opacity=1 ][fill={rgb, 255:red, 126; green, 211; blue, 33 }  ,fill opacity=1 ][line width=0.75]      (0, 0) circle [x radius= 1.34, y radius= 1.34]   ;
  
\draw [color={rgb, 255:red, 126; green, 211; blue, 33 }  ,draw opacity=1 ]   (261.37,59.76) -- (283.85,97.84) ;
\draw [shift={(283.85,97.84)}, rotate = 59.44] [color={rgb, 255:red, 126; green, 211; blue, 33 }  ,draw opacity=1 ][fill={rgb, 255:red, 126; green, 211; blue, 33 }  ,fill opacity=1 ][line width=0.75]      (0, 0) circle [x radius= 1.34, y radius= 1.34]   ;
\draw [shift={(261.37,59.76)}, rotate = 59.44] [color={rgb, 255:red, 126; green, 211; blue, 33 }  ,draw opacity=1 ][fill={rgb, 255:red, 126; green, 211; blue, 33 }  ,fill opacity=1 ][line width=0.75]      (0, 0) circle [x radius= 1.34, y radius= 1.34]   ;
  
\draw [color={rgb, 255:red, 126; green, 211; blue, 33 }  ,draw opacity=1 ]   (297.32,59.76) -- (319.8,97.84) ;
\draw [shift={(319.8,97.84)}, rotate = 59.44] [color={rgb, 255:red, 126; green, 211; blue, 33 }  ,draw opacity=1 ][fill={rgb, 255:red, 126; green, 211; blue, 33 }  ,fill opacity=1 ][line width=0.75]      (0, 0) circle [x radius= 1.34, y radius= 1.34]   ;
\draw [shift={(297.32,59.76)}, rotate = 59.44] [color={rgb, 255:red, 126; green, 211; blue, 33 }  ,draw opacity=1 ][fill={rgb, 255:red, 126; green, 211; blue, 33 }  ,fill opacity=1 ][line width=0.75]      (0, 0) circle [x radius= 1.34, y radius= 1.34]   ;
  
\draw [color={rgb, 255:red, 126; green, 211; blue, 33 }  ,draw opacity=1 ]   (243.39,59.76) -- (265.87,97.84) ;
\draw [shift={(265.87,97.84)}, rotate = 59.44] [color={rgb, 255:red, 126; green, 211; blue, 33 }  ,draw opacity=1 ][fill={rgb, 255:red, 126; green, 211; blue, 33 }  ,fill opacity=1 ][line width=0.75]      (0, 0) circle [x radius= 1.34, y radius= 1.34]   ;
\draw [shift={(243.39,59.76)}, rotate = 59.44] [color={rgb, 255:red, 126; green, 211; blue, 33 }  ,draw opacity=1 ][fill={rgb, 255:red, 126; green, 211; blue, 33 }  ,fill opacity=1 ][line width=0.75]      (0, 0) circle [x radius= 1.34, y radius= 1.34]   ;
  
\draw    (236.33,59) -- (416.08,59) ;
  
\draw    (236,98.5) -- (415.75,98.5) ;
   
\draw   (261.33,55.67) .. controls (261.33,53.18) and (260.08,51.93) .. (257.59,51.93) -- (257.59,51.93) .. controls (254.02,51.94) and (252.24,50.69) .. (252.24,48.2) .. controls (252.24,50.69) and (250.46,51.94) .. (246.9,51.95)(248.5,51.95) -- (246.9,51.95) .. controls (244.41,51.96) and (243.16,53.21) .. (243.17,55.7) ;
  
\draw [color={rgb, 255:red, 74; green, 144; blue, 226 }  ,draw opacity=1 ]   (256.13,98.07) -- (317.58,78.59) ;
\draw [shift={(317.58,78.59)}, rotate = 342.41] [color={rgb, 255:red, 74; green, 144; blue, 226 }  ,draw opacity=1 ][fill={rgb, 255:red, 74; green, 144; blue, 226 }  ,fill opacity=1 ][line width=0.75]      (0, 0) circle [x radius= 1.34, y radius= 1.34]   ;
\draw [shift={(256.13,98.07)}, rotate = 342.41] [color={rgb, 255:red, 74; green, 144; blue, 226 }  ,draw opacity=1 ][fill={rgb, 255:red, 74; green, 144; blue, 226 }  ,fill opacity=1 ][line width=0.75]      (0, 0) circle [x radius= 1.34, y radius= 1.34]   ;
  
\draw  [dash pattern={on 0.84pt off 2.51pt}]  (342.93,63.67) -- (358.13,63.67) ;
  
\draw  [dash pattern={on 0.84pt off 2.51pt}]  (291.73,93.67) -- (306.93,93.67) ;
  
\draw  [dash pattern={on 0.84pt off 2.51pt}]  (362.93,93.27) -- (378.13,93.27) ;
  
\draw  [dash pattern={on 0.84pt off 2.51pt}]  (275.73,64.47) -- (290.93,64.47) ;
  
\draw  [color={rgb, 255:red, 208; green, 2; blue, 27 }  ,draw opacity=1 ][fill={rgb, 255:red, 208; green, 2; blue, 27 }  ,fill opacity=1 ] (264.07,95.17) .. controls (264.07,94.74) and (264.41,94.4) .. (264.83,94.4) .. controls (265.26,94.4) and (265.6,94.74) .. (265.6,95.17) .. controls (265.6,95.59) and (265.26,95.93) .. (264.83,95.93) .. controls (264.41,95.93) and (264.07,95.59) .. (264.07,95.17) -- cycle ;
  
\draw  [color={rgb, 255:red, 208; green, 2; blue, 27 }  ,draw opacity=1 ][fill={rgb, 255:red, 208; green, 2; blue, 27 }  ,fill opacity=1 ] (309.16,81.3) .. controls (309.16,80.88) and (309.51,80.54) .. (309.93,80.54) .. controls (310.35,80.54) and (310.7,80.88) .. (310.7,81.3) .. controls (310.7,81.73) and (310.35,82.07) .. (309.93,82.07) .. controls (309.51,82.07) and (309.16,81.73) .. (309.16,81.3) -- cycle ;
  
\draw  [color={rgb, 255:red, 208; green, 2; blue, 27 }  ,draw opacity=1 ][fill={rgb, 255:red, 208; green, 2; blue, 27 }  ,fill opacity=1 ] (279.21,90.59) .. controls (279.21,90.16) and (279.56,89.82) .. (279.98,89.82) .. controls (280.4,89.82) and (280.75,90.16) .. (280.75,90.59) .. controls (280.75,91.01) and (280.4,91.35) .. (279.98,91.35) .. controls (279.56,91.35) and (279.21,91.01) .. (279.21,90.59) -- cycle ;

\draw (458.22,88.4) node [anchor=north west][inner sep=0.75pt]  [font=\tiny]  {$a\ positive\ intersection$};
 
\draw (458.72,58.73) node [anchor=north west][inner sep=0.75pt]  [font=\tiny]  {$the\ line\ segment\ [ -l-1,n+l+1]^{+}$};
 
\draw (458.72,72.73) node [anchor=north west][inner sep=0.75pt]  [font=\tiny]  {$the\ line\ segments\ in\ \widetilde{[ -k-1,k+1]}$};
 
\draw (219.67,96.4) node [anchor=north west][inner sep=0.75pt]  [font=\tiny]  {$\partial ^{\prime }$};
 
\draw (220.67,56.4) node [anchor=north west][inner sep=0.75pt]  [font=\tiny]  {$\partial $};
 
\draw (241.4,41.3) node [anchor=north west][inner sep=0.75pt]  [font=\tiny]  {$n\ units$};

\end{tikzpicture}

\end{figure}
\noindent Consequently, we have
\[{\rm dim_{k}Ext^{1}}(X,Y)=\max\{0,\lfloor\frac{l -k -1}{n}\rfloor+1\}=I_{\text{int}}(\widehat{\phi}^{-1}(X),\widehat{\phi}^{-1}(Y)).\]
		
\item [(iv)] If $l_2<k_2$, then $l_2=0$ and $k_2=1$. The proof parallels (iii).
\end{itemize}	

	\item[\emph{Case 2:}] $\mathrm{rk} X=1$ and $\mathrm{rk} Y=2$. Then   $X=L_0(\vec{x})$ and $Y=E_{L_0(-i\vec{x}_3)}\langle (k-1)\vec{x}_3\rangle$, for some $\vec{x}\in\mathbb{L}$ and $i,k\in \mathbb{Z}$, where $1\leq k\leq n-1$. Write $\vec{x}=l_1(\vec{x}_1-\vec{x}_2)+l_2\vec{x}_2+l \vec{x}_3$, where $l_1,l_2\in \{0,1\}$ and $l \in \mathbb{Z}$. Without loss of generality, we may assume $l_1=1,l_2=0$. From Proposition \ref{intersection and exact seq} Case 3, we see that $I_{\rm{int}}(\widehat{\phi}^{-1}(X),\widehat{\phi}^{-1}(Y))=0$ if ${\rm Ext^{1}}(X,Y)=0$.  Let us consider the case where ${\rm Ext^{1}}(X,Y)\neq 0$. By Proposition \ref{Hom and Ext}(1), we have
		\begin{align*}
			{\rm dim_{k}Ext^{1}}(X,Y)
			& ={\rm dim_{k}Ext^{1}}(X,L_0^*(-(i+1)\vec{x}_3))+{\rm dim_{k}Ext^{1}}(X,L_0((k-i-1)\vec{x}_3))      \\
			& =I_{\text{int}}(\widetilde{[-l-1,l+1]^+},\widetilde{[i,-i]^+})+I_{\text{int}}(\widetilde{[-l-1,l+1]^+},\widetilde{[i-k,k-i]^-})      \\
			& =\max\{0,\lfloor\frac{l+i }{n}\rfloor\}+\max\{0,\lfloor\frac{l -k+i }{n}\rfloor+1\} \\
			& =|\{m\in \mathbb{Z}\;|\; i+mn > -l-1  \;\,\mbox{and}\;\, k-i+mn < l+1 \}|  \\
			& =I_{\text{int}}(\widehat{\phi}^{-1}(X),\widehat{\phi}^{-1}(Y)).
		\end{align*}
\item[\emph{Case 3:}] $\mathrm{rk} X=2$ and $\mathrm{rk} Y=1$. The proof is similar to that of Case 2.
		  
\item[\emph{Case 4:}] $\mathrm{rk} X=\mathrm{rk} Y=2$. Then $X=E_{L_0(-i\vec{x}_3)}\langle (k-1)\vec{x}_3\rangle$ and $Y=E_{L_0(-s\vec{x}_3)}\langle (t-1)\vec{x}_3\rangle$ for some $i, k, s, t\in \mathbb{Z}$ and $1\leq k, t\leq n-1$. From Proposition \ref{intersection and exact seq} Case 1, we see that $I_{\rm{int}}(\widehat{\phi}^{-1}(X),\widehat{\phi}^{-1}(Y))=0$ if ${\rm Ext^{1}}(X,Y)=0$. Let us consider the case where ${\rm Ext^{1}}(X,Y)\neq 0$. By Proposition \ref{Hom and Ext}(2) and the result in Case 1, through meticulous calculations, we can obtain
		\begin{align*}
			{\rm dim_{k}Ext^{1}}(X,Y) 
			& ={\rm dim_{k}Ext^{1}}(L_0^*(-(i+1)\vec{x}_3),Y)+{\rm dim_{k}Ext^{1}}(L_0((k-i-1)\vec{x}_3),Y)      \\
			& =I_{\text{int}}(\widetilde{[i,-i]^+},\widetilde{[s,t-s]})+I_{\text{int}}(\widetilde{[i-k,k-i]^-},\widetilde{[s,t-s]})      \\
			& =I_{\text{int}}(\widetilde{[i,k-i]},\widetilde{[s,t-s]}) \\
			& =I_{\text{int}}(\widehat{\phi}^{-1}(X),\widehat{\phi}^{-1}(Y)).
		\end{align*}
  \end{itemize}
   The proof is complete.
	\end{proof}

\section{Projective cover and injective hull}\label{Projective cover and injective hull}
In this section, we give the geometric interpretation for the projective covers and injective hulls of extension bundles.
	\begin{defn}(\cite[Definition 3.1]{MR3028577})
				A sequence $0 \rightarrow X^{\prime} \stackrel{u}{\rightarrow} X \stackrel{v}{\rightarrow} X^{\prime \prime} \rightarrow 0$ in ${\rm vect}\mbox{-}\mathbb{X} $ is called {\emph distinguished exact} if for each line bundle $L$ the induced sequence $0 \rightarrow \operatorname{Hom}\left(L, X^{\prime}\right) \rightarrow \operatorname{Hom}(L, X) \rightarrow\operatorname{Hom}\left(L, X^{\prime \prime}\right) \rightarrow 0$ is exact.
		\end{defn}
  
	The distinguished exact sequences define an exact structure on ${\rm vect}\mbox{-}\mathbb{X}$ which
				is Frobenius, such that the indecomposable projectives (resp. injectives) are exactly the line
				bundles. Moreover, ${\rm vect}\text{-}\mathbb{X}$ is equivalent to $\mathrm{CM}^{\mathbb{L}}$-$R$ as Frobenius category. \cite[Proposition 3.2]{MR3028577} With this exact structure, \({\rm vect}\mbox{-}\mathbb{X}\) is equipped to have projective covers and, dually, injective hulls \cite[Proposition 3.5]{MR3028577}. 
	\begin{defn}\label{def line}
	     
	Let $[i,j]$ be a segment in $\operatorname{Seg}(M) \setminus \operatorname{Seg}_0(M)$ and $k$ be the unique integer such that  $$kn<i+j< (k+1)n.$$ Then the segments $[i,\leftarrow], [i,\rightarrow] $ and $[\leftarrow,j], [\rightarrow,j],$ are intuitively defined as follows: 
		
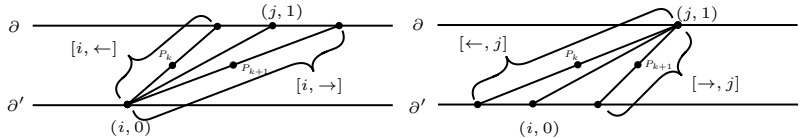
\begin{figure}[H]

			\tikzset{every picture/.style={line width=0.75pt}}          
			
			\begin{tikzpicture}[x=0.75pt,y=0.75pt,yscale=-1,xscale=1]

				\draw    (275.33,109) -- (455.08,109) ;
				
				\draw    (275.58,149) -- (455.33,149) ;
				
				\draw [color={rgb, 255:red, 0; green, 0; blue, 0 }  ,draw opacity=1 ]   (322.88,148.64) -- (395.33,109) ;
				\draw [shift={(395.33,109)}, rotate = 331.32] [color={rgb, 255:red, 0; green, 0; blue, 0 }  ,draw opacity=1 ][fill={rgb, 255:red, 0; green, 0; blue, 0 }  ,fill opacity=1 ][line width=0.75]      (0, 0) circle [x radius= 1.34, y radius= 1.34]   ;
				\draw [shift={(322.88,148.64)}, rotate = 331.32] [color={rgb, 255:red, 0; green, 0; blue, 0 }  ,draw opacity=1 ][fill={rgb, 255:red, 0; green, 0; blue, 0 }  ,fill opacity=1 ][line width=0.75]      (0, 0) circle [x radius= 1.34, y radius= 1.34]   ;
				
				\draw    (395.33,109) -- (295.33,149) ;
				\draw [shift={(295.33,149)}, rotate = 158.2] [color={rgb, 255:red, 0; green, 0; blue, 0 }  ][fill={rgb, 255:red, 0; green, 0; blue, 0 }  ][line width=0.75]      (0, 0) circle [x radius= 1.34, y radius= 1.34]   ;
				\draw [shift={(345.33,129)}, rotate = 158.2] [color={rgb, 255:red, 0; green, 0; blue, 0 }  ][fill={rgb, 255:red, 0; green, 0; blue, 0 }  ][line width=0.75]      (0, 0) circle [x radius= 1.34, y radius= 1.34]   ;
				\draw [shift={(395.33,109)}, rotate = 158.2] [color={rgb, 255:red, 0; green, 0; blue, 0 }  ][fill={rgb, 255:red, 0; green, 0; blue, 0 }  ][line width=0.75]      (0, 0) circle [x radius= 1.34, y radius= 1.34]   ;
				
				\draw    (226.21,109.31) -- (120.54,148.97) ;
				\draw [shift={(120.54,148.97)}, rotate = 159.42] [color={rgb, 255:red, 0; green, 0; blue, 0 }  ][fill={rgb, 255:red, 0; green, 0; blue, 0 }  ][line width=0.75]      (0, 0) circle [x radius= 1.34, y radius= 1.34]   ;
				\draw [shift={(173.38,129.14)}, rotate = 159.42] [color={rgb, 255:red, 0; green, 0; blue, 0 }  ][fill={rgb, 255:red, 0; green, 0; blue, 0 }  ][line width=0.75]      (0, 0) circle [x radius= 1.34, y radius= 1.34]   ;
				\draw [shift={(226.21,109.31)}, rotate = 159.42] [color={rgb, 255:red, 0; green, 0; blue, 0 }  ][fill={rgb, 255:red, 0; green, 0; blue, 0 }  ][line width=0.75]      (0, 0) circle [x radius= 1.34, y radius= 1.34]   ;
				
				\draw    (73.25,109.33) -- (253,109.33) ;
				
				\draw    (73.25,149.33) -- (253,149.33) ;
				
				\draw [color={rgb, 255:red, 0; green, 0; blue, 0 }  ,draw opacity=1 ]   (355.33,149) -- (395.33,109) ;
				\draw [shift={(395.33,109)}, rotate = 315] [color={rgb, 255:red, 0; green, 0; blue, 0 }  ,draw opacity=1 ][fill={rgb, 255:red, 0; green, 0; blue, 0 }  ,fill opacity=1 ][line width=0.75]      (0, 0) circle [x radius= 1.34, y radius= 1.34]   ;
				\draw [shift={(375.33,129)}, rotate = 315] [color={rgb, 255:red, 0; green, 0; blue, 0 }  ,draw opacity=1 ][fill={rgb, 255:red, 0; green, 0; blue, 0 }  ,fill opacity=1 ][line width=0.75]      (0, 0) circle [x radius= 1.34, y radius= 1.34]   ;
				\draw [shift={(355.33,149)}, rotate = 315] [color={rgb, 255:red, 0; green, 0; blue, 0 }  ,draw opacity=1 ][fill={rgb, 255:red, 0; green, 0; blue, 0 }  ,fill opacity=1 ][line width=0.75]      (0, 0) circle [x radius= 1.34, y radius= 1.34]   ;
				
				\draw    (120.54,148.97) -- (165.46,109.69) ;
				\draw [shift={(165.46,109.69)}, rotate = 318.83] [color={rgb, 255:red, 0; green, 0; blue, 0 }  ][fill={rgb, 255:red, 0; green, 0; blue, 0 }  ][line width=0.75]      (0, 0) circle [x radius= 1.34, y radius= 1.34]   ;
				\draw [shift={(143,129.33)}, rotate = 318.83] [color={rgb, 255:red, 0; green, 0; blue, 0 }  ][fill={rgb, 255:red, 0; green, 0; blue, 0 }  ][line width=0.75]      (0, 0) circle [x radius= 1.34, y radius= 1.34]   ;
				\draw [shift={(120.54,148.97)}, rotate = 318.83] [color={rgb, 255:red, 0; green, 0; blue, 0 }  ][fill={rgb, 255:red, 0; green, 0; blue, 0 }  ][line width=0.75]      (0, 0) circle [x radius= 1.34, y radius= 1.34]   ;
				
				\draw [color={rgb, 255:red, 0; green, 0; blue, 0 }  ,draw opacity=1 ]   (120.54,148.97) -- (193,109.33) ;
				\draw [shift={(193,109.33)}, rotate = 331.32] [color={rgb, 255:red, 0; green, 0; blue, 0 }  ,draw opacity=1 ][fill={rgb, 255:red, 0; green, 0; blue, 0 }  ,fill opacity=1 ][line width=0.75]      (0, 0) circle [x radius= 1.34, y radius= 1.34]   ;
				\draw [shift={(120.54,148.97)}, rotate = 331.32] [color={rgb, 255:red, 0; green, 0; blue, 0 }  ,draw opacity=1 ][fill={rgb, 255:red, 0; green, 0; blue, 0 }  ,fill opacity=1 ][line width=0.75]      (0, 0) circle [x radius= 1.34, y radius= 1.34]   ;
				
				\draw   (392.67,105.5) .. controls (390.88,101.19) and (387.83,99.92) .. (383.52,101.71) -- (321.53,127.39) .. controls (315.37,129.94) and (311.4,129.06) .. (309.61,124.75) .. controls (311.4,129.06) and (309.21,132.49) .. (303.05,135.04)(305.83,133.89) -- (297.46,137.36) .. controls (293.15,139.15) and (291.88,142.2) .. (293.67,146.51) ;
				
				\draw   (123.33,152) .. controls (125.01,156.35) and (128.03,157.69) .. (132.38,156.02) -- (201.61,129.4) .. controls (207.83,127.01) and (211.78,128) .. (213.46,132.35) .. controls (211.78,128) and (214.05,124.62) .. (220.28,122.23)(217.48,123.3) -- (223.81,120.87) .. controls (228.16,119.2) and (229.5,116.18) .. (227.83,111.82) ;
				
				\draw   (162.5,107.33) .. controls (159.49,103.77) and (156.2,103.5) .. (152.63,106.51) -- (133.09,123.05) .. controls (128,127.36) and (123.95,127.73) .. (120.94,124.17) .. controls (123.95,127.73) and (122.91,131.66) .. (117.82,135.97)(120.11,134.03) -- (117.82,135.97) .. controls (114.26,138.98) and (113.99,142.27) .. (117,145.84) ;
				
				\draw   (360.17,152.17) .. controls (363.49,155.45) and (366.79,155.43) .. (370.07,152.1) -- (385.34,136.63) .. controls (390.03,131.89) and (394.03,131.16) .. (397.35,134.44) .. controls (394.03,131.16) and (394.71,127.15) .. (399.4,122.4)(397.29,124.54) -- (399.4,122.4) .. controls (402.68,119.08) and (402.66,115.78) .. (399.34,112.5) ;

				\draw (259.33,144.6) node [anchor=north west][inner sep=0.75pt]  [font=\tiny]  {$\partial ^{\prime }$};
				
				\draw (261.13,104.27) node [anchor=north west][inner sep=0.75pt]  [font=\tiny]  {$\partial $};
				
				\draw (314.17,155.73) node [anchor=north west][inner sep=0.75pt]  [font=\tiny]  {$( i,0)$};
				
				\draw (392.8,98.2) node [anchor=north west][inner sep=0.75pt]  [font=\tiny]  {$( j,1)$};
				
				\draw (58.87,143.73) node [anchor=north west][inner sep=0.75pt]  [font=\tiny]  {$\partial ^{\prime }$};
				
				\draw (59.47,105.4) node [anchor=north west][inner sep=0.75pt]  [font=\tiny]  {$\partial $};
				
				\draw (110.33,155.07) node [anchor=north west][inner sep=0.75pt]  [font=\tiny]  {$( i,0)$};
				
				\draw (185.53,96.8) node [anchor=north west][inner sep=0.75pt]  [font=\tiny]  {$( j,1)$};
				
				\draw (337.37,120.47) node [anchor=north west][inner sep=0.75pt]  [font=\tiny]  {\scalebox{0.6}[0.6]{$P_{k}$}};
				
				\draw (90.53,115.33) node [anchor=north west][inner sep=0.75pt]  [font=\tiny]  {$\left[ i,\leftarrow \right]$};
				
				\draw (202.93,134.17) node [anchor=north west][inner sep=0.75pt]  [font=\tiny]  {$\left[ i,\rightarrow \right]$};
				
				\draw (400.4,133.13) node [anchor=north west][inner sep=0.75pt]  [font=\tiny]  {$\left[\rightarrow ,j\right]$};
				
				\draw (284.37,113.07) node [anchor=north west][inner sep=0.75pt]  [font=\tiny]  {$\left[\leftarrow ,j\right]$};
				
				\draw (175.88,127.6) node [anchor=north west][inner sep=0.75pt]  [font=\tiny]  {\scalebox{0.6}[0.6]{$P_{k+1}$}};
				
				\draw (134.43,120.47) node [anchor=north west][inner sep=0.75pt]  [font=\tiny]  {\scalebox{0.6}[0.6]{$P_{k}$}};
				
				\draw (377.33,125.4) node [anchor=north west][inner sep=0.75pt]  [font=\tiny]  {\scalebox{0.6}[0.6]{$P_{k+1}$}};

			\end{tikzpicture}
		\end{figure}
		\end{defn} 
	
Let $X$ be an extension bundle in ${\rm vect}\mbox{-}\mathbb{X}$. We denote by $\mathfrak{P}(X) \xrightarrow{\pi_X} X$ the projective cover, and dually by $X \xrightarrow{j_X} \mathfrak{I}(X)$ the injective hull of $X$. The following general result has been stated in \cite{MR3028577}. We provide a brief proof by employing the correspondence $\widehat{\phi}$ in Propostion \ref{correspondence}. It's worth emphasizing that we realize the universal properties of projective cover and injective envelope as the proximity between line segments.

\begin{prop}\label{cover and hull}
Assume $X=\widehat{\phi}(\widetilde{[i,j]^*})$ is an extension bundle in ${\rm vect}\mbox{-}\mathbb{X}$. Then 
\begin{itemize}
    \item $\mathfrak{P}(X)=\widehat{\phi}(\widetilde{[i,\leftarrow]})\oplus \widehat{\phi}(\widetilde{[\rightarrow,j]})$;
    \item $\mathfrak{I}(X)=\widehat{\phi}(\widetilde{[i,\rightarrow]})\oplus \widehat{\phi}(\widetilde{[\leftarrow,j]})$, 
\end{itemize}
\end{prop}
\begin{proof}
    We only prove (1); the proof of (2) is similar. Let $k$ be the unique integer such that  $kn< i+j< (k+1)n$. Observe that $[i,j]$ is a diagonal of the quadrilateral determined by $[i,\leftarrow]$ and $[\rightarrow,j]$, while $[(k+1)n-j,kn-i]$ is another diagonal quadrilateral in this quadrilateral.  By Proposition \ref{exact sequence1}, there is an exact sequence 
\[\begin{tikzcd}[ampersand replacement=\&,cramped,sep=small]
	0 \& Y \& {\mathfrak{P}(X)} \& X \& 0,
	\arrow[from=1-1, to=1-2]
	\arrow[from=1-2, to=1-3]
	\arrow["{{\pi_X}}", from=1-3, to=1-4]
	\arrow[from=1-4, to=1-5]
\end{tikzcd}\]
where $Y=\phi(\widetilde{[(k+1)n-j,kn-i]})$. Let $\mathscr{L}$ be the system of all line bundles. We want to show that for any $L\in \mathscr{L}$, each nonzero morphism $f: L \rightarrow X$ factors through $\pi_X:\mathfrak{P}(X)\rightarrow X$. It is sufficient to show $\Ext^1(L,Y)=0$.

We may assume that $L=\widehat{\phi}(\widetilde{[s,tn-s]^*})$ for some integers $s,t$, and whether $*$ is $+$ or $-$ is determined by the parity of $t$. Since $f\in \Hom(L,X)\neq 0$, by Serre duality, $\Ext^1(\tau^{-1}X,L)\neq 0$. Now, assume $\Ext^1(L,Y)\neq0$. Then by Theorem \ref{dimension and positive intersection}, we have 
\[I_{\rm{int}}(\widehat{\phi}^{-1}(\tau^{-1}X),\widehat{\phi}^{-1}(L))\neq 0 \text{ and  }I_{\rm{int}}(\widehat{\phi}^{-1}(L),\widehat{\phi}^{-1}(Y))\neq 0.\]
It follows that there exist integers $m,m^\prime$ such that 
		\begin{align*}
			\left\{
			\begin{array}{ll}
				s+mn>i-1 \\
				tn-s+mn<j+1
			\end{array}
			\right. \qquad
			\text{and} \qquad
			\left\{
			\begin{array}{ll}
				s+m^\prime n<(k+1)n-j \\
				tn-s+m^\prime n>kn-i
			\end{array}
			\right..
		\end{align*}
We can get 
   $$(k-m-m^\prime)-\frac{1}{n}<(k-m^\prime)+\frac{s-i}{n}<t<-m+\frac{s+j+1}{n}<(k-m-m^\prime+1)+\frac{1}{n}.$$ Such an integer $t$ does not exist, a contradiction. Therefore, $\Ext^1(L,Y)=0$. Minimality of $\pi_X$ follows from the fact that the line bundle constituents of $\mathfrak{P}(X)$ are mutually Hom-orthogonal.
\end{proof}
\begin{rem}
   From a graphical perspective, we can observe that $\mathfrak{I}(Y)$=$\mathfrak{P}(X)$. 
\end{rem}

\appendix
\section{}\label{Appendix} 
In this appendix, we collect some exact sequences containing generalized extension bundles in ${\rm coh}\mbox{-}\mathbb{X}$.  For brevity, when the morphisms in the commutative diagrams are clear from the context, we omit writing them. 
\begin{prop}{}{}\label{exact sequence11}
Assume $\vec{x}=l_3\vec{x}_3+l\vec{c} \in  \mathbb{L}$ is written in normal form. Then there exists an exact sequence in ${\rm vect}\mbox{-}\mathbb{X}$
\[\begin{tikzcd}[ampersand replacement=\&,cramped,sep=small]
	0 \& {\mathsf{E}_L\langle \vec{x}\rangle} \& {\mathsf{E}_L \langle \vec{x}+\vec{x}_3\rangle} \& {S_3(\vec{x}+\vec{x}_3)} \& 0
	\arrow[from=1-2, to=1-3]
	\arrow[from=1-3, to=1-4]
	\arrow[from=1-4, to=1-5]
	\arrow[from=1-1, to=1-2].
\end{tikzcd}\]
\end{prop}
\begin{proof}
 For $0 \leq l_3 \leq n-2$, $\mathsf{E}_L\langle \vec{x}\rangle=E_{L(l\vec{x}_1)}\langle l_3\vec{x}_3\rangle$. By applying $\Ext^1(-, L(\vec{\omega}+l\vec{x}_1))$ to $x_3$, we obtain a pushout commutative diagram
\[\begin{tikzcd}[ampersand replacement=\&,cramped]
	\&\&\& 0 \& 0 \\
	{\xi_1:} \& 0 \& {L(\vec{\omega}+l\vec{x}_1)} \& {E_{L(l\vec{x}_1)} \langle l_3\vec{x}_3\rangle} \& {L(l_3\vec{x}_3+l\vec{x}_1)} \& 0 \\
	{\xi_2:} \& 0 \& {L(\vec{\omega}+l\vec{x}_1)} \& Z \& {L((l_3+1)\vec{x}_3+l\vec{x}_1)} \& 0 \\
	\&\&\& {S_3(\vec{x}+\vec{x}_3)} \& {S_3(\vec{x}+\vec{x}_3)} \\
	\&\&\& 0 \& 0
	\arrow[from=2-2, to=2-3]
	\arrow[from=2-3, to=2-4]
	\arrow[from=2-4, to=2-5]
	\arrow[from=2-5, to=2-6]
	\arrow[from=3-2, to=3-3]
	\arrow[from=3-3, to=3-4]
	\arrow[from=3-4, to=3-5]
	\arrow[from=3-5, to=3-6]
	\arrow[Rightarrow, no head, from=2-3, to=3-3]
	\arrow[from=2-4, to=3-4]
	\arrow[from=3-4, to=4-4]
	\arrow[from=4-4, to=5-4]
	\arrow[from=1-4, to=2-4]
	\arrow["{{{{{{x_3}}}}}}", from=2-5, to=3-5]
	\arrow[from=3-5, to=4-5]
	\arrow[from=4-5, to=5-5]
	\arrow[Rightarrow, no head, from=4-4, to=4-5]
	\arrow[from=1-5, to=2-5]
\end{tikzcd}\]
with exact rows and columns. Moreover, by the heredity, $\xi_2$ is a non-split exact sequence. Since $\Ext^1(L((l_3+1)\vec{x}_3+l\vec{x}_1), L(\vec{\omega}+l\vec{x}_1))\cong \mathbf{k}$, by \eqref{extension bundle}, the middle term of the exact sequence $\xi_2$ is $Z=\mathsf{E}_{L(l\vec{x}_1)}\langle (l_3+1)\vec{x}_3\rangle=\mathsf{E}_L \langle \vec{x}+\vec{x}_3\rangle$. 

For $l_3 = n-1$,  by pullback of the almost-split sequence $\xi_3$ along $x_3$, we have the following commutative diagram
\begin{equation}\label{PPC2}
\begin{tikzcd}[ampersand replacement=\&,cramped]
	\&\&\& 0 \& 0 \\
	\& 0 \& {L(\vec{\omega})} \& {L(\vec{\omega})\oplus L^*(\vec{\omega})} \& {L^*(\vec{\omega})} \& 0 \\
	{\xi_3:} \& 0 \& {L(\vec{\omega})} \& {E_L} \& L \& 0 \\
	\&\&\& {S_3} \& {S_3} \\
	\&\&\& 0 \& 0
	\arrow["{{{{\eta_1}}}}", from=3-3, to=3-4]
	\arrow["{{{{{\pi_1}}}}}", from=3-4, to=3-5]
	\arrow[from=3-5, to=3-6]
	\arrow[from=3-2, to=3-3]
	\arrow[Rightarrow, no head, from=2-3, to=3-3]
	\arrow["{{{x_3}}}", from=2-5, to=3-5]
	\arrow[from=2-2, to=2-3]
	\arrow[from=2-5, to=2-6]
	\arrow["{{{\left( \begin{smallmatrix} 1 \\ 0 \end{smallmatrix} \right )}}}", from=2-3, to=2-4]
	\arrow["{{{\left( \begin{smallmatrix} 0 & 1 \end{smallmatrix} \right )}}}", from=2-4, to=2-5]
	\arrow[from=3-4, to=4-4]
	\arrow[from=4-4, to=5-4]
	\arrow["{{{\left( \begin{smallmatrix} \eta_1 & \eta_2 \end{smallmatrix} \right )}}}", from=2-4, to=3-4]
	\arrow[from=3-5, to=4-5]
	\arrow[Rightarrow, no head, from=4-4, to=4-5]
	\arrow[from=4-5, to=5-5]
	\arrow[from=1-4, to=2-4]
	\arrow[from=1-5, to=2-5]
\end{tikzcd}
\end{equation}
with exact rows and columns, where $\eta_2$ is the left minimal almost split starting with $L^*$ such that $\pi_1\eta_2=x_3$. Due to $\Ext^1(L^*(\vec{\omega}),L(\vec{\omega}))=0$, the first exact row is split. Observe that $\mathsf{E}_L\langle\vec{x}\rangle=E_{L(l\vec{x}_1)}\langle(n-1)\vec{x}_3\rangle=(L(\vec{\omega})\oplus L^*(\vec{\omega}))((l+1)\vec{x}_1)$ and $\mathsf{E}_L\langle\vec{x}+\vec{x}_3\rangle=E_L((l+1)\vec{x}_1)$ for all $l \in \mathbb{Z}$. 
Hnece, the required exact sequence can be obtained by applying the degree shift $(l+1)\vec{x}_1$ to the middle column exact row in \eqref{PPC2}.
\end{proof}

\begin{cor}{}{}\label{exact sequence2}
Assume $\vec{x}=l_3\vec{x}_3+l\vec{c} \in  \mathbb{L}$ is written in normal form. Then there exists an exact sequence in ${\rm vect}\mbox{-}\mathbb{X} $
\[\begin{tikzcd}[ampersand replacement=\&,cramped,sep=small]
	0 \& {\mathsf{E}_L\langle \vec{x}\rangle} \& {\mathsf{E}_{L(\vec{x}_3)}\langle \vec{x}-\vec{x}_3\rangle} \& {S_3} \& 0
	\arrow[from=1-2, to=1-3]
	\arrow[from=1-3, to=1-4]
	\arrow[from=1-4, to=1-5]
	\arrow[from=1-1, to=1-2].
\end{tikzcd}\]
\end{cor}
\begin{proof}

Replacing $L$ and $\vec{x}$ in Proposition \ref{exact sequence11} with $L(\vec{x}_1)$ and $\vec{\delta}-\vec{x}$ respectively, we obtain the exact sequence
\[\begin{tikzcd}[ampersand replacement=\&,cramped,sep=small]
	0 \& {E_{L(\vec{x}_1)}\langle \vec{\delta}-\vec{x}\rangle} \& {E_{L(\vec{x}_1)} \langle \vec{\delta}-\vec{x}+\vec{x}_3\rangle} \& {S_3(\vec{\delta}-\vec{x}+\vec{x}_3)} \& 0
	\arrow[from=1-2, to=1-3]
	\arrow[from=1-3, to=1-4]
	\arrow[from=1-4, to=1-5]
	\arrow[from=1-1, to=1-2]
\end{tikzcd}.\]
According to Remark \ref{rem1}, it is the  exact sequence
\[\begin{tikzcd}[ampersand replacement=\&,cramped,sep=small]
	0 \& {E_{L(\vec{\delta}-\vec{x}+\vec{x}_3)}\langle \vec{x}\rangle} \& {E_{L(\vec{\delta}-\vec{x}+2\vec{x}_3)} \langle \vec{x}-\vec{x}_3\rangle} \& {S_3(\vec{\delta}-\vec{x}+\vec{x}_3)} \& 0
	\arrow[from=1-2, to=1-3]
	\arrow[from=1-3, to=1-4]
	\arrow[from=1-4, to=1-5]
	\arrow[from=1-1, to=1-2].
\end{tikzcd}\]
 Then applying the degree shift by $-\vec{\delta}+\vec{x}-\vec{x}_3$ to this exact sequence, we get the required exact sequence. 
\end{proof}

The following proposition is crucial in the sequel.
\begin{prop}{}{}\label{PPS1}
    Assume $\vec{x}=l_3\vec{x}_3+l\vec{c} \in  \mathbb{L}$ is  written in normal form. Then there exists a bicartesian square in ${\rm vect}\mbox{-}\mathbb{X} $
\[\begin{tikzcd}[ampersand replacement=\&,cramped,sep=small]
	{\mathsf{E}_L\langle \vec{x}\rangle} \& {\mathsf{E}_{L(\vec{x}_3)}\langle \vec{x}-\vec{x}_3\rangle} \\
	{\mathsf{E}_L\langle \vec{x}+\vec{x}_3\rangle} \& {\mathsf{E}_{L(\vec{x}_3)}\langle \vec{x}\rangle.}
	\arrow[from=1-1, to=1-2]
	\arrow[from=1-2, to=2-2]
	\arrow[from=1-1, to=2-1]
	\arrow[from=2-1, to=2-2]
\end{tikzcd}\]Consequently, this bicartesian square admits a exact sequence  
\[\begin{tikzcd}[ampersand replacement=\&,cramped,sep=small]
	{0 } \& {\mathsf{E}_L\langle \vec{x}\rangle} \& {\mathsf{E}_L\langle \vec{x}+\vec{x}_3\rangle\oplus \mathsf{E}_{L(\vec{x}_3)}\langle \vec{x}-\vec{x}_3\rangle} \& {\mathsf{E}_{L(\vec{x}_3)}\langle \vec{x}\rangle} \& 0
	\arrow[from=1-1, to=1-2]
	\arrow[from=1-2, to=1-3]
	\arrow[from=1-3, to=1-4]
	\arrow[from=1-4, to=1-5]
\end{tikzcd}.\]
\end{prop}

\begin{proof}
  If $n=2$ and $l_3=0$, the there exists a subquiver of the Auslander-Reiten quiver $\Gamma({\rm vect}\mbox{-}\mathbb{X})$
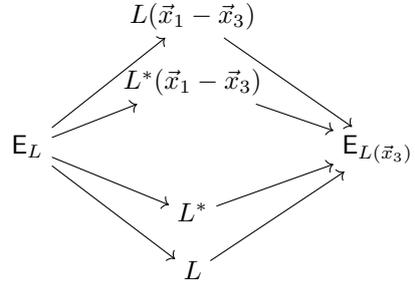
  \[\begin{tikzcd}[ampersand replacement=\&,cramped,sep=small]
	\&\& {L(\vec{x}_1-\vec{x}_3)} \\
	\&\& {L^*(\vec{x}_1-\vec{x}_3)} \\
	{\mathsf{E}_L} \&\&\&\& {\mathsf{E}_{L(\vec{x}_3)}.} \\
	\&\& {L^*} \\
	\&\& L
	\arrow[from=1-3, to=3-5]
	\arrow[from=2-3, to=3-5]
	\arrow[from=3-1, to=1-3]
	\arrow[from=3-1, to=2-3]
	\arrow[from=3-1, to=4-3]
	\arrow[from=3-1, to=5-3]
	\arrow[from=4-3, to=3-5]
	\arrow[from=5-3, to=3-5]
\end{tikzcd}\]
 Since $\mathsf{E}_L\langle\vec{x}_3\rangle=L(\vec{x}_1-\vec{x}_3)\oplus L^*(\vec{x}_1-\vec{x}_3)$ and $\mathsf{E}_{L(\vec{x}_3)}\langle -\vec{x}_3\rangle=L\oplus L^*$, this subquiver admits the almost-split sequence 
\[\begin{tikzcd}[ampersand replacement=\&,cramped,sep=small]
	{0 } \& {\mathsf{E}_L} \& {\mathsf{E}_L\langle \vec{x}_3\rangle\oplus \mathsf{E}_{L(\vec{x}_3)}\langle -\vec{x}_3\rangle} \& {\mathsf{E}_{L(\vec{x}_3)}} \& 0
	\arrow[from=1-1, to=1-2]
	\arrow[from=1-2, to=1-3]
	\arrow[from=1-3, to=1-4]
	\arrow[from=1-4, to=1-5]
\end{tikzcd}.\]
Assume $n=2$ and $l_3=1$ or $n\geq 3$. By Proposition \ref{exact sequence11} and Corollary \ref{exact sequence2}, there exist two non-split exact sequences $\xi_1$ and $\xi_2$.  Then we can obtain the pushout commutative diagram
\[\begin{tikzcd}[ampersand replacement=\&,cramped]
	\&\& {} \\
	\&\& 0 \& 0 \\
	{\xi_1:} \& 0 \& {\mathsf{E}_L\langle \vec{x}\rangle} \& {\mathsf{E}_{L(\vec{x}_3)}\langle \vec{x}-\vec{x}_3\rangle} \& {S_3} \& 0 \\
	{\xi_3:} \& 0 \& {\mathsf{E}_L\langle \vec{x}+\vec{x}_3\rangle} \& Z \& {S_3} \& 0 \\
	\&\& {S_3(\vec{x}+\vec{x}_3)} \& {S_3(\vec{x}+\vec{x}_3)} \\
	\&\& 0 \& 0
	\arrow["{{{\eta^\prime}}}", from=3-3, to=3-4]
	\arrow[from=3-4, to=4-4]
	\arrow["\eta"', from=3-3, to=4-3]
	\arrow[from=4-3, to=4-4]
	\arrow[from=4-3, to=5-3]
	\arrow[from=4-4, to=5-4]
	\arrow[from=3-4, to=3-5]
	\arrow[from=2-3, to=3-3]
	\arrow[from=2-4, to=3-4]
	\arrow[from=5-3, to=6-3]
	\arrow[from=4-4, to=4-5]
	\arrow[from=5-4, to=6-4]
	\arrow[from=4-2, to=4-3]
	\arrow[from=3-2, to=3-3]
	\arrow[from=3-5, to=3-6]
	\arrow[from=4-5, to=4-6]
	\arrow[Rightarrow, no head, from=3-5, to=4-5]
	\arrow[Rightarrow, no head, from=5-3, to=5-4]
	\arrow["{{\xi_2:}}"{description, pos=0.3, scale=1.4}, draw=none, from=1-3, to=2-3]
\end{tikzcd}\]
 with exact rows and columns.

For $0\leq l_3\leq n-2$, we apply $\Hom(S_3,-)$ to $\xi_2$. Since $\Hom(S_3,S_3(\vec{x}+\vec{x}_3))$ vanishes, $\Ext^1(S_3,\eta)$ is a monomorphism. It follows that $\xi_3$ is non-split. Further, because $\Ext^1(S_3,\mathsf{E}_L\langle \vec{x}+\vec{x}_3\rangle)\cong\Ext^1(S_3,L(\vec{\omega}))\cong \mathbf{k}$, by Corollary \ref{exact sequence2}, we obtain $Z=\mathsf{E}_{L(\vec{x}_3)}\langle \vec{x}\rangle$. 

For $l_3=n-1$, let $\eta_1$ and $\eta_2$ be the morphisms in \eqref{PPC2}. By pullback of the exact sequence $\xi_4$ along $k_1$, we have the following commutative diagram
\begin{equation}\label{PPC1} 
\begin{tikzcd}[ampersand replacement=\&,cramped]
	\&\&\& 0 \& 0 \\
	\&\&\& {L^*(\vec{\omega})} \& {L^*(\vec{\omega})} \\
	\& 0 \& {L(\vec{\omega})} \& {E_L} \& L \& 0 \\
	{\xi_4:} \& 0 \& {L(\vec{\omega})} \& {L^*} \& {S_3} \& 0 \\
	\&\&\& 0 \& 0
	\arrow[from=1-4, to=2-4]
	\arrow[from=1-5, to=2-5]
	\arrow[Rightarrow, no head, from=2-4, to=2-5]
	\arrow["{{{{{{\eta_2}}}}}}", from=2-4, to=3-4]
	\arrow["{{{{{x_3}}}}}", from=2-5, to=3-5]
	\arrow[from=3-2, to=3-3]
	\arrow["{{{{{{\eta_1}}}}}}", from=3-3, to=3-4]
	\arrow[Rightarrow, no head, from=3-3, to=4-3]
	\arrow["{{{{{{{\pi_1}}}}}}}", from=3-4, to=3-5]
	\arrow["{{{{{{{\pi_2}}}}}}}", from=3-4, to=4-4]
	\arrow[from=3-5, to=3-6]
	\arrow["{{k_1}}", from=3-5, to=4-5]
	\arrow[from=4-2, to=4-3]
	\arrow["{{{{{x_3}}}}}", from=4-3, to=4-4]
	\arrow["{{k_2}}", from=4-4, to=4-5]
	\arrow[from=4-4, to=5-4]
	\arrow[from=4-5, to=4-6]
	\arrow[from=4-5, to=5-5]
\end{tikzcd}
\end{equation}
with exact rows and columns. Particularly, both the middle row and the middle column are almost-split sequences. Note that, by \eqref{hom space}, we have
\[\Hom(L(\vec{\omega}), L^*(\vec{\omega}))\cong\Hom(L^*(\vec{\omega}), L(\vec{\omega}))=0.\]By Proposition \ref{exact sequence11} and Corollary \ref{exact sequence2}, we can obtain the pushout commutative diagram  
\[\begin{tikzcd}[ampersand replacement=\&,cramped]
	\& {} \\
	\& 0 \& 0 \\
	0 \& {L(\vec{\omega})\oplus L^*(\vec{\omega})} \& {\mathsf{E}_{L}} \& {S_3} \& 0 \\
	0 \& {\mathsf{E}_L} \& {L\oplus L^*} \& {S_3} \& 0 \\
	\& {S_3} \& {S_3} \\
	\& 0 \& 0 \\
	\& {}
	\arrow[from=2-2, to=3-2]
	\arrow[from=2-3, to=3-3]
	\arrow[from=3-1, to=3-2]
	\arrow["{\left( \begin{smallmatrix} \eta_1 & \eta_2 \end{smallmatrix} \right )}", from=3-2, to=3-3]
	\arrow["{\left( \begin{smallmatrix} -\eta_1 & \eta_2 \end{smallmatrix} \right )}"', from=3-2, to=4-2]
	\arrow["{2k_1\pi_1}", from=3-3, to=3-4]
	\arrow["{\left( \begin{smallmatrix}\pi_1 \\ \pi_2 \end{smallmatrix} \right )}", from=3-3, to=4-3]
	\arrow[from=3-4, to=3-5]
	\arrow[Rightarrow, no head, from=3-4, to=4-4]
	\arrow[from=4-1, to=4-2]
	\arrow["{\left( \begin{smallmatrix} \pi_1 \\ -\pi_2 \end{smallmatrix} \right )}", from=4-2, to=4-3]
	\arrow["{2k_1\pi_1}"', from=4-2, to=5-2]
	\arrow["{\left( \begin{smallmatrix} k_1 & k_2 \end{smallmatrix} \right )}", from=4-3, to=4-4]
	\arrow["{\left( \begin{smallmatrix} k_1 & -k_2 \end{smallmatrix} \right )}", from=4-3, to=5-3]
	\arrow[from=4-4, to=4-5]
	\arrow[Rightarrow, no head, from=5-2, to=5-3]
	\arrow[from=5-2, to=6-2]
	\arrow[from=5-3, to=6-3]
\end{tikzcd}\]
with exact rows and columns. This bicartesian square admits the exact sequence
\begin{equation}\label{important exact}
   \begin{tikzcd}[ampersand replacement=\&,cramped]
	0 \& {L(\vec{\omega})\oplus L^*(\vec{\omega})} \& {\mathsf{E}_L\oplus\mathsf{E}_L} \& {L\oplus L^*} \& 0.
	\arrow[from=1-1, to=1-2]
	\arrow["{{\left( \begin{smallmatrix} \eta_1 & \eta_2\\-\eta_1 & \eta_2 \end{smallmatrix} \right )}}", from=1-2, to=1-3]
	\arrow["{{\left( \begin{smallmatrix} \pi_1 & -\pi_1\\ \pi_2 &\pi_2 \end{smallmatrix} \right )}}", from=1-3, to=1-4]
	\arrow[from=1-4, to=1-5]
\end{tikzcd}
\end{equation}
Observe that $\mathsf{E}_L\langle\vec{x}\rangle=\mathsf{E}_{L(l\vec{x}_1)}\langle(n-1)\vec{x}_3\rangle=(L(\vec{\omega})\oplus L^*(\vec{\omega}))((l+1)\vec{x}_1)$, $\mathsf{E}_{L(\vec{x}_3)}\langle \vec{x}\rangle=(L\oplus L^*)((l+1)\vec{x}_1)$ and $\mathsf{E}_L\langle\vec{x}+\vec{x}_3\rangle=\mathsf{E}_{L(\vec{x}_3)}\langle \vec{x}-\vec{x}_3\rangle=\mathsf{E}_L((l+1)\vec{x}_1)$ for all $l \in \mathbb{Z}$. Hence, the required exact sequence can be obtained by applying the degree shift $(l+1)\vec{x}_1$ to the exact sequence \eqref{important exact}.
\end{proof}

\begin{prop}\label{bicartesian square}
   Assume $\vec{x}=a\vec{x}_3, \vec{y}=b\vec{x}_3\in \mathbb{L}$ with $a\in \mathbb{N}, b\in \mathbb{Z}_+$. Then there exists a bicartesian square in ${\rm vect}\mbox{-}\mathbb{X} $
     \[\begin{tikzcd}[ampersand replacement=\&,cramped,sep=small]
	{L(\vec{\omega})} \& {L(\vec{\omega}+\vec{y})} \\
	{\mathsf{E}_{L}\langle \vec{x}+\vec{y}\rangle} \& {\mathsf{E}_{L(\vec{y})}\langle \vec{x}\rangle.}
	\arrow[from=1-1, to=1-2]
	\arrow[from=1-2, to=2-2]
	\arrow[from=1-1, to=2-1]
	\arrow[from=2-1, to=2-2]
\end{tikzcd}\]
\end{prop}
\begin{proof}

    In the proof of Proposition \ref{exact sequence11}, we observe that for any $L$, there exists a bicartesian square:
\[\begin{tikzcd}[ampersand replacement=\&,cramped,sep=small]
{\mathsf{E}_L} \& {\mathsf{E}_L\langle\vec{x}_3\rangle} \\
L \& {L(\vec{x}_3).}
\arrow[from=1-1, to=1-2]
\arrow[from=1-2, to=2-2]
\arrow[from=1-1, to=2-1]
\arrow[from=2-1, to=2-2]
\end{tikzcd}\]
Replacing $L$ with $L^\vee(-\vec{\omega}-\vec{y})$, and then applying vector bundle duality $^\vee$ to this bicartesian square, we obtain the following bicartesian square:
\begin{equation}\label{square 111}
    \begin{tikzcd}[ampersand replacement=\&,cramped,sep=small]
{L(\vec{\omega})} \& {L(\vec{\omega})(\vec{x}_3)} \\
{\mathsf{E}_{L}\langle \vec{x}_3\rangle} \& {\mathsf{E}_{L(\vec{x}_3).}}
\arrow[from=1-1, to=1-2]
\arrow[from=1-2, to=2-2]
\arrow[from=1-1, to=2-1]
\arrow[from=2-1, to=2-2]
\end{tikzcd}
\end{equation}
By varying the choice of $L$ in square \eqref{square 111} and  Proposition \ref{PPS1}, we obtain the following commutative diagrams
  \[\begin{adjustbox}{scale=1}
  \begin{tikzcd}[ampersand replacement=\&,cramped,sep=small,cells={nodes={font=\tiny}}]
	{L(\vec{\omega})} \& {L(\vec{\omega})(\vec{x}_3)} \& {L(\vec{\omega})(2\vec{x}_3)} \& \cdots \& {L(\vec{\omega})(\vec{y}-\vec{x}_3)} \& {L(\vec{\omega})(\vec{y})} \\
	{\mathsf{E}_{L}\langle \vec{x}_3\rangle} \& {\mathsf{E}_{L(\vec{x}_3)}} \&\& \cdots \\
	{\mathsf{E}_{L}\langle 2\vec{x}_3\rangle} \& {\mathsf{E}_{L(\vec{x}_3)}\langle \vec{x}_3\rangle} \& {\mathsf{E}_{L(2\vec{x}_3)}} \& \cdots \\
	\vdots \& \vdots \& \vdots \& \cdots \& \vdots \\
	{\mathsf{E}_{L}\langle\vec{y}-\vec{x}_3\rangle} \& {\mathsf{E}_{L(\vec{x}_3)}\langle\vec{y}-3\vec{x}_3\rangle} \& {\mathsf{E}_{L(2\vec{x}_3)}\langle\vec{y}-4\vec{x}_3\rangle} \& \cdots \& {\mathsf{E}_{L(\vec{y}-\vec{x}_3)}} \\
	{\mathsf{E}_{L}\langle\vec{y}\rangle} \& {\mathsf{E}_{L(\vec{x}_3)}\langle\vec{y}-\vec{x}_3\rangle} \& {\mathsf{E}_{L(2\vec{x}_3)}\langle\vec{y}-2\vec{x}_3\rangle} \& \cdots \& {\mathsf{E}_{L(\vec{y}-\vec{x}_3)}\langle \vec{x}_3\rangle} \& {\mathsf{E}_{L(\vec{y})}} \\
	{\mathsf{E}_{L}\langle\vec{y}+\vec{x}_3\rangle} \& {\mathsf{E}_{L(\vec{x}_3)}\langle\vec{y}\rangle} \& {\mathsf{E}_{L(2\vec{x}_3)}\langle\vec{y}-\vec{x}_3\rangle} \& \cdots \& {\mathsf{E}_{L(\vec{y}-\vec{x}_3)}\langle 2\vec{x}_3\rangle} \& {\mathsf{E}_{L(\vec{y})}\langle \vec{x}_3\rangle} \\
	\vdots \& \vdots \& \vdots \& \cdots \& \vdots \& \vdots \\
	{\mathsf{E}_{L}\langle\vec{y}+\vec{x}\rangle} \& {\mathsf{E}_{L(\vec{x}_3)}\langle\vec{y}+\vec{x}-\vec{x}_3\rangle} \& {\mathsf{E}_{L(2\vec{x}_3)}\langle\vec{y}+\vec{x}-2\vec{x}_3\rangle} \& \cdots \& {\mathsf{E}_{L(\vec{y}-\vec{x}_3)}\langle\vec{x}+\vec{x}_3\rangle} \& {\mathsf{E}_{L(\vec{y})}\langle \vec{x}\rangle}
	\arrow[from=1-1, to=1-2]
	\arrow[from=1-1, to=2-1]
	\arrow[from=1-2, to=1-3]
	\arrow[from=1-2, to=2-2]
	\arrow[from=1-3, to=1-4]
	\arrow[from=1-3, to=3-3]
	\arrow[from=1-4, to=1-5]
	\arrow[from=1-5, to=1-6]
	\arrow[from=1-5, to=4-5]
	\arrow[from=1-6, to=6-6]
	\arrow[from=2-1, to=2-2]
	\arrow[from=2-1, to=3-1]
	\arrow[from=2-2, to=3-2]
	\arrow[from=3-1, to=3-2]
	\arrow[from=3-1, to=4-1]
	\arrow[from=3-2, to=3-3]
	\arrow[from=3-2, to=4-2]
	\arrow[from=3-3, to=4-3]
	\arrow[from=4-1, to=5-1]
	\arrow[from=4-2, to=5-2]
	\arrow[from=4-3, to=5-3]
	\arrow[from=4-5, to=5-5]
	\arrow[from=5-1, to=5-2]
	\arrow[from=5-1, to=6-1]
	\arrow[from=5-2, to=5-3]
	\arrow[from=5-2, to=6-2]
	\arrow[from=5-3, to=5-4]
	\arrow[from=5-3, to=6-3]
	\arrow[from=5-4, to=5-5]
	\arrow[from=5-5, to=6-5]
	\arrow[from=6-1, to=6-2]
	\arrow[from=6-1, to=7-1]
	\arrow[from=6-2, to=6-3]
	\arrow[from=6-2, to=7-2]
	\arrow[from=6-3, to=6-4]
	\arrow[from=6-3, to=7-3]
	\arrow[from=6-4, to=6-5]
	\arrow[from=6-5, to=6-6]
	\arrow[from=6-5, to=7-5]
	\arrow[from=6-6, to=7-6]
	\arrow[from=7-1, to=7-2]
	\arrow[from=7-1, to=8-1]
	\arrow[from=7-2, to=7-3]
	\arrow[from=7-2, to=8-2]
	\arrow[from=7-3, to=7-4]
	\arrow[from=7-3, to=8-3]
	\arrow[from=7-4, to=7-5]
	\arrow[from=7-5, to=7-6]
	\arrow[from=7-5, to=8-5]
	\arrow[from=7-6, to=8-6]
	\arrow[from=8-1, to=9-1]
	\arrow[from=8-2, to=9-2]
	\arrow[from=8-3, to=9-3]
	\arrow[from=8-5, to=9-5]
	\arrow[from=8-6, to=9-6]
	\arrow[from=9-1, to=9-2]
	\arrow[from=9-2, to=9-3]
	\arrow[from=9-3, to=9-4]
	\arrow[from=9-4, to=9-5]
	\arrow[from=9-5, to=9-6]
\end{tikzcd}
\end{adjustbox}\]
    where the smallest squares are bicartesian. 
     By the pullback lemma, the largest square is the bicartesian square, as required.
\end{proof}

 \noindent {\bf Acknowledgements.}\quad
 Jianmin Chen and Jinfeng Zhang were partially supported by the National Natural Science Foundation of China (No.s 12371040, 12131018 and 12161141001). Shiquan Ruan was partially supported by the Natural Science Foundation of Xiamen (No. 3502Z20227184), and the Natural Science Foundation of Fujian Province (No. 2022J01034).

	\vskip 5pt
	\noindent {\tiny  \noindent Jianmin Chen, Shiquan Ruan and Jinfeng Zhang\\
		School of Mathematical Sciences, \\
		Xiamen University, Xiamen, 361005, Fujian, PR China.\\
		E-mails: chenjianmin@xmu.edu.cn, sqruan@xmu.edu.cn,
		zhangjinfeng@stu.xmu.edu.cn\\ }
	\vskip 3pt

\end{document}